\documentclass[11pt]{amsart}
\usepackage{a4wide}
\usepackage{t1enc}
\usepackage[latin1]{inputenc}
\usepackage{graphicx}
\usepackage{amsthm}
\usepackage{amsmath}
\usepackage{amssymb}
\usepackage{mathtools}
\usepackage{sseq}
\usepackage[pagebackref]{hyperref}
\usepackage{hyperref}
\usepackage[curve]{xypic}
\usepackage[left=3.2cm,right=3.2cm,top=3cm,bottom=3cm]{geometry} 
\usepackage{pdfpages}
\usepackage{stmaryrd}
\usepackage{marvosym}
\usepackage{mathrsfs}
\usepackage{tikz}
\usetikzlibrary{cd}
\usepackage{wrapfig}
\usepackage[capitalise]{cleveref}
\usepackage{todonotes}
\usepackage[mathscr]{eucal}

\newtheorem{thm}{Theorem}[section]

\newtheorem*{thm*}{Theorem}
\newtheorem*{mthm*}{Main Theorem}

\newtheorem{theorem}[thm]{Theorem}

\newtheorem{lemma}[thm]{Lemma}

\newtheorem{cor}[thm]{Corollary}

\newtheorem{prop}[thm]{Proposition}
\newtheorem{proposition}[thm]{Proposition}

\newtheorem*{conjecture*}{Conjecture}

\newtheorem*{question*}{Question}

\theoremstyle{definition}
\newtheorem{defi}[thm]{Definition}
\newtheorem{construction}[thm]{Construction}

\newtheorem{definition}[thm]{Definition}
\newtheorem{notation}[thm]{Notation}

 \newtheorem{example}[thm]{Example}
  \newtheorem*{example*}{Example}
 
 \newtheorem{remark}[thm]{Remark}

\newtheorem{warning}[thm]{Warning}

\theoremstyle{remark}

\newcommand{\A}{\mathbb{A}}
\newcommand{\B}{\mathbf{B}}
\newcommand{\C}{\mathbb{C}}
\newcommand{\E}{\mathbb{E}}

\newcommand{\G}{\mathbb{G}}
\newcommand{\Gm}{\mathbb{G}_m}

\newcommand{\N}{\mathbb{N}}

\newcommand{\R}{\mathbb{R}}
\let\oldS\S
\newcommand{\parag}{\oldS}
\renewcommand{\S}{\mathbb{S}}
\newcommand{\T}{\mathbf{T}}

\newcommand{\Z}{\mathbb{Z}}

\newcommand{\CP}{\mathbb{CP}}

\newcommand{\con}{\tau_{\geq 0}}

\newcommand{\BB}{\mathbf{B}}
\newcommand{\CC}{\mathcal{C}}
\newcommand{\DD}{\mathcal{D}}
\newcommand{\EE}{\mathcal{E}}
\newcommand{\FF}{\mathcal{F}}

\newcommand{\UU}{\mathcal{U}}

\newcommand{\LL}{\mathcal{L}}
\newcommand{\PP}{\mathcal{P}}
\newcommand{\RR}{\mathcal{R}}
\renewcommand{\SS}{\mathcal{S}}
\newcommand{\XX}{\mathcal{X}}

\newcommand{\YY}{\mathcal{Y}}

\newcommand{\Spc}{\mathcal{S}}
\newcommand{\OO}{\mathcal{O}}
\newcommand{\QQ}{\mathcal{Q}}
\newcommand{\MM}{\mathcal{M}}
\newcommand{\NN}{\mathcal{N}}
\newcommand{\TT}{\mathcal{T}}
\newcommand{\WW}{\mathcal{W}}

\renewcommand{\i}{\infty}

\newcommand{\Ef}{\mathsf{E}}
\newcommand{\Mf}{\mathsf{M}}
\newcommand{\Nf}{\mathsf{N}}
\newcommand{\Xf}{\mathsf{X}}
\newcommand{\Yf}{\mathsf{Y}}
\newcommand{\Zf}{\mathsf{Z}}

\newcommand{\Dual}{\widehat}

\newcommand{\et}{\mathrm{\acute{e}t}}

\newcommand{\tensor}{\otimes}
\newcommand{\wEllT}{\widetilde{\mathcal{E}ll}_{\T}}
\newcommand{\EllT}{\mathcal{E}ll_{\T}}
\newcommand{\GQ}{\widehat{\mathbf{G}}^{\mathscr{Q}}}
\DeclareMathOperator{\Ell}{Ell}

\DeclareMathOperator{\TopCat}{TopCat}
\DeclareMathOperator{\coh}{coh}
\DeclareMathOperator{\SeCat}{SeCat}
\DeclareMathOperator{\PreAb}{PreAb}
\DeclareMathOperator{\TopAb}{TopAb}
\DeclareMathOperator{\FGLie}{FGLie}
\DeclareMathOperator{\CptLie}{CptLie}
\DeclareMathOperator{\Gpd}{Gpd}

\DeclareMathOperator{\ad}{ad}
\DeclareMathOperator{\Stk}{Stk}

\DeclareMathOperator{\rep}{rep}
\DeclareMathOperator{\LocSys}{LocSys}

\DeclareMathOperator{\Rep}{Rep}
\DeclareMathOperator{\pt}{pt}

\DeclareMathOperator{\Tot}{Tot}

\DeclareMathOperator{\Sh}{Sh}
\DeclareMathOperator{\Vect}{Vect}

\DeclareMathOperator{\Aut}{Aut}
\DeclareMathOperator{\op}{op}

\DeclareMathOperator{\Sp}{Sp}
\DeclareMathOperator{\Cat}{Cat}
\DeclareMathOperator{\RelCat}{RelCat}
\DeclareMathOperator{\Sub}{Sub}

\DeclareMathOperator{\fin}{fin}

\DeclareMathOperator{\Shv}{Shv}
\DeclareMathOperator{\Spf}{Spf}

\DeclareMathOperator{\Ab}{Ab}
\DeclareMathOperator{\ab}{ab}

\DeclareMathOperator{\Pre}{Pre} 
\DeclareMathOperator{\TopGpd}{TopGpd} 
\DeclareMathOperator{\PreStk}{PreStk} 
\DeclareMathOperator{\ev}{ev}

\DeclareMathOperator{\comodules}{\text{-}comod}
\DeclareMathOperator{\Aff}{Aff}
\DeclareMathOperator{\id}{id}

\DeclareMathOperator{\im}{im}

\DeclareMathOperator{\cof}{cof}
\DeclareMathOperator{\dual}{dual}
\DeclareMathOperator{\sm}{\wedge}
\DeclareMathOperator{\colim}{colim}

\DeclareMathOperator{\Lie}{Lie}

\DeclareMathOperator{\Top}{Top}

\DeclareMathOperator{\Stab}{Stab}

\DeclareMathOperator{\Hom}{Hom}
\DeclareMathOperator{\Spec}{Spec}
\DeclareMathOperator{\cSpec}{cSpec}
\DeclareMathOperator{\fib}{fib}

\DeclareMathOperator{\hocolim}{hocolim}
\DeclareMathOperator{\QCoh}{QCoh}

\DeclareMathOperator{\cn}{cn}
\DeclareMathOperator{\afp}{afp}

\DeclareMathOperator{\Mod}{Mod}

\DeclareMathOperator{\res}{res}

\DeclareMathOperator{\Ind}{Ind}

\DeclareMathOperator{\Fun}{Fun}
\DeclareMathOperator{\sCat}{sCat}

\DeclareMathOperator{\Sing}{Sing}
\DeclareMathOperator{\sSet}{sSet}

\DeclareMathOperator{\Map}{Map}
\DeclareMathOperator{\uMap}{\underline{Map}}
\DeclareMathOperator{\CAlg}{CAlg}
\DeclareMathOperator{\cCAlg}{cCAlg}

\DeclareMathOperator{\Orb}{Orb}

\mathchardef\mhyphen="2D

\makeatletter
\newcommand*{\doublerightarrow}[2]{\mathrel{
  \settowidth{\@tempdima}{$\scriptstyle#1$}
  \settowidth{\@tempdimb}{$\scriptstyle#2$}
  \ifdim\@tempdimb>\@tempdima \@tempdima=\@tempdimb\fi
  \mathop{\vcenter{
    \offinterlineskip\ialign{\hbox to\dimexpr\@tempdima+1em{##}\cr
    \rightarrowfill\cr\noalign{\kern.5ex}
    \rightarrowfill\cr}}}\limits^{\!#1}_{\!#2}}}

\def\fibdown{\ar@{->>}[d]}
\def\hookdown{\ar@<-.5ex>[d]|{\phantom{a}}|<<{\put(-.7,2){$\scriptstyle\cap$}}}

\def\rrarrow{  \hspace{.05cm}\mbox{\,\put(0,-2){$\rightarrow$}\put(0,2){$\rightarrow$}\hspace{.45cm}}}

\def\rrrarrow{ \hspace{.05cm}\mbox{\,\put(0,-3){$\rightarrow$}\put(0,1){$\rightarrow$}\put(0,5){$\rightarrow$}\hspace{.45cm}}}

\def\rrrrarrow{\hspace{.05cm}\mbox{\,\put(0,-3.5){$\rightarrow$}\put(0,0){$\rightarrow$}\put(0,3.5){$\rightarrow$}\put(0,7){$\rightarrow$}
               \hspace{.45cm}}}

\makeatletter
\DeclareRobustCommand\widecheck[1]{{\mathpalette\@widecheck{#1}}}
\def\@widecheck#1#2{%
    \setbox\z@\hbox{\m@th$#1#2$}%
    \setbox\tw@\hbox{\m@th$#1%
       \widehat{%
          \vrule\@width\z@\@height\ht\z@
          \vrule\@height\z@\@width\wd\z@}$}%
    \dp\tw@-\ht\z@
    \@tempdima\ht\z@ \advance\@tempdima2\ht\tw@ \divide\@tempdima\thr@@
    \setbox\tw@\hbox{%
       \raise\@tempdima\hbox{\scalebox{1}[-1]{\lower\@tempdima\box
\tw@}}}%
    {\ooalign{\box\tw@ \cr \box\z@}}}
\makeatother


\begin{document}
\title{On equivariant topological modular forms}
\author{David Gepner}
\address{Johns Hopkins University,
Zanvyl Krieger School of Arts \& Sciences,
Department of Mathematics,
3400 N. Charles Street, Baltimore, MD 21218, United States}
\email{gepner@jhu.edu}

\author{Lennart Meier}
\address{Universiteit Utrecht, Mathematical Institute, Budapestlaan 6, 3584 CD Utrecht, The Netherlands}
\email{f.l.m.meier@uu.nl}

\begin{abstract}
    Following ideas of Lurie, we give in this article a general construction of equivariant elliptic cohomology without restriction to characteristic zero. Specializing to the universal elliptic curve we obtain in particular equivariant spectra of topological modular forms. We compute the fixed points of these spectra for the circle group and more generally for tori. 
\end{abstract}

\maketitle
\setcounter{tocdepth}{1}
\tableofcontents

\section{Introduction}
The aim of this paper is to construct an integral theory of equivariant elliptic cohomology for an arbitrary compact Lie group and prove some of its basic properties. In particular, this applies to elliptic cohomology based on the universal elliptic curve over the moduli stack of elliptic curves, yielding compatible $G$-equivariant spectra of topological modular forms $\mathrm{TMF}$ for every compact abelian Lie group $G$. The construction follows the ideas sketched in \cite{Lur07} and builds crucially on the theory Lurie detailed in subsequent work.
One of our main results beyond the construction is the computation of the fixed points of $\mathrm{TMF}$ with respect to the circle-group $\T$, identifying these with $\mathrm{TMF} \oplus \Sigma \mathrm{TMF}$. 

\subsection{Motivation and main results}
The $G$-equivariant complex K-theory of a point agrees per definition with the representation ring of $G$. Together with Bott periodicity this gives for example $KU^{\mathrm{even}}_{\T}(\pt) \cong \Z[t^{\pm 1}]$ and $KU^{\mathrm{odd}}_{\T}(\pt) = 0$. The goal of the present paper is to repeat this calculation at the level of equivariant \emph{elliptic} cohomology. 

The basic idea of equivariant elliptic cohomology is inspired by the following algebro-geometric interpretation of the computation above: The spectrum $\Spec \Z[t^{\pm 1}]$ coincides with the multiplicative group $\mathbb{G}_m$. Thus, for every $\T$-space $X$ the K-theory $KU^{\mathrm{even}}_{\T}(X)$ defines a quasi-coherent sheaf on $\mathbb{G}_m$. In elliptic cohomology, one replaces $\G_m$ by an elliptic curve and thus $\T$-equivariant elliptic cohomology takes values in quasi-coherent sheaves on an elliptic curve. 

There have been several realizations of this basic idea. The first was given by Grojnowski \cite{Groj} in 1995, taking a point $\tau$ in the upper half-plane and the resulting complex elliptic curve $\C/\Z + \tau\Z$ as input and producing a sheaf-valued $G$-equivariant cohomology theory for every compact connected Lie group $G$. This was motivated by applications in geometric representation theory (see \cite{GKV}, and e.g.\ \cite{AganagicOkounkov} \cite{YangZhao} \cite{FRV} for later developments). Another motivation was provided by Miller's suggestion to use $\T$-equivariant elliptic cohomology to reprove the rigidity of the elliptic genus, which was realized by Rosu \cite{Rosu} and Ando--Basterra \cite{AndoBasterra}. We do not want to summarize all work done on equivariant elliptic cohomology, but want to mention Greenlees's work on rational $\T$-equivariant cohomology \cite{GreenleesElliptic}, work of Kitchloo \cite{KitchlooElliptic}, Rezk \cite{RezkElliptic}, Spong \cite{Spong} and Berwick-Evans--Tripathy \cite{BerwickTripathy} in the complex-analytic setting and the work of Devoto \cite{DevotoEquiv} on equivariant elliptic cohomology for finite groups and its applications to moonshine phenomena in \cite{GanterHecke}. 

Many more works could be named, but for us the most relevant work has been done by Lurie in  \cite{Lur07} \cite{LurEllIII}. The starting point is the definition of an elliptic cohomology theory, consisting of an even-periodic ring spectrum $R$, an elliptic curve $\Ef$ over $\pi_0R$, and an isomorphism between the formal group associated with $R$ and that associated with $\Ef$. Lurie refined this to the notion of an \emph{oriented spectral elliptic curve}, consisting of an even-periodic $E_{\infty}$-ring spectrum $R$, an elliptic curve $\Ef$ over $R$ and an equivalence between the formal groups $\Spf R^{\CP^{\infty}}$ and $\widehat{\Ef}$ over $R$. Thus we have moved completely into the land of spectral algebraic geometry, instead of the hybrid definition of an elliptic cohomology theory. This seems to be of key importance in order to obtain integral equivariant elliptic cohomology theories, without restriction to characteristic zero. Following the outline given in \cite{Lur07} and extending work of \cite{LurEllIII} from finite groups to compact Lie groups, we associate to every oriented spectral elliptic curve $\Ef$ a ``globally'' equivariant elliptic cohomology theory.

To make this precise, we work with orbispaces, an $\infty$-category incorporating $G$-equivariant homotopy theory for all compact Lie groups $G$ at once, which was introduced in \cite{GH}. As any topological groupoid or stack (such as the stack $[X/G]$ associated to a $G$-space $X$) determines an orbispace, they are a convenient source category for cohomology theories which are equivariant for all $G$ simultaneously, often called \emph{global cohomology theories}.

We construct from the datum of an oriented elliptic curve $\Ef$ over an $E_{\infty}$-ring $R$, for any orbispace $Y$, a stack $\Ell(Y)$ in the world of spectral algebraic geometry. 
For example, in the case of $\B\T = [\pt/\T]$, the stack $\Ell(\B\T)$ is precisely the spectral elliptic curve $\Ef$ and $\Ell(\B C_n)$ agrees with the $n$-torsion $\Ef[n]$. Moreover, for any compact abelian Lie group $G$ pushforward of the structure sheaf along $\Ell([X/G]) \to \Ell(\B G)$ defines a contravariant functor from (finite) $G$-spaces to quasi-coherent sheaves on $\Ell(\B G)$. In particular, we obtain a functor
\[\EllT\colon \left(\text{finite } \T\text{-spaces}\right)^{\op} \to \QCoh(\Ef). \] 
Thus, we obtain a derived realization of the basic idea of $\T$-equivariant elliptic cohomology sketched above. Taking homotopy groups defines functors $\EllT^{i} = \pi_{-i}\EllT$ to quasi-coherent sheaves on the underlying classical elliptic curve. 
Finally, to push the $\T$-equivariant elliptic cohomology functor
into the realm of ordinary equivariant homotopy theory, we postcompose $\EllT$ with the global sections functor $\Gamma$, to obtain a functor
\[\Gamma\EllT\colon \left(\text{finite } \T\text{-spaces}\right)^{\op} \to \QCoh(\Ef) \overset{\Gamma}\rightarrow\mathrm{Spectra}.\]
We show in \cref{constr:GSpectra} that this functor is representable by a genuine $\T$-spectrum, which we also denote $R$. In particular, $\Gamma\EllT(\pt)$ coincides with the $\T$-fixed points $R^{\B\T}$.\footnote{We use the notation $R^{\B\T}$ for what is more commonly denoted $R^{\T}$, namely the $\T$-fixed point spectrum, as our notation both stresses the importance of the stack $\B\T$ and avoids confusion with the function spectrum $\uMap(\Sigma^{\infty}_+\T, R)$, which we will also have opportunity to use.} This allows us to formulate the main computation of this article, \cref{thm:main}. 
\begin{thm}
Restriction and degree-shifting transfer determine an equivalence of spectra
\[
R\oplus \Sigma R \to R^{\B\T}.
\]
\end{thm}
Instead of an oriented spectral elliptic curve over an affine spectral scheme, we might also directly work with the universal oriented elliptic curve. The associated equivariant theory has as underlying spectrum $\mathrm{TMF}$, the Goerss--Hopkins--Miller spectrum of topological modular forms. Thus we may speak of a genuine $\T$-spectrum $\mathrm{TMF}$ (and likewise of a genuine $G$-spectrum $\mathrm{TMF}$ for all compact abelian Lie groups $G$) and easily deduce from our main theorem the following corollaries: 
\begin{cor}
Restriction and degree-shifting transfers determine an equivalence of spectra
\[
\mathrm{TMF}\oplus \Sigma \mathrm{TMF} \to \mathrm{TMF}^{\B\T}.
\]
This coincides with the ring spectrum of global sections of the structure sheaf of the universal oriented spectral elliptic curve. 
In particular, the reduced theory
$$
\widetilde{\mathrm{TMF}}(\B\T)\simeq\fib(\mathrm{TMF}^{\B\T} \to \mathrm{TMF})\simeq\Sigma\mathrm{TMF}
$$
is an invertible $\mathrm{TMF}$-module.
\end{cor}

\begin{cor}
Restriction and degree-shifting transfers determine an equivalence of spectra 
\[
\bigoplus_{S \subset \{1,\dots, n\}} \Sigma^{|S|} \mathrm{TMF} \to \mathrm{TMF}^{\BB\T^n}.
\]
\end{cor}
Note that this in particular amounts to a computation of the graded ring $\mathrm{TMF}_{\T^n}^*(\pt)$.
The previous corollary also implies a dualizability property for equivariant $\mathrm{TMF}$, which is in stark contrast with the situation for K-theory, where $\pi_0KU^G = R(G)$ has infinite rank over $\pi_0KU = \Z$ if $G$ is a compact Lie group of positive dimension. More precisely, we establish the following result as \cref{cor:dualizable}:
\begin{cor}
For every compact abelian Lie group $G$, the $G$-fixed points $\mathrm{TMF}^{\B G}$ are a dualizable $\mathrm{TMF}$-module. 
\end{cor}
We conjecture that the corresponding statement is true for all compact Lie groups. For finite groups, this is a consequence of Lurie's Tempered Ambidexterity Theorem \cite{LurEllIII} in his setting of equivariant elliptic cohomology. 

\subsection{The construction of $\Ell$}
Equivariant K-theory defines a contravariant functor, associating to each finite $G$-CW complex $X$ an $E_{\infty}$-ring spectrum $KU_G(X)$, based on the theory $G$-vector bundles. Moving to spectral algebraic geometry, one can consider the covariant functor, sending $X$ to $\Spec KU_G(X)$. To make the functoriality with respect to $G$ more transparent, one can also use the $\infty$-category $\Spc_{\Orb}$ of orbispaces as a source of this functor. The $\infty$-category $\Spc_{\Orb}$ is defined as space-valued presheaves on the global orbit $\infty$-category $\Orb$ with objects $\B G$ for every compact Lie group $G$. We will give a precise definition in \cref{sec:orbispaces}.

Constructing $\Ell$, there are several differences: 
\begin{enumerate}
    \item our construction is based on the choice of an oriented spectral elliptic curve $\Ef$ over some non-connective spectral Deligne--Mumford stack $S$; 
    \item since we associate to the $\T$-space $\pt$ the elliptic curve $\Ef$ itself (which is not affine), we can only work with the covariant picture;
    \item most importantly, we do not know of a geometric construction of $\Ell$. 
\end{enumerate}

As the target of $\Ell$, we can choose any cocomplete  $\infty$-category $\XX$ with a functor from $S$-schemes. In this paper, we will choose $\XX$ to be an $\infty$-topos of sheaves. 

To remedy the lack of a direct geometric construction of $\Ell$, we will construct it in steps: 
\begin{enumerate}
    \item We first define $\Ell$ on $\Orb^{\ab} \subset \Orb$, the full subcategory on $\B G$ for $G$ abelian. This is essentially done by the formula $\Ell(\B G) = \Hom(\widehat{G}, \Ef)$, where $\widehat{G}$ is the Pontryagin dual of $G$. To obtain the correct functoriality, we exploit the preorientation of $\Ef$ using categorified Pontryagin duality, as explained in \cref{sec:Abelian} and especially \cref{con:PicardDuality}. 
    \item We left Kan extend the functor $\Orb^{\ab} \to \XX$ first to $\Orb$ and then to $\Spc_{\Orb} = \PP(\Orb)$. 
\end{enumerate}
While Kan extending from $\Orb$ to $\Spc_{\Orb}$ is dictated by our wish that $\Ell$ preserves colimits, just Kan extending from $\Orb^{\ab}$ to $\Orb$ may be more suprising: we let abelian information dictate non-abelian information. The reasons are essentially twofold:
\begin{enumerate}
    \item If we do the same in the case of K-theory, we exactly recover equivariant K-theory. In some incarnation, this result goes at least back to \cite{AdamsHaeberlyJackowskiMay} and we will discuss it in more detail in \cref{sec:KTheory}. 
    \item The resulting theory for non-abelian groups of equivariance has many good properties. This was already outlined in \cite{Lur07} and will also be established in a sequel to this paper \cite{GepnerMeierII}. 
\end{enumerate}

\subsection{Outline of the paper}
We start with a section about orbispaces. We construct the $\infty$-category of orbispaces $\Spc_{\Orb} = \PP(\Orb)$ as presheaves on the global orbit category $\Orb$, whose objects are the classifying stacks $\B G$ for compact Lie groups $G$. Our treatment of this is $\infty$-categorical and mostly self-contained. 

\cref{sec:Abelian} is devoted to a general framework for constructing global cohomology theories using a preoriented abelian group object in a nice $\i$-category $\XX$.
This is strongly inspired by the sketch provided in \cite{Lur07}. More precisely, we associate to a preoriented abelian group object a functor $\Spc_{\Orb} \to \XX$. 

In \cref{sec:KTheory}, we establish that the construction from \cref{sec:Abelian} recovers equivariant K-theory when applied to the strict multiplicative group. 

In \cref{sec:Orientations}, we introduce our main example of a nice $\infty$-category $\XX$, namely the $\infty$-category $\Shv(\Mf)$ of sheaves on a given non-connective spectral Deligne--Mumford stack $\Mf$. Moreover, we give a precise definition of an orientation of an elliptic curve over such an $\Mf$. While not necessary for the definition of equivariant elliptic cohomology, it is crucial for its finer properties. 

The actual construction of equivariant elliptic cohomology is given in \cref{sec:EquivariantElliptic}. From a preoriented elliptic curve, we both construct a functor $\Ell\colon \Spc_{\Orb}\to \Shv(\Mf)$ and variants taking values in quasi-coherent sheaves. 

In \cref{sec:EllCohSymMon} we show that the resulting functor
\[\EllT\colon \left(\text{finite }\T\text{-spaces}\right)^{\op} \to \QCoh(\Ef) \]
is symmetric monoidal. This is based on the orientation $\widehat{\Ef} \simeq \Spf R^{B\T}$ (if $\Mf = \Spec R$). A crucial ingredient is the study of fiber products of the form $\Spf R^{BC_n} \times_{\Spf R^{B\T}} \Spf R^{BC_m}$. This is done in the equivalent setting of coalgebras and their cospectra in \cref{app:EM}, employing essentially the convergence of the Eilenberg--Moore spectral sequence. 

Using the symmetric monoidality and the universal property of $G$-spectra as inverting representation spheres (see \cref{app:GSpectra}), we show in \cref{sec:ellipticspectra} that $\EllT$ factors over finite $\T$-spectra. This allows us to use the Wirthm{\"u}ller isomorphism calculating the dual of $\Sigma^{\infty}\T_+$ to compute the dual of $\EllT(\T)$, which is the key ingredient in the proof of our main theorem in \cref{sec:computation}. Note that while by construction $\Gamma\EllT(\pt)$ coincides with global sections of $\OO_{\Ef}$ and this is an object purely in spectral algebraic geometry, we are able to use $\T$-equivariant homotopy theory for its computation.

We end with four appendices. \cref{app:q} discusses quotient $\i$-categories, which are used in our treatment of orbispaces. \cref{app:GlobalOrbitModels} compares our treatment of the global orbit category $\Orb$ with the original treatment in \cite{GH}.   \cref{app:GSpectra} gives an $\i$-categorical construction of (genuine) $G$-spectra for arbitrary compact Lie groups $G$ and compares it to orthogonal $G$-spectra, and uses work of Robalo \cite{RobaloBridge} to establish a universal property of the $\i$-category of $G$-spectra. \cref{app:SAG} establishes some statements in spectral algebraic geometry that are relevant for working with the big \'etale site. 

\subsection*{Conventions}\label{Conventions}
In general, we will freely use the terminology of $\infty$-categories and spectral algebraic geometry, for which we refer to the book series \cite{HTT}, \cite{HA} and \cite{SAG}. In particular, an $\infty$-category will be for us a quasicategory. One difference though is that we will assume all (non-connective) spectral Deligne--Mumford stacks to be locally noetherian, i.e.\ they are \'etale locally of the form $\Spec A$ with $\pi_0A$ noetherian and $\pi_iA$ a finitely generated $\pi_0A$-module for $i\geq 0$. Moreover, we assume (non-connective) spectral Deligne--Mumford stacks to be quasi-separated, i.e.\ the fiber product of any two affines over such a stack is quasi-compact again, and that all iterated diagonals are quasi-separated (i.e.\ $n$-quasi-separated in the sense of \cref{def:nqsep} for all $n\geq 1$). As discussed in \cref{app:SAG}, these are mild conditions, which are usually satisfied in practice and form the natural higher analogue of the quasi-separated-condition in classical algebraic geometry.

If we write $\Mf = (\MM, \OO_{\Mf})$, then $\MM$ denotes the underlying $\infty$-topos of the (non-connective) spectral Deligne--Mumford stack $\Mf$ and $\OO_{\Mf}$ its structure sheaf. In contrast, $\Shv(\Mf)$ will denote the sheaves on the big \'etale site (see \cref{def:Shv}). 

We use $\Spc$ as notation for the $\infty$-category of spaces and $\Sp$ for the $\infty$-category of spectra. Likewise, we use $\Spc^G$ for the $\infty$-category of $G$-spaces and $\Sp^G$ for the $\infty$-category of genuine $G$-spectra (see \cref{app:GSpectra} for details). We use the notation $E^{\B H}$ for the $H$-fixed points of a $G$-spectrum $E$ for $H\subset G$. We will use $\PP$ for $\Spc$-valued presheaves. When speaking about topological spaces, we assume them to be compactly generated and weak Hausdorff.

We will generally follow the convention that $\Map$ denotes mapping spaces (or mapping groupoids), while $\uMap$ denotes internal mapping objects; depending on the context, these might be mapping spectra or mapping (topological) groupoids. 
This choice of notation emphasizes the fact that we always work $\infty$-categorically, and that all limits and colimits are formed in the appropriate $\infty$-category. Moreover, we use $|X_{\bullet}|$ as a shorthand for $\colim_{\Delta^{\op}}X_{\bullet}$. Regarding other special limits, we often write $\pt$ for the terminal object of an $\infty$-category. Moreover, we recall that every cocomplete $\i$-category is tensored over $\Spc$ and we use the symbol $\tensor$ to refer to this tensoring.

\subsection*{Acknowledgments}
Despite lacking firm foundations until more 
recently, equivariant elliptic cohomology is by now an old and diverse subject, going back 
to the 1980s and admitting applications across much of modern mathematics and physics.
Many people have contributed to the subject and their contributions are too numerous to name. But
we'd like to thank Matthew Ando, John Greenlees, and Jacob Lurie in particular for shaping our thinking about equivariant elliptic cohomology. We furthermore thank Thomas Nikolaus for his input at the beginning of this project, Viktoriya Ozornova and Daniel Sch{\"a}ppi for useful conversations about categorical questions and Bastiaan Cnossen and Stefan Schwede for providing references. We thank Bastiaan Cnossen and Sil Linskens for catching a mistake about orbispaces. Lastly, we thank the referee for their careful reading and insightful comments. 

The authors would like to thank the Isaac Newton Institute for Mathematical Sciences for support and hospitality during the programme ``Homotopy harnessing higher structures'' when work on this paper was undertaken. This work was supported by EPSRC grant number EP/R014604/1.
The authors would like to thank the Mathematical Sciences Research Institute for providing an inspiring working environment during the program ``Higher categories and categorification''. Lastly, we want to thank the Hausdorff Research Institute for Mathematics in Bonn for their excellent working conditions.

\section{Orbispaces}\label{sec:orbispaces}
In this section, we will introduce the particular framework of global unstable homotopy theory we will work in: orbispaces. These were first introduced in \cite{GH}, but we choose to give a (mostly) self-contained and $\infty$-categorical treatment. While philosophically our approach is similar to \cite{GH}, there are certain technical differences and we refer to \cref{app:GlobalOrbitModels} for a precise comparison. A different approach is taken in \cite{SchGlobal}, where unstable global homotopy theory is based on orthogonal spaces, which is in the same spirit as using orthogonal spectra to model stable global homotopy theory. The orbispace approach of \cite{GH} and the orthogonal space approach have been shown to be equivalent in \cite{SchwedeOrbi} and \cite{KorOrbi}. Another valuable source on orbispaces are \cite{RezkGlobal} and \cite{LinskensNardinPol}, though beware that they use the term `global spaces' for our orbispaces. 

We will construct the $\i$-category $\Spc_{\Orb}$ of orbispaces in a three step process. We will first define an $\infty$-category $\Stk_{\infty}$ of topological stacks, then define $\Orb$ as the full subcategory on the orbits $[\pt/G]$ for $G$ compact Lie and lastly define $\Spc_{\Orb}$ as presheaves on $\Orb$. This requires first recalling basic concepts about topological stacks, which we will do next. 

We write $\mathrm{Top}$ for the cartesian closed category of compactly generated weak Hausdorff topological spaces.
We will also be interested in the cartesian closed $2$-category $\TopGpd$ of topological groupoids, and we will typically write $X_\bullet$, $Y_\bullet$,... for topological groupoids, viewed as simplicial topological spaces.\footnote{The topological groupoid $\underline{\Map}_{\TopGpd}(Y_\bullet,X_\bullet)$ of maps has as its objects space the space of enriched functors and as morphism space the space of enriched natural transformations. Their definition is analogous to that for classical groupoids, e.g.\ the space of enriched functors is the evident subspace of the product of $\underline{\Map}_{\Top}(Y_0,X_0)$ and $\underline{\Map}_{\Top}(Y_1,X_1)$. Here the mapping spaces are equipped with the compact-open topology. This is a special case of the cartesian closedness of internal categories \cite{nlab:internal_category}.}
Since we will also consider topological stacks, we will additionally need to equip the category $\Top$ with a Grothendieck topology. While it is customary to define the covering sieves using open sets, one obtains the same topology using \'etale covers instead, as in \cite{GH}.

We write $\mathrm{Gpd}$ for the $2$-category of groupoids. This is an example of a $(2,1)$-category, i.e.\ a $2$-category whose $2$-morphisms are invertible. We will view $(2,1)$-categories implicitly as $\infty$-categories via their Duskin nerve and refer to \cite[Appendix A]{GHN} and \cite[Section 2.3 009P]{kerodon} for details. More generally, a weak $(2,1)$-category can be viewed as an $\infty$-category via its Duskin nerve, and an $\i$-category arises in this way if and only if it is a $2$-category in the sense of \cite[Section 2.3.4]{HTT}; that is, the inner horn liftings are unique in dimensions greater than $2$ \cite[Theorem 8.6]{Duskin}.

For us, a {\em topological stack} will mean a sheaf of groupoids on $\mathrm{Top}$, in the sense of higher category theory.\footnote{Often, \emph{topological stacks} only refers to those stacks that admit an atlas, i.e.\ those that can be represented by a topological groupoid. While these examples are all we care about, we find it unnecessary to restrict to these from the outset.}
In order to avoid universe issues we must index our stacks on a small full subcategory of $\Top'\subset\Top$. We suppose that it contains up to homeomorphism all countable CW-complexes and furthermore that it is closed under finite products and subspaces. 
The particular choice of indexing category is not very relevant for our purposes. We will assume in doubt that the object and morphism space of every topological groupoid is in $\Top'$.

\begin{definition}
A \emph{topological stack} is a functor $X\colon\Top'^{\op}\to\mathrm{Gpd}$ which satisfies the sheaf condition;
that is, for all $T\in \Top'$ and all coverings $p\colon U\to T$ in $\Top'$, the canonical map
\[
\xymatrix{
X(T)\ar[r] & \lim \{ X(U)\ar@<.5ex>[r]\ar@<-.5ex>[r] & X(U\times_T U)\ar@<1ex>[r]\ar[r]\ar@<-1ex>[r] & X(U\times_T U\times_T U)\}
}
\]
is an equivalence.
The $2$-category of stacks is the full subcategory
\[
\Stk\subset\Pre\Stk=\Fun(\Top'^{\op},\mathrm{Gpd})
\]
of the $2$-category of prestacks (that is, presheaves of groupoids on topological spaces).
\end{definition}

Our actual interest is in an $\i$-categorical quotient of the $2$-category of topological stacks, which we denote $\Stk_\i$ and refer to as the $\i$-category of ``stacks modulo homotopy''.
The idea is to view the $2$-category of topological stacks as enriched over itself via its cartesian closed (symmetric) monoidal structure, and then to use the realization functor $|-|\colon\Stk\to\SS$ to change enrichment (see \cite{GH} for a more precise formulation along this line).
What this essentially amounts to is imposing a homotopy relation upon the mapping spaces, which motivates our implementation of this idea. 
\begin{definition}\label{def:Stki}
The $\i$-category $\Stk_\i$ (respectively, $\Pre\Stk_\i$) is the quotient of the $2$-category $\Stk$ of topological stacks (respectively, the $2$-category $\Pre\Stk$ of topological prestacks) induced by the action of the standard cosimplicial simplex $\Delta_{\Top}^\bullet\colon\Delta\to\Top$.\footnote{Here and in the following we regard topological spaces as topological stacks via their representable functors.}
Heuristically, this $\i$-category has the same objects, with mapping spaces
\[
\Map(X,Y)\simeq |\Map_{\PreStk}(X\times \Delta_{\Top}^{\bullet},Y)|.
\]
Composition uses that geometric realizations commute with finite products and the diagonal map $\Delta^{\bullet}_{\Top}\to{\Delta^\bullet}_{\Top}\times \Delta^{\bullet}_{\Top}$.
See \cref{con:cosimplicialquotient} for details of this construction.
\end{definition}

\begin{remark}
The definition is motivated by the following observation: If we take the quotient of the $1$-category of CW-complexes by the action of the cosimplicial simplex $\Delta_{\Top}^\bullet\colon\Delta\to\Top$, we obtain the $\infty$-category $\Spc$ of spaces. Taking homwise $\Sing$ of the topological category of CW-complexes, we obtain a simplicial category $\CC$, which is equivalent to the simplicial category of Kan complexes. Moreover, the quotient category we are considering is by definition $\colim_{\Delta^{\op}} N\CC_{\bullet}$, where $\CC_n$ is the category obtained from taking the $n$-simplices in each mapping space and the colimit is taken in $\Cat_{\i}$. That this agrees up to equivalence with $N^{\coh}\CC \simeq \Spc$ can be shown analogously to \cref{cor:hocolim}. 
\end{remark}

\begin{remark}
The assignment $X \mapsto \Map_{\Stk_{\i}}(\pt, X)$ is one way to assign a homotopy type to a topological stack. For alternative treatments we refer to \cite{GH}, \cite{Noohi} and \cite{EbertHomotopy}; geometric applications have been given in \cite{Carchedi}. Given a presentation of a topological stack by a topological groupoid, all of these sources show the equivalence of the homotopy type of the stack with the homotopy type of the topological groupoid. As we will do so as well in \cref{prop:HomotopyType}, our approach is compatible with the cited sources. 
\end{remark}
An especially important class of stacks for us will be orbit stacks. 

\begin{definition}
A topological stack $X$ is an {\em orbit} if there exists a compact Lie group $G$ and an equivalence of topological stacks between $X$ and the stack quotient $[\pt/G]$. 
\end{definition}
\begin{remark}
Given a compact Lie group $G$, we write $\BB G$ for the stack of principal $G$-bundles.
This is an orbit stack, as a choice of basepoint induces an equivalence of  topological stacks $[\pt/G]\simeq\BB G$, and all choices of basepoint are equivalent.
We will typically write $\BB G$ for an arbitrary orbit stack; strictly speaking, however, one cannot recover the group $G$ from the stack $\BB G$ without a choice of basepoint, in which case $G\simeq\Omega\BB G$.
\end{remark}

The reason we are primarily concerned with orbit stacks is that they generate the $\i$-category of orbispaces \cite{GH}.\footnote{The theory in \cite{GH} allows for an arbitrary family of groups, but we will only consider the case of the family of
compact Lie groups. We allow morphisms between orbit stacks do be induced by arbitrary continuous homomorphisms, as these are automatically smooth \cite[Theorem 3.7.1]{Feres}.}
One should think of an orbispace as encoding the  various ``fixed point spaces'' of a topological stack, just as the equivariant homotopy type of a $G$-space is encoded by the fixed point subspaces, ranging over all (closed) subgroups $H$ of $G$.
Since topological stacks are typically not global quotient stacks, one cannot restrict to subgroups of one given group, but simply indexes these spaces on all compact Lie groups, regarded as orbit stacks.
These are the objects of the $\i$-category $\Orb$ of orbit stacks.

\begin{definition}
The \emph{global orbit category} $\Orb$ is the full subcategory of the $\i$-category of $\Stk_{\i}$ on those objects of the form $\BB G$, for $G$ a compact Lie group.
\end{definition}

\begin{definition}
An \emph{orbispace} is a functor $\Orb^{\op}\to\SS$.
The $\infty$-category of orbispaces is the $\infty$-category $\Spc_{\Orb} = \Fun(\Orb^{\op},\SS)$ of functors $\Orb^{\op}\to\SS$.
\end{definition}

\begin{remark}
The Yoneda embedding defines a functor
\[\Stk_{\i} \to \Spc_{\Orb}, \qquad X \,\mapsto\, (\B G \mapsto \Map_{\Stk_{\i}}(\B G, X)).\]
In this way, we may associate to every stack an orbispace in a canonical and functorial way.
In particular, when later defining cohomology theories for orbispaces, these will also determine cohomology theories for topological stacks.
\end{remark}

In order to understand the $\i$-category of orbispaces, we must first understand the mapping spaces in $\Orb$ or more generally mapping spaces in $\Stk_{\i}$. As with almost any computation about stacks, this is done by choosing \emph{presentations} of the relevant stacks: given a topological groupoid $X_{\bullet}$, we denote the represented prestack the same way and obtain a stack $X_{\bullet}^{\dag}$ by stackification. We view $X_{\bullet}$ as a presentation of the associated stack. The example we care most about is the presentation  $\{G \rrarrow \pt\}$ of $\B G$. 

One thing a presentation allows us to compute is the homotopy type of a stack, based on the homotopy type of a topological groupoid: for a topological groupoid $X_{\bullet}$, we consider the space $|X_{\bullet}|$ arising as the homotopy colimit $\colim_{\Delta^{\op}}X_{\bullet}$ or equivalently as the fat realization of $X_{\bullet}$ (cf.\ \cref{app:GlobalOrbitModels}). While the following two results are not strictly necessary to understand $\Orb$, they are reassuring pieces of evidence that our notion of mapping spaces between topological stacks is ``correct''. 

\begin{lemma}\label{lem:pstkmap}
     Let $X_\bullet$ be a topological groupoid.
     There are natural weak equivalences of spaces
    \[
     |X_\bullet|\simeq |\Map_{\Pre\Stk}(\Delta_{\Top}^{\ast},X_\bullet)|\simeq \Map_{\Pre\Stk_\i}(\pt,X_\bullet),
     \]
     where the geometric realization in the middle is with respect to $\ast$. 
 \end{lemma}
 \begin{proof}
     By the Yoneda lemma, the space of maps from $\Delta_{\Top}^n$ to $X_{\bullet}$ in $\PreStk$ is naturally equivalent to $X_{\bullet}(\Delta_{\Top}^n)$ (where we view a groupoid via the nerve as an object in $\Spc$). This in turn is represented by the simplicial set $\Hom_{\Top}(\Delta^n_{\Top}, X_{\bullet}) = (\Sing X_{\bullet})_n$. Since every simplicial set is the homotopy colimit along $\Delta^{\op}$ of its levels, we obtain $X_{\bullet}(\Delta_{\Top}^n) \simeq |(\Sing X_{\bullet})_n|$. Thus, 
    \[
        |\Map_{\PreStk}(\Delta^{\ast}, X_{\bullet})| \simeq | |(\Sing X_{\bullet})_{\ast}||,
    \]
     where the inner realization is with respect to $\bullet$ and the outer with respect to $\ast$.\footnote{We stress that the double bars stand here for an iterated geometric realization and not for a fat geometric realization.} 
    
     To obtain the first equivalence in the statement, we just have to use the natural equivalence $|(\Sing X_n)_{\ast}| \to X_n$ for every $n$ and that the order of geometric realization does not matter (as these are just colimits over $\Delta^{\op}$). 

    The second equivalence in the statement of the lemma holds by definition.
 \end{proof}
 
\begin{proposition}\label{prop:HomotopyType}
Let $X_{\bullet}$ be a topological groupoid. Then there is a natural equivalence $|X_{\bullet}| \simeq \Map_{\Stk_{\i}}(\pt, X_{\bullet}^{\dag})$. 
\end{proposition}
\begin{proof}
We recall from \cite[Section 2.5]{GH} that there is a class of \emph{fibrant} topological groupoids. Moreover, there is a fibrant replacement functor $X_{\bullet} \to \fib(X_{\bullet})$, which induces by \cite[Proposition 3.7]{GH} an equivalence 
$X_{\bullet}^{\dag} \to \fib(X_{\bullet})^{\dag}$. By \cite[Section 3.4]{GH}, this implies that $|X_{\bullet}| \to |\fib(X_{\bullet})|$ is an equivalence. This reduces the statement of the proposition to the class of fibrant topological groupoids $X_{\bullet}$. By \cite[Section 3.3]{GH}, $\fib(X_{\bullet}) \to \fib(X_{\bullet})^{\dag}$ induces an equivalence on paracompact spaces like $\Delta^n_{\Top}$.

Thus, 
\[\uMap_{\Stk_\i} (\pt, X_{\bullet}^{\dag})\simeq |\Map_{\Stk}(\Delta^n_{\Top},X_{\bullet}^{\dag})|\simeq |\Map_{\PreStk}(\Delta^n_{\Top}, X_\bullet)|.\]  
\cref{lem:pstkmap} gives $\Map_{\PreStk}(\Delta^n_{\Top}, X_\bullet)\simeq |X_\bullet|$, finishing the proof.
\end{proof}

 We are now ready to calculate the mapping spaces in $\Orb$. In the proof we will use a concept of interest in its own right, namely action groupoids: Given a topological group $G$ and a $G$-space $X$, the action of $G$ on $X$ may be encoded via a topological groupoid, called the {\em action groupoid}, which we will denote by $\{G\times X\rrarrow X\}$, where one of the arrows is the projection and the other one the action. 

\begin{proposition}\label{prop:orbcalc}
For compact Lie groups $G$ and $H$, there is a chain of equivalences
\[
\Map_{\Orb}(\BB H,\BB G) \simeq |\uMap_{\TopGpd}(H \rrarrow \pt, G \rrarrow \pt)| \simeq \Map_{\Lie}(H,G)_{hG},
\]
which is natural in morphisms of compact Lie groups. Here, the action of $G$ on the homomorphisms is by conjugation.
\end{proposition}
\begin{proof}
    By \cref{lem:GpdStk}, there is a natural equivalence
\[\Map_{\Orb}(\BB H,\BB G) \simeq |\uMap_{\TopGpd}(H \rrarrow \pt, G \rrarrow \pt)|.\]
One easily calculates that the mapping groupoid $\underline{\Map}_{\TopGpd}(H\rrarrow \pt,G\rrarrow \pt)$ is isomorphic to the action groupoid $\{\Map_{\Lie}(H,G)\times G \rrarrow\Map_{\Lie}(H,G)\}$, with the action given by conjugation. As the realization of an action groupoid is the homotopy quotient, the result follows.
\end{proof}

\begin{remark}
An alternative approach to prove \cref{prop:orbcalc} would be to use fibrant topological groupoids again as in the proof of \cref{prop:HomotopyType}. And there is no reason to expect in general an equivalence between $\Map_{\Stk_{\i}}(X_{\bullet}^{\dag}, Y_{\bullet}^{\dag})$ and $|\uMap_{\TopGpd}(X_{\bullet}, Y_{\bullet})|$ if $Y_{\bullet}$ is not fibrant and one should rather consider derived mapping groupoids between topological groupoids as in \cite{GH}. The special feature we used in the preceding proposition is that the object space of the source is contractible. 
\end{remark}

\begin{remark}\label{Orbast}
Denote by $\CptLie$ the topological category of compact Lie groups, which we view as an $\infty$-category. We claim that the association $G \mapsto \B G$ defines an equivalence $\CptLie \to \Orb_*$, where $\B G$ is equipped with the base point coming from the inclusion of the trivial group into $G$.

To see this, we first invoke \cref{prop:orbcalc}, which implies the existence of a functor $\CptLie\to\Orb\subset\Stk_\i$ which sends $G$ to $\B G$.

Since $\pt$ is an initial and terminal object of $\CptLie$, we see that $\CptLie$ is naturally a pointed $\i$-category, and the functor $\CptLie\to\Orb$ therefore factors through the projection $\Orb_*\to\Orb$.
The equivalence $\CptLie\simeq\Orb_*$ now follows from direct calculation: The diagram
\[
\xymatrix{
\Map_{\CptLie}(H,G)\ar[r]\ar[d] & \Map_{\CptLie}(H,G)_{hG}\ar[r]\ar[d] & BG\ar[d]\\
\Map_{\Orb_*}(\BB H,\BB G)\ar[r] & \Map_{\Orb}(\BB H,\BB G)\ar[r] & \Map_{\Orb}(\pt,\BB G)}
\]
defines a morphism of fiber sequences in which the middle and rightmost vertical maps are equivalences.
It follows that the leftmost vertical map is also an equivalence.
\end{remark}
We conclude this section with a couple of observations about the relationship between orbispaces and $G$-spaces.
Above we have defined for a topological group $G$ and a $G$-space $X$ an action groupoid.
This is functorial in morphisms of $G$-spaces: given a second $G$-space $Y$ and a $G$-equivariant map $Y\to X$, we obtain a morphism of action groupoids \[\{G\times Y\rrarrow Y\}\to\{G\times X\rrarrow X\}\]
by applying $f\colon Y\to X$ on the level of objects and $(\id_G,f)\colon G\times Y\to G\times X$ on the level of morphisms.
Moreover, the resulting functor $\Top^G \to \TopGpd$ is compatible with the topological enrichement.
However, it is not fully faithful unless $G$ is trivial, as morphisms of topological groupoids need not be injective on stabilizer groups.
This is remedied by insisting that the functors are compatible with the projection to $\{G\rrarrow\pt\}$.
\begin{proposition}\label{prop:fullfaithful}
    The action groupoid functor $\Top^G\to\TopGpd$, for any topological group $G$, factors through the projection $\TopGpd_{/\{G\rrarrow\pt\}}\to\TopGpd$, and the induced functor
    \[
    \Top^G\to\TopGpd_{/\{G\rrarrow\pt\}}
    \]
    is fully faithful (even as a functor of topologically enriched categories).
\end{proposition}
\begin{proof}
    Let $X$ and $Y$ be topological spaces equipped with an action of the topological group $G$.
    The space of topological groupoid morphisms between the associated topological groupois is the subspace of the product
    \[
    \underline{\Map}_{\Top}(Y,X)\times\underline{\Map}_{\Top}(G\times Y,G\times X)
    \]
    consisting of those maps which are compatible with the groupoid structure as well as the projection to $\{G\rrarrow\pt\}$.
    These conditions are exactly what is needed to ensure that for an allowed pair $(f,h)$, the map $h$ is determined by $f$ and $f$ is $G$-equivariant:
    Writing $\underline{\Map}_{\Top}(G\times Y,G\times X)\cong\underline{\Map}_{\Top}(G\times Y,G)\times\underline{\Map}_{\Top}(G\times Y,X)$, we see that the map $G\times Y\to G$ must be the projection in order to be compatible with the map down to $\{G\rrarrow\pt\}$.
    Moreover, in order to be compatible with the groupoid structure, the map $G\times Y\to X$ must factor as the composite of the projection $G\times Y\to Y$ followed by $f$ and the map $f$ must be equivariant.
    It is straightforward to verify that the topologies agree.
\end{proof}

To deduce a corresponding $\infty$-categorical statement would require to compare the associated $\infty$-category of a slice category with the slice category in an associated $\infty$-category. As checking the necessary fibrancy conditions would take us too far afield, we will rather cite the following result from \cite[Lemma 6.13]{LinskensNardinPol}, who give a more $\infty$-categorical proof. For a compact Lie group $G$, we denote by $\Orb_G$ the $\infty$-category associated to the full topological subcategory of $G$-spaces on the $G/H$ for $H\subset G$ a closed subgroup.

\begin{proposition}\label{cor:OrbBG}
Given a compact Lie group $G$, there is a fully faithful embedding
\[\Orb_G \to \Orb_{/\B G}, \qquad G/H \mapsto (\B H \to \B G),\] with essential image those morphisms that are representable, i.e.\ correspond to inclusions of subgroups.
\end{proposition}

Call more generally a morphism $X \to Y$ in $\Spc_{\Orb}$ \emph{representable} or \emph{faithful} if is right orthogonal to $\B G \to \B G/K$ for all closed normal subgroups $K\subset G$ for all compact Lie groups $G$, i.e.\ we have an essentially unique lift in every diagram
\[
\xymatrix{
\B G\ar[d] \ar[r] & X \ar[d] \\
\B G/K \ar[r] \ar@{-->}[ur]& Y.
}
\]
Intuitively, this means that $X \to Y$ is injective on all automorphism groups of points of these stacks. It was observed in \cite[Example 3.4.2]{RezkGlobal} and proven in \cite[Proposition 6.14]{LinskensNardinPol} that indeed a morphism $\B H \to \B G$ is representable if and only if it corresponds to an injection $H\to G$. 

Recall that a \emph{factorization system} in an $\infty$-category $\CC$ consists of two classes of morphisms, $L$ and $R$, that are orthogonal to each other and such that every morphism in $\CC$ factors as $fg$ with $f\in L$ and $g\in R$ (see \cite[Section 24]{JoyalNotes} and \cite[Section 5.2.8]{HTT} for details). Similar statements to the following were already considered in \cite{RezkGlobal} and \cite{LinskensNardinPol}.
\begin{lemma}\label{lem:factorizationOrb}
    There exists a factorization system $(L,R)$ on $\Spc_{\Orb}$ such that $R$ consists exactly of the representable morphisms and every quotient map $\B G \to \B G/K$, for $K\subset G$ closed, lies in $L$. This factorization system restricts to $\Orb$. 
\end{lemma}
\begin{proof}
    Let $S$ be the set of morphisms $\B G \to \B G/K$ for closed subgroups $K\subset G$. By \cite[Proposition 5.5.5.7]{HTT}, there exists a factorization system $(L,R)$ on $\Spc_{\Orb}$ such that $L$ is the saturation of $S$ and $R$ consists of those morphisms that are right-orthogonal to $L$ (or, equivalently, to $S$). By definition, these are exactly the representable morphisms. 
    
    In $\Orb$, we can factor a morphism $\B G \to \B H$ corresponding to a group homomorphism $\varphi\colon G \to H$ as $\B G \to \B(G/\ker(\varphi)) \simeq \B\im(\varphi)\to \B H$.  The first arrow lies in $L$ and the second in $R$; thus the factorization system restricts. 
\end{proof}

\begin{lemma}\label{lem:factorizationfinal}
    Let $(L, R)$ be a factorization system on an $\infty$-category $\CC$ and let $X \in \CC$ be an object. Let $\DD \subseteq \CC$ be a full subcategory such that every morphism $D \to X$ with $D\in \DD$ factors as $D \xrightarrow{f} D' \xrightarrow{g} X$ with $D'\in \DD$, $f\in L$ and $g\in R$. Then the inclusion of the full subcategory of $\DD_{/X}$ on the morphisms $Y \to X$ lying in $R$ is final.\footnote{There is, unfortunately, no uniform terminology for final/cofinal functors. We use \emph{final} for what is called cofinal in \cite{HTT}, i.e.\ for functors that induce an equivalence of colimits indexed by the source and target $\infty$-categories. This is compatible with the terminology in \cite{JoyalNotes} and \cite{Cisinski}.}
\end{lemma}
\begin{proof}
    Consider for every $f\colon Y \to X$ with $Y\in \DD$, the full subcategory $\EE$ of $(\DD_{/X})_{f/}$ on those factorizations
    \[\xymatrix{Y\ar[dr]^g \ar[rr]^f && X \\
    &Z \ar[ur]^h&
    }\]
    with $h\in R$. We have to show that $\EE$ is contractible. Choose a factorization $Y \xrightarrow{l} T \xrightarrow{r} X$ with $T\in \DD$, $l\in L$ and $r\in R$. We claim that this is terminal in $\EE$. Equivalently, the space of dashed arrows in 
    \[
    \xymatrix{
    Y \ar[r]^g \ar[d]^l & Z \ar[d]^h \\
    T \ar@{-->}[ur]\ar[r]^r & X}
    \]
    is contractible. But this is part of the statement that $L$ and $R$ are orthogonal.
\end{proof}

As in \cref{app:GSpectra}, we write $\Spc^G$ for the $\infty$-category of $G$-spaces. By Elmendorf's theorem, there is an equivalence of $\i$-categories $\Spc^G \xrightarrow{\simeq} \PP(\Orb_G)$ which associates to a $G$-space to its diagram of fixed point spaces. The following result was already observed with a different proof in \cite[Proposition 3.5.1]{RezkGlobal}
\begin{prop}\label{prop:OntoRepresentables}
For any compact Lie group $G$, the functor  
\[\Spc^G \simeq \PP(\Orb_G)\to\PP(\Orb_{/\BB G})\simeq\PP(\Orb)_{/\BB G} = (\Spc_{\Orb})_{/\B G},\]
induced by left Kan extending along the embedding $\Orb_G \to \Orb_{/\B G}$ from \cref{cor:OrbBG},
is fully faithful, with essential image the representable maps $X\to\BB G$.    
\end{prop}
\begin{proof}
    The equivalence $\PP(\Orb_{/\BB G})\simeq\PP(\Orb)_{/\BB G}$ is formal. Moreover, since $\Orb_G \to \Orb_{/\B G}$ is fully faithful, so is $\PP(\Orb_G) \to \PP(\Orb_{/\B G})$. Since mapping out of any $\B H$ in $\Spc_{\Orb} = \PP(\Orb)$ is colimit-preserving, the class of representable morphisms is closed under colimits in $\PP(\Orb)_{/\B G}$. Since the image of $\Orb_G$ in $\Orb$ consists of representable morphisms, the image of $\PP(\Orb_G) \to \PP(\Orb)_{/\B G}$ consists of representable morphisms as well. It remains to show that every representable morphism lies in the image. 

    Let $X \to \B G$ be representable. Restricting $X \in \PP(\Orb_{/\B G})$ to the subcategory of representables $\Orb^{\rep}_{/\B G}$ and left Kan extending again, results in the morphism
    \[
    \colim_{(\B H \to \B G)\in(\Orb_{/\BB G}^{\rep})_{/X}}\BB H\to \colim_{(\B H \to \B G)\in(\Orb_{/\BB G})_{/X}}\BB H \simeq X,
    \]
    approximating $X$ as a colimit of representable orbits $\BB H$ over $\BB G$. Thus it suffices to show that $(\Orb_{/\BB G}^{\rep})_{/X}$ is final in $(\Orb_{/\BB G})_{/X}$. This follows from \cref{lem:factorizationfinal} and the following observation: Given any $\B H \to X$, factor $\B H \to X \to \B G$ as $\B H \to \B H/Q \to \B G$ with $\B H/Q \to \B G$ representable, resulting in 
    \[\xymatrix{ \B H \ar[r] \ar[d] & X \ar[d] \\
    \B H/Q \ar[r] & \B G. }\]
    Since $X \to \B G$ is representable, we obtain a lift $\B H/Q \to X$, which is automatically representable as well. This gives the required factorization $\B H \to \B H/Q \to X$.  
\end{proof}

Recall that given a $G$-space $X$, we can define the stack quotient $[X/G]$ as the stackification of the action groupoid or, equivalently, as the classifying stack for $G$-principal bundles with a $G$-equivariant map to $X$. Every $G$-homotopy equivalence induces an equivalence between the corresponding stack quotients in $\Stk_{\i}$ and thus sending a $G$-CW-complex $X$ to $[X/G]$ defines a functor $\Spc^G \to \Stk_{\i}$.

\begin{proposition}\label{prop:embeddingG}
For every compact Lie group $G$, 
the functor $\SS^G\to\Stk_{\i} \to \Spc_{\Orb}$, which sends $X$ to $[X/G]$, preserves colimits. In particular, the resulting functor $\SS^G \to (\Spc_{\Orb})_{/\B G}$ agrees with the one considered in \cref{prop:OntoRepresentables}. 
\end{proposition}
\begin{proof}
We need to show that $\Map_{\Stk_{\i}}(\B H, [-/G])\colon \Spc^G \to \Spc$ preserves colimits for every compact Lie group $H$. Letting $H$ vary, this implies that the functor 
\[\SS^G\to\Stk_{\i} \to \Spc_{\Orb}, \qquad X \mapsto [X/G]\]
preserves colimits. 

Let $X$ be a $G$-space and $\Phi\colon \B H \to \B G$ be a map in $\Stk_{\i}$, which after choice of basepoints corresponds to a homomorphism $\varphi\colon H \to G$. As in the proof of \cref{lem:GpdStk}, one shows using that all principal bundles on $\Delta^n$ are trivial that the natural maps
\begin{align*}
    \Map_{\TopGpd/(G \rrarrow \pt)}(H \times \Delta^n\rrarrow \Delta^n, G\times X\rrarrow X) &\to \Map_{\Stk/\B G}(\B H \times \Delta^n, [X/G])
\end{align*}
are equivalences of (discrete) groupoids for all $n$. After geometrically realizing the resulting simplicial objects, the right-hand side becomes $\Map_{(\Stk_{\i})_{/\B G}}(\B H, [X/G])$, while we can identify the left-hand side with the space of $\varphi$-equivariant maps $\pt \to X$ (using essentially the argument of \cref{prop:fullfaithful}). This in turn is the same as the $\im(\varphi)$-fixed points of $X$. As a colimit of $G$-spaces induces a colimit of $\im(\varphi)$-fixed points, the functor $\Map_{(\Stk_{\i})_{/\B G}}(\B H, [-/G])\colon \Spc^G \to \Spc$ preserves colimits. As it is the fiber of
\[
\Map_{\Stk_{\i}}(\B H, [-/G]) \to \Map_{\Stk_{\i}}(\B H, \B G)
\]
over $\Phi$, we see that $\Map_{\Stk_{\i}}(\B H, [-/G])$ preserves colimits as well. 

The last point follows since the functor $\Orb_G \to \Orb\to \Spc_{\Orb}$ agrees with $X \mapsto [X/G]$ and a colimit-preserving functor from $\Spc^G \simeq \PP(\Orb_G)$ is determined by its restriction to $\Orb_G$. 
\end{proof}

\section{Abelian group objects in $\infty$-categories}\label{sec:Abelian}
The goal of this section is to construct from any preoriented abelian group object in a suitable $\infty$-category $\XX$ a colimit-preserving functor $\Spc_{\Orb} \to \XX$. We achieve this in \cref{MainConstruction}, which is the basis for our construction of equivariant elliptic cohomology.

To start, we consider for an $\i$-category $\XX$ with finite products the $\i$-category $\Ab(\XX)$ of abelian group objects in $\XX$.
We caution the reader that ``abelian group object'' is meant in the {\em strict} sense as opposed to the more initial notion of commutative (that is, $E_\i$) group object.
In order to make this precise, it is convenient to invoke the language of universal algebra.
We refer the reader to \cite[Section 5.5.8]{HTT} or \cite[Appendix B]{GGN} for relevant facts about Lawvere theories in the context of $\infty$-categories.

Let $\TT_{\Ab}$ denote the Lawvere theory of abelian groups, which we regard as an $\i$-category.
That is, $\TT_{\Ab}^{\op}\subset\Ab=\Ab(\mathrm{Set})$ is the full subcategory of abelian groups (in sets) consisting of the finitely generated free abelian groups.
A skeleton of $\TT_{\Ab}$ is given by the $\i$-category with object set $\mathbb{N}$ and (discrete) mapping spaces
\[\Map(q,p)\cong\Map_{\Ab}(\Z^p,\Z^q)\cong (\Z^q)^{\times p}.\]

\begin{definition}
Let $\XX$ be an $\infty$-category.
The $\infty$-category $\Ab(\XX)$ of \emph{abelian group objects} in $\XX$ is the $\infty$-category $\Fun^{\Pi}(\TT_{\Ab}^{\op},\XX)$ of finite product preserving functors from $\TT_{\Ab}$ to $\XX$.
\end{definition}

In particular, if $\XX$ is the ordinary category of sets, then $\Ab(\XX)\simeq\Ab$ recovers the ordinary category of abelian groups.
However, we will be primarily interested in the case in which $\XX$ is an $\infty$-topos, such as the $\infty$-category $\SS$ of spaces, or sheaves of spaces on a site.
\begin{remark}
It is important to note that $\Ab(\SS)$ is not equivalent to the $\infty$-category of (grouplike) $\E_\infty$-spaces since $\TT_{\Ab}$ is a full subcategory of the ordinary category of abelian groups.
Should one need to consider this $\infty$-category, one would instead start with the Lawvere theory $\TT_{\E_\infty}$ of (grouplike) $\E_\infty$-spaces.
Rather, $\Ab(\SS)$ is equivalent to the $\i$-category of connective chain complexes, or equivalently, connective $\Z$-module spectra. 
Note also that $\TT_{\Ab}^{\op}$ is equivalent to the ordinary category of lattices.
See \cite[Section 1.2]{LurEllI} for further details on (strict) abelian group objects in $\infty$-categories.
\end{remark}

\begin{example}\label{TopAb}
Topological abelian groups give examples of abelian group objects in $\Spc$. 
To make this relationship respect the simplicial enrichment of the category of topological abelian groups, we first need to investigate the relationship between abelian group objects in a simplicial category and its associated $\i$-category. Thus let $\CC$ be a simplicial category. Denote by $\Ab(\CC)$ the simplicial category of product preserving functors from $\TT_{\Ab}$ to $\CC$. Denoting the coherent nerve by $N^{\coh}$, there is a natural functor $N^{\coh}\Ab(\CC) \to \Ab(N^{\coh}\CC)$ constructed as follows: Taking the adjoint of the coherent nerve of the evaluation map $\TT_{\Ab} \times \Fun(\TT_{\Ab}, \CC) \to \CC$ gives a map $N^{\coh}(\Fun(\TT_{\Ab}, \CC)) \to \Fun(\TT_{\Ab}, N^{\coh}(\CC))$. Restricting to product preserving functors gives the desired functor $N^{\coh}\Ab(\CC) \to \Ab(N^{\coh}\CC)$.

Let us denote by $\TopAb_{\infty}$ the $\infty$-category arising via the coherent nerve from the simplicial category whose objects are topological abelian groups and whose mapping spaces are the singular complexes of spaces of homomorphisms. From the last paragraph we get a functor $\TopAb_{\infty} \to \Ab(\Spc)$. 

We are interested in the full subcategories of $\TopAb_{\infty}$ on compact abelian Lie groups $\CptLie^{\ab}$ and the slightly bigger full subcategory $\FGLie^{\ab}$ on products of tori and finitely generated abelian groups. We claim that the functor $\FGLie^{\ab} \to \Ab(\Spc)$ is fully faithful. 

To that purpose recall from \cite[Remark 1.2.10]{LurEllI} that there is an equivalence \[B^{\infty}\colon \Ab(\Spc) \to \Mod_{\Z}^{\mathrm{cn}}\]
to connective $\Z$-module spectra, preserving homotopy groups. Thus, for a discrete group $A$, we have $B^{\infty}(A) \simeq A$ and $B^{\infty}(\T^n) \simeq \Sigma \Z^n$. 
Since both $\FGLie^{\ab}$ and $\Mod_{\Z}^{\mathrm{cn}}$ are additive, the general claim reduces to the case of the (topologically) cyclic groups $\T$, $\Z$, and $\Z/n$. We compute that $\Map(\Z,-)$ is the underlying space and $\Map(-,\Z)=0$ unless the source is $\Z$.
Moreover, $\Map(\T,\T)=\Z$, $\Map(\T,\Z/n)=0$, $\Map(\Z/n,\T)=\T[n]$ and $\Map(\Z/n,\Z/n)=\Z/n$, in either $\FGLie^{\ab}$ or $\Mod_\Z^{\cn}$.
\end{example}

\begin{example}
Let $\Orb^{\ab}\subset\Orb$ denote the full subcategory consisting of the abelian orbits, i.e.\ those orbits of the form $\BB G$ for $G$ an abelian compact Lie group.
Every such $\BB G$ has naturally the structure of an abelian group object: The classifying stack functor $\B\colon\CptLie\to\Stk_\i$ preserves products and every abelian compact Lie group defines an abelian group object in the category of compact Lie groups. 
\end{example}

We recall from \cite[Proposition B.3]{GGN}:
\begin{lemma}\label{lem:AbTensored}
    For a presentable $\infty$-category $\XX$, there is an equivalence
    \[\tensor\colon \Ab(\Spc) \tensor \XX \to \Ab(\XX),\]
    where the tensor product is of presentable $\infty$-categories. 
\end{lemma}
 
We recall the notion of a preoriented abelian group object from \cite{Lur07}. 
\begin{defi}
Let $\XX$ be an $\i$-category with finite products.
A \emph{preoriented abelian group object} in $\XX$ is an abelian group object $A\in\Ab(\XX)$ equipped with a morphism of abelian group objects $B\T \to \Map_{\XX}(\pt, A)$. 
We define the $\i$-category $\PreAb(\XX)$ of preoriented abelian group objects in $\XX$  as $\Ab(\XX)\times_{\Ab(\Spc)} \Ab(\Spc)_{B\T/}$, where $\Ab(\XX) \to \Ab(\Spc)$ is corepresented by $\pt$. 
\end{defi}
Equivalently, a preoriented abelian group object is an object $A\in\Ab(\XX)$ with a pointed map $S^2 \to \Map(\pt, A)$. If $\XX$ is presentable, we can alternatively define $\PreAb(\XX)$ as $\Ab(\XX)_{\B \T \tensor \pt/}$. Here, we use that $\Map(\pt, -)\colon \Ab(\XX) \to \Ab(\Spc)$ is right adjoint to $- \tensor \pt$. 

A third equivalent definition is via an algebraic theory. To that purpose, denote by $\TT^{\op}_{\mathrm{preab}} \subset \PreAb(\Spc)$ the full subcategory of preoriented abelian group objects free on a finite set, i.e.\ those equivalent to $\B\T \to \Z^n\times \B\T$.

\begin{lemma}\label{lem:PreAbTheory}
    Let $\XX$ be an $\i$-category with finite products. There is a functorial equivalence $\PreAb(\XX) \simeq \Fun^{\Pi}(\TT_{\mathrm{preab}}, \XX)$. 
\end{lemma}
\begin{proof}
Since the forgetful functor $\PreAb(\Spc) \to \Spc$ is conservative and preserves sifted colimits, \cite[Proposition B.7]{GGN} implies $\PreAb(\Spc) \simeq \Fun^{\Pi}(\TT_{\mathrm{preab}}, \Spc)$. This implies the analogous equivalence for presheaf $\i$-categories. 

For a general $i$-category $\XX$ with finite products, consider the composite
\[\PreAb(\XX) \to \PreAb(\PP(X)) \xrightarrow{\simeq} \Fun^{\Pi}(\TT_{\mathrm{preab}}, \PP(\XX)),\]
where we used that the Yoneda embedding $\XX \to \PP(\XX)$ preserves products. (Note that we have potentially passed to a larger universe.) 

Since the Yoneda embedding is fully faithful, it suffices to show that for any preoriented abelian group $A \in \PreAb(\XX)$, the corresponding functor $\TT_{\mathrm{preab}} \to \PP(\XX)$ takes values in representable presheaves. 
    Thus let $L \in \TT_{\mathrm{preab}}$ be free on a set with $n$ elements. Then for any $A\in \PreAb(\XX)$, the functor \[\Map_{\PreAb(\Spc)}(L, \Map_{\XX}(-, A))\colon \XX^{\op} \to \Spc \]
    is indeed representable by the underlying object of $A^n$. 
\end{proof}

\begin{construction}\label{con:PicardDuality}
Morally, the elliptic cohomology groups for a compact abelian Lie group $G$ are related to the sheaf cohomology groups of the elliptic curve and other related abelian varieties.
In order to make this precise, we invoke a stacky version of Pontryagin duality, which we refer to as \emph{shifted Pontryagin duality}. More precisely, it is given by the functor
\[
\Map_{\Orb}(-,\B\T)\colon(\Orb^{\ab})^{\op}\to\Ab(\SS)_{B\T/} = \PreAb(\Spc),\qquad
X \mapsto \widehat{X}
\]
which sends the abelian orbit stack $X$ to the preoriented abelian group
\[
B\T\simeq\Map(\pt,\B\T)\to\Map(X,\B\T).
\]
Here the preorientation is induced from the projection $X\to\pt$, and the abelian group structure on the mapping space $\Map(X,\B\T)$ is induced pointwise from that on $\BB\T$.

By \cref{prop:orbcalc}, we have $\widehat{\B G} \simeq \widehat{G} \times B \T$, where $\widehat{G}$ is the Pontryagin dual, i.e.\ the space of homomorphisms from $G$ to $\T$.
 \end{construction}

The following proposition is not used in our construction of equivariant elliptic cohomology but shows that shifted Pontryagin duality does not lose any information. 
\begin{proposition}
The shifted Pontryagin duality functor
\[
\Map(-,\B\T)\colon(\Orb^{\ab})^{\op}\to \PreAb(\Spc)
\]
is fully faithful.
\end{proposition}

\begin{proof}
Let $\Orb^{\ab}_{n\ast}\simeq\Orb^{\ab}_*\times_{\Orb^{\ab}}\cdots\times_{\Orb^{\ab}}\Orb^{\ab}_*$ denote the iterated fibered product, which we might regard as the $\infty$-category of $n$-pointed orbits.
We have a morphism of augmented simplicial $\i$-categories
\begin{equation}\label{eq:DoubleSimplicial}
\xymatrix{
(\Orb^{\ab})^{\op}\ar[d] & (\Orb_*^{\ab})^{\op}\ar[l]\ar[d] & (\Orb^{\ab}_{2*})^{\op}\ar@<-.5ex>[l]\ar@<.5ex>[l]\ar[d] & (\Orb^{\ab}_{3*})^{\op}\ar@<-1ex>[l]\ar[l]\ar@<1ex>[l]\ar[d] & \cdots\ar@<-1.5ex>[l]\ar@<-.5ex>[l]\ar@<.5ex>[l]\ar@<1.5ex>[l]\\
\Ab(\SS)_{B\T/} & \Ab(\SS)\ar[l]  & \Ab(\SS)_{\T/}\ar@<-.5ex>[l]\ar@<.5ex>[l] & \Ab(\SS)_{\T\times\T/}\ar@<-1ex>[l]\ar[l]\ar@<1ex>[l] & \cdots\ar@<-1.5ex>[l]\ar@<-.5ex>[l]\ar@<.5ex>[l]\ar@<1.5ex>[l]}
\end{equation}
in which the vertical maps are given by shifted Pontryagin duality, sending an $n$-pointed $X$ to the abelian group object $\Map_{\Orb^{\ab}_{n*}}(X, \B \T)$.
The vertical maps factor as claimed, since the undercategories in question are determined by the shifted Pontryagin dual
\[
\Map_{\Orb^{\ab}_{n\ast}}(\pt,\B\T)\simeq\T^{n-1}
\]
of the initial object $\pt$ in $(\Orb^{\ab})^{\op}$.

For every $n>0$, we claim a natural equivalence 
\begin{equation}\label{eq:ZnEquivalence}\Map_{\Orb^{\ab}_{n\ast}}(\B H, \B G) \simeq \Map_{\Lie}(H,G) \times G^{n-1} \simeq \Map_{\Lie_{/\Z^{n-1}}}(H\times\Z^{n-1},G\times\Z^{n-1}).\end{equation}
Indeed, the $\infty$-category $\Orb_*^{\ab}$ is equivalent to $\CptLie^{\ab}$ by \cref{Orbast}, settling the case $n=1$. Every two pointings of orbits are equivalent and we can thus restrict for $n>1$ to the case that all pointings of $\B G$ are the same canonical map $\pt \xrightarrow{p} \B G$. By definition of the slice category, $\Map_{\Orb^{\ab}_{n\ast}}(\B H, \B G)$ is the product of $\Map_{\CptLie}(H, G) \simeq \Map_{\Orb^{\ab}_{\ast}}(\B H, \B G)$ and the $(n-1)$-fold product of the space of self-homotopies of $p$ in $\Map_{\Orb^{\ab}}(\pt, \B G)$, i.e.\ $G^{n-1}$, giving the claim. One checks that the equivalence in \eqref{eq:ZnEquivalence} is compatible with composition, at least up to homotopy. 

On mapping spaces, the $n$-th vertical arrow in \eqref{eq:DoubleSimplicial} is
\[\Map_{\Orb^{\ab}_{n*}}(\B H, \B G) \xrightarrow{P} \Map_{\Ab(\Spc)_{/T^{n-1}}}(\Map_{\Orb_{n*}^{\ab}}(\B G, \B T), \Map_{\Orb_{n*}^{\ab}}(\B H, \B T)).\]
Using \eqref{eq:ZnEquivalence}, for $n>0$, the target is equivalent to \[\Map_{\Ab(\Spc)_{/T^{n-1}}}(\Map_{\Lie}(G,T)\times T^{n-1}, \Map_{\Lie}(H,T)\times T^{n-1}).\]
Note that $\Map_{\Lie}(G,T)\times T^{n-1}$ is the Pontryagin dual of $G\times \Z^{n-1}$. By the fully faithfulness of Pontryagin duality and of the embedding $\FGLie \to \Ab(\Spc)$ from \cref{TopAb}, this is equivalent to 
\[ \Map_{\Lie_{/\Z^{n-1}}}(H\times\Z^{n-1},T\times\Z^{n-1}) \simeq \Map_{\Orb^{\ab}_{n*}}(\B H, \B G).\]
Tracing through the equivalences shows that this provides the required inverse up to homotopy of $P$. Thus, the vertical maps in \eqref{eq:DoubleSimplicial} (except, possibly, for the one on the far left) are fully faithful.

Observe now that the leftmost horizontal maps in \eqref{eq:DoubleSimplicial} are quotient functors (see Appendix \ref{app:q} for details).
It follows that the left-hand vertical map is also fully faithful (cf.\ the description of mapping spaces in the proof of \cref{prop:qcolim}).
\end{proof}

Our next goal is to associate to any preoriented abelian group object in some $\XX$ a functor $\PreAb(\Spc)^{\op} \to \XX$; precomposing this functor with shifted Pontryagin duality will allow us to define a functor from orbispaces to $\XX$. The following categorical result will not only construct the functor, but also show that it does no lose information. At least parts of the result are well-known (see e.g.\ \cite[Proposition B.3]{GGN}). 
\begin{proposition}\label{prop:FunRAb}
    Let $\TT$ be a theory and $\XX$ a complete $\i$-category with finite products. Denoting $\Fun^{\Pi}(\TT, \XX)$ by $\TT(X)$, the functor
       \[\TT(\XX) \to \Fun^R(\TT(\Spc)^{\op}, \XX), \qquad A \mapsto (B \mapsto \uMap_{\TT}(B, A))\] 
    is an equivalence. Here, for $A\in \TT(\XX)$ and $B\in \TT(\Spc)$, the object $\uMap_{\TT}(B, A)$ represents the functor $X \mapsto \Map_{\TT(\Spc)}(B, \Map_{\XX}(X, A))$. 

    In particular, this applies to $\TT = \TT_{\mathrm{ab}}$ with $\TT(\XX) = \Ab(\XX)$ and $\TT = \TT_{\mathrm{preab}}$ with $\TT(\XX) = \PreAb(\XX)$. 
\end{proposition}
\begin{proof}
    By definition, $\TT(\Spc) = \Fun^{\Pi}(\TT, \Spc)$ agrees with $P_{\Sigma}(\TT^{\op})$ in the sense of \cite[Section 5.5.8]{HTT}. Introducing opposite categories in Proposition 5.5.8.15 in op.cit.\ implies that the restriction functor
    \[\res\colon \Fun^R(\TT(\Spc)^{\op}, \XX) \to \Fun^{\Pi}(\TT, \XX)\simeq \TT(\XX)\]
    is an equivalence. It remains to identify its inverse explicitly. 

    Fixing $A\in \TT(\XX)$, consider the class of $B\in \TT(\Spc)$ such that the functor \[X \mapsto \Map_{\TT(\Spc)}(B, \Map_{\XX}(X, A))\] is representable. Since $\XX$ is complete, this class is closed under colimits. Moreover, for $B$ being free on one generator, the representing object is the underlying object in $\XX$ of $A$ itself, yielding representability for all $B$. Thus, the functor 
    \[\TT(\Spc)^{\op}\times \TT(\XX) \to \PP(\XX), \qquad (A,B)\mapsto (X \mapsto \Map_{\TT(\Spc)}(B, \Map_{\XX}(X, A)))\] 
    factors through a functor \[\TT(\Spc)^{\op}\times \TT(\XX) \to \PP(\XX)\]
    we denote by $\uMap_{\TT}$. 
    This yields in particular the functor $\TT(\XX) \to \Fun^R(\TT(\Spc)^{\op}, \XX)$ from the claim of our proposition. This is inverse to $\res$ if $\uMap_{\TT}(L, A)$ is naturally equivalent to $A(L)$ for all $L\in \TT$. But the functor 
    $\Map_{\TT(\Spc)}(L, \Map_{\XX}(-, A)))$
    is equivalent to $\Map_{\XX}(-, A)(L) = \Map_{\XX}(-, A(L))$ by Yoneda, which is indeed represented by $A(L)$. 
\end{proof}

\begin{remark}\label{rem:PresentableCase}
    If $\XX$ is presentable, $\uMap_{\TT}(B, -)$ defines a right adjoint to \[B\tensor -\colon \TT(\XX)\to \TT(\XX)\] in the sense of \cref{lem:AbTensored}. We further can identify this with the internal mapping object in $\TT(\XX)$ from $B\tensor \pt$.
\end{remark}

\begin{remark}\label{rem:fin}
    Let $\TT(\Spc)^{\fin}$ be the smallest subcategory of $\TT(\Spc)$ that contains the free $\TT$-algebra on one object and is closed under finite colimits. By the same proof as above, $\uMap_{\TT}(B,A)$ exists for $B\in \TT(\Spc)^{\fin}$ and $A\in \TT(\XX)$ if $\XX$ has \emph{finite} limits. The functor
    \[\TT(\XX) \to \Fun^{\mathrm{lex}}(\TT(\Spc)^{\fin, \op}, \XX), \qquad A \mapsto (B \mapsto \uMap(B, A))\]
    into finite-limit preserving functors is an equivalence. Indeed: if $\XX$ is complete,    
    the equivalence $\Ind(\TT(\Spc)^{\fin})\simeq \TT(\Spc)$ implies that the restriction functor \[\Fun^R(\TT(\Spc)^{\op}, \XX) \to  \Fun^{\mathrm{lex}}(\TT(\Spc)^{\fin, \op}, \XX)\]
    is an equivalence. Then \cref{prop:FunRAb} gives the result. The general case can be reduced to that of presheaf $\i$-categories, similarly to the proof of \cref{lem:PreAbTheory}. 
\end{remark}

Combining the preceding proposition with \cref{con:PicardDuality} leads to one of the main constructions of this article. 
\begin{construction}\label{MainConstruction}
Let $\XX$ be a finitely complete $\infty$-category. Note that the shifted Pontryagin duality functor $\Orb^{\ab} \to \PreAb(\Spc)$ actually takes values in $\PreAb(\Spc)^{\fin}$ in the sense of \cref{rem:fin}: $\widehat{\B G}$ is equivalent to the product of  $\widehat{G}$ (with trivial preorientation) and $B\T$ (with tautological preorientation) and the Pontryagin dual of any compact abelian Lie group is a finitely generated discrete group. Precomposition with shifted Pontryagin duality turns thus the equivalence $\PreAb(\XX) \to \Fun^{\mathrm{lex}}(\PreAb(\Spc)^{\op, \fin}, \XX)$ from \cref{rem:fin} into a functor
\[\PreAb(\XX) \to \Fun(\Orb^{\ab}, \XX), \qquad A \,\mapsto\, (X \mapsto \uMap_{\PreAb}(\widehat{X}, A)). \]
This functor is also functorial in $\XX$ in the following sense: Given a functor $F\colon \XX \to \YY$ preserving finite limits, we obtain a commutative square
\[
\xymatrix{ \PreAb(\XX) \ar[r]\ar[d]^{F_*} & \Fun(\Orb^{\ab}, \XX) \ar[d]^{F_*} \\
\PreAb(\YY) \ar[r] & \Fun(\Orb^{\ab}, \YY). }
\] 

Now assume that $\XX$ is additionally cocomplete. Given $A \in \PreAb(\XX)$, left Kan extension along $\Orb^{\ab} \subset \Orb$ defines a functor $\Orb \to \XX$. As $\Spc_{\Orb}$ is the presheaf $\i$-category on $\Orb$, i.e.\ its universal cocompletion, this extends essentially uniquely to a colimit-preserving functor $\Spc_{\Orb}\to \XX$. More precisely, we obtain a functor 
\[\PreAb(\XX) \to \Fun^{\mathrm{L}}(\Spc_{\Orb}, \XX)\] 
into colimit-preserving functors. This construction is natural in $\XX$ with respect to functors that preserve colimits and finite limits.
\end{construction}

\begin{remark}
The preceding construction is closely related to \cite[Proposition 3.1]{Lur07}, whose proof was only outlined though and did not contain the notion of shifted Pontryagin duality. 
\end{remark}

\begin{prop}\label{prop:computation}
    Let $\XX$ be an $\infty$-category with finite limits and $A\in \PreAb(\XX)$. The associated functor $F\colon \Orb^{\ab} \to \XX$ from \cref{MainConstruction} sends $\B G$ to $A[\widehat{G}] := \uMap_{\Ab}(\widehat{G}, A)$. Here, $\widehat{G}$ is the Pontryagin dual, and we leave the forgetful functor $\PreAb(\XX) \to \Ab(\XX)$ implicit. If $G = \fib(\T^n\xrightarrow{\varphi}\T^m)$, then $A[\widehat{G}] \simeq \fib(A(\Z^m) \xrightarrow{\widehat{\varphi}^*} A(\Z^n))$.
\end{prop}
\begin{proof}
    The forgetful functor $\PreAb(\Spc) \to \Ab(\Spc)$ has a left adjoint sending $H$ to $H\times \B\T$, with  preorientation given by inclusion into the second factor. As computed in \cref{con:PicardDuality}, the shifted Pontryagin dual of $\B G$ is equivalent to $\widehat{G} \times \B \T$, i.e.\ lies in the image of the left adjoint. Thus, 
    \[F(\B G) = \uMap_{\PreAb}(\widehat{\B G}, A) \simeq \uMap_{\Ab}(\widehat{G}, A).\]
    The last equivalence uses the defining property of $\uMap$ as in \cref{prop:FunRAb}.
    
    Note that Pontryagin duality sends limits to colimits. Thus, $F$ preserves limits and the last claim follows.
\end{proof}
\begin{example}
We take $\XX = \Spc_{\Orb}$ and $A = \B T$, with the preorientation given by the canonical identification of $B\T$ with $\Map_{\Orb}(\pt, \B\T)$. The corresponding functor $\Orb^{\ab} \to \Spc_{\Orb}$ sends any $X$ to its double shifted Pontryagin dual, which is easily seen to be naturally equivalent to $X$ again. Thus, the functor $\Orb^{\ab} \to \Spc_{\Orb}$ is just the natural inclusion as representable presheaves. We obtain a functor $\Spc_{\Orb} \to \Spc_{\Orb}$ by left Kan extending this inclusion along itself. This functor is \emph{not} the identity, but we claim it rather to be the colocalization functor given by the composite
\[\Spc_{\Orb} \xrightarrow{\iota^*} \Fun(\Orb^{\ab, \op}, \Spc) \xrightarrow{\iota_!}  \Spc_{\Orb},\]
where we denote by $\iota\colon \Orb^{\ab} \to \Orb$ the inclusion and by $\iota_!$ the left adjoint to restriction of presheaves. 

This claim is a special case of the following more general claim: Let $F\colon \CC \to \DD$ be a functor of $\infty$-categories. Denoting by $Y_{\CC}\colon \CC \to \Fun(\CC^{\op}, \Spc)$ the Yoneda embedding, we claim that the natural transformation $\mathrm{Lan}_{F}Y_{\CC} \to F^*Y_{\DD}$ of functors $\DD \to \Fun(\CC^{\op}, \Spc)$, induced by $Y_{\CC} \to F^*Y_{\DD}F$, is an equivalence. In the case that $\CC = \pt$ and $F$ corresponds to an object $d\in \DD$, this boils down to the natural equivalence $\colim_{\DD_{F/e}} \pt \simeq \Map_{\DD}(d,e)$, where $\DD_{F/e}$ is the comma category of $F$ and $\pt \xrightarrow{e} \DD$. For the general case, we note that it suffices to show that  $i^*\mathrm{Lan}_{F}Y_{\CC} \to i^*F^*Y_{\DD}$ is an equivalence for every functor $i\colon \pt \to \CC$. This natural transformation fits into a commutative square

\[
\xymatrix{ \mathrm{Lan}_{Fi}(Y_{\pt})\ar[d]^{\simeq} \ar[r]^-{\simeq} & (Fi)^*Y_{\DD} \simeq i^*F^*Y_{\DD} \\
\mathrm{Lan}_F\mathrm{Lan}_i Y_{\pt} \ar[d]^{\simeq} & \\
\mathrm{Lan}_F i^*Y_{\CC} \ar[r]^{\simeq} &i^*\mathrm{Lan}_F Y_{\CC}  \ar[uu]
}
\]
and is thus an equivalence as the upper horizontal and lower left vertical morphisms are equivalences by the special case treated above and the upper left vertical and the lower horizontal morphism are equivalences by general properties of left Kan extensions (as $i^*$ preserves colimits). 
\end{example}

\section{The case of K-theory}\label{sec:KTheory}
To motivate our subsequent definition of equivariant elliptic cohomology, we will demonstrate how the construction in the preceding section reproduces equivariant K-theory from the strict multiplicative group over K-theory. We do not claim significant originality in this section, as an analogous result for finite groups of equivariance was already proven in \cite[Theorem 4.1.2]{LurEllIII}, and a key point already goes back to \cite{AdamsHaeberlyJackowskiMay}. Our proof is, however, more direct than the one in \cite{LurEllIII}. 

We first recall the definition of equivariant K-theory: For every compact Lie group $G$, we denote by $\Rep_G$ the $\infty$-category associated to the topological category of finite-dimensional complex representations of $G$. This carries the compatible symmetric monoidal structures $\oplus$ and $\tensor$, i.e.\ has the structure of an $E_{\infty}$-semiring category. Multiplicative infinite-loop space theory (originally in particular due to May, Elmendorf and Mandell and developed $\infty$-categorically in \cite{GGN}) yields a connective $E_{\infty}$-ring spectrum $ku^{\B G}$ that deloops the group completion of the $\infty$-groupoid $\Rep_G^{\simeq}$. 

Note that $\Rep_G$ can be identified with the category of vector bundles $\Vect(\B G)$ on $\B G$ and thus is only dependent on the (unpointed) orbispace $\B G$. Thus, $\Rep$ refines to a functor $\Vect$ from $\Orb^{\op}$ to $E_{\infty}$-semiring categories. Hence, $ku^{-}$ refines to a functor $\Orb^{\op} \to \CAlg$. Since $ku^{\pt} = ku$, the functor actually refines further to a functor $\Orb^{\op} \to \CAlg_{ku}$ and we define $KU^{-}\colon \Orb^{\ab} \to \CAlg$ as $ku^{-}\tensor_{ku} KU$. 

The space $\Omega^{\infty}ku^{\B G}$ agrees with the $G$-fixed points of the classical $G$-space $BUP_G \simeq BU_G \times R(G)$ representing equivariant K-theory, and equivariant Bott periodicity implies $ku^{\B G} \to KU^{\B G}$ induces an equivalence on infinite loop spaces. Via limits, we can extend $KU^{-}$ to a functor $\Spc_{\Orb}^{\mathrm{op}} \to \CAlg$ and one can show that $\pi_0$ of the resulting $KU^{[X/G]}$ agrees with the classical $G$-equivariant K-theory of $X$ for every finite $G$-CW-complex $X$. 

This ends our recollections on equivariant K-theory. Next, we need to recall from \cite[Section 1.6.3]{LurEllII} the definition of the strict multiplicative group: We define an abelian group object $\Gm$ in affine spectral schemes by 
\[\Gm(R) = \Map_{\CAlg}(\Sigma^{\infty}_+\Z, R)\simeq \Map_{\mathrm{CMon}(\Spc)}(\Z, \Omega^{\infty}R).\] More precisely, the group structure is given by considering $\Gm(R)$ as a functor from $\TT_{\Ab} \to \Spc$, sending $L$ to $\Map_{\CAlg}(\Sigma^{\infty}_+L, R)$. (Equivalently, $\Gm(R)$ arises by applying the right adjoint of the forgetful functor $\Ab(\Spc)\to \mathrm{CMon}(\Spc)$ to $\Omega^{\infty}R$ with the multiplicative $E_{\infty}$-structure.) By definition, $\Gm$ is corepresented by $\Sigma^{\infty}_+\Z$. We define $\G_{m,S}$ for an $E_{\infty}$-ring $S$ as the (nonconnective) spectral affine $S$-scheme $\Spec S \tensor \Sigma^{\infty}_+\Z$. 

\begin{theorem}\label{thm:ku}
    There is a preorientation on $\G_{m,KU}$ such that the associated functor from $\Orb$ to nonconnective spectral affine schemes constructed in \cref{MainConstruction} is equivalent to $\Spec KU^{-}$. Extending via colimits yields an equivalence of functors on $\Spc_{\Orb}$. 
\end{theorem}

We prove \cref{thm:ku} in two main steps. In the first, we deduce from a result of Adams, Haeberly, Jackowski and May that $KU^{-}$ is Kan extended from $\Orb^{\ab}$ (cf.\ also \cite[Example 6.4.27]{SchGlobal}). In the second, we will identify $KU^{-}$ on $\Orb^{\ab}$ with what we get via \cref{MainConstruction} from $\G_{m,KU}$. 

\begin{lemma}\label{lem:adjointrep}
    Let $\FF$ be a global family of groups (i.e.\ $\FF$ is closed under isomorphisms, subgroups and quotients), and let $\Orb^{\FF}$ be the full subcategory of $\Orb$ on those $\B G$ such that $G\in \FF$. Fix a compact Lie group $G$ and let $\RR$ be the full subcategory of $\Orb^{\FF}_{/\B G}$ on the representable maps $\B H \to \B G$ (i.e.\ those corresponding to injections). Then the inclusion $\RR \subset \Orb^{\FF}_{/\B G}$ is final. 
\end{lemma}
\begin{proof}
Recall that there is a factorization system on $\Orb$: for a homomorphism $\varphi\colon H \to G$, the morphism $\B H \to \B G$ is in the left class iff $\varphi$ is surjective and in the right class iff $\varphi$ is injective. (See \cite[Proposition 6.14]{LinskensNardinPol} for further details.) Concretely, the factorization systems correspond to the factorization of some $\varphi\colon H\to G$ as $H \twoheadrightarrow \im(\varphi) \hookrightarrow G$. This restricts clearly to $\Orb^{\FF}$ for any global family $\FF$ of groups. The conclusion follows from \cref{lem:factorizationfinal}.
\end{proof}

\begin{prop}[Adams, Haeberly, Jackowski, May]\label{prop:InducedfromAbelian}
    The functor $KU^{-}\colon \Orb^{\op} \to \CAlg$ is right Kan extended from $\Orb^{\FF, \op}$ for any global family $\FF$ containing all topologically cyclic groups. This applies in particular to the family of all abelian compact Lie groups.
\end{prop}
\begin{proof}
    By the pointwise formula for Kan extensions, we need to show that the comparison map $KU^{\B G} \to \lim_{\Orb^{\FF}_{/\B G}} KU^{-}$ is an equivalence. Consider the full subcategory $\RR \subset  \Orb^{\FF}_{/\B G}$ of \emph{representable} morphisms $\B H \to \B G$. By \cref{lem:adjointrep}, the inclusion $\RR \subset  \Orb^{\FF}_{/\B G}$ is final, and hence $\RR^{\op} \subset (\Orb^{\FF}_{/\B G})^{\op}$ is initial. Thus it suffices to show that the map $KU^{\B G} \to \lim_{\RR}KU^{-}$ is an equivalence. By \cref{cor:OrbBG}, $\RR \simeq \Orb_G^{\FF}$ for $\Orb_G^{\FF} \subset \Orb_G$ the full subcategory on $G/H$ such that $H\in \FF$. Moreover, $\colim_{G/H\in \Orb_G^{\FF}} G/H$ is a universal space $E\FF$ for the family of subgroups of $G$ that are in $\FF$. Thus, $\pi_*\lim_{\RR}KU^{-} \cong KU_G^{-*}(E\FF)$.
    By Corollary 1.3 from \cite{AdamsHaeberlyJackowskiMay}, $K_G^*(\pt) \to K_G^*(E\FF)$ is an isomorphism and the result follows. 
\end{proof}

The infinite-loop space $\Omega^{\infty}ku$ is the group completion of $\Vect_{\C}^{\simeq} \simeq \coprod BU(n)$. The resulting $E_{\infty}$-map $BU(1) \to \Omega^{\infty}ku$ (with the multiplicative $E_{\infty}$-structure on $\Omega^{\infty}ku$) is adjoint to a map $\CP^{\infty} = BU(1) \to \G_{m,ku}(ku)$ in $\Ab(\Spc)$. This is our choice of preorientation on $\G_{m,ku}$. Postcomposing with $\G_{m,ku} \to \G_{m,KU}$ defines a preorientation on $\G_{m,KU}$. 

\begin{prop}\label{prop:EquivalenceOnAbelian}
    The functor 
    \[\Orb^{\ab}\times \CAlg_{ku} \to \Spc, \qquad (X,R) \mapsto \Map_{\PreAb}(\Map(X, \B\T), \G_{m, ku}(R)) \] 
    is equivalent to 
    \[\Orb^{\ab}\times \CAlg_{ku} \to \Spc, \qquad (X,R) \mapsto \Map_{\CAlg_{ku}}(ku^X, R) \]
\end{prop}
The corresponding result for $KU$ instead of $ku$ follows. Thus, $\Spec KU^{-}$ is naturally equivalent to $\Map_{\PreAb}(\Map(-, \B\T)\tensor \pt, \G_{m,KU})$, where $\pt$ is the final object of affine $KU$-schemes (i.e.\ $\Spec KU$). Thus, \cref{prop:InducedfromAbelian} and \cref{prop:EquivalenceOnAbelian} imply \cref{thm:ku}.

The basic idea of \cref{prop:EquivalenceOnAbelian} is that all irreducible $G$-representations for an abelian group $G$ are $1$-dimensional, which implies $ku^{\B G} \simeq ku\tensor \widehat{G}_+$. To ensure functoriality in $\Orb^{\ab}$, we will formulate this in orbispace terms. 

\begin{lemma}\label{lem:OneDimensional}
     For $X\in \Orb^{\ab}$, there is a natural equivalence
    \[\Phi\colon ku \tensor_{\Sigma^{\infty}_+BU(1)} \Sigma^{\infty}_+\Map_{\Orb}(X, \B U(1)) \xrightarrow{\simeq} ku^{X}\]
    of $E_{\infty}$-$ku$-algebras.
\end{lemma}
\begin{proof}
    Let $X \in \Orb$. By adjunction, maps 
    \[\Phi\colon ku \tensor_{\Sigma^{\infty}_+BU(1)} \Sigma^{\infty}_+\Map_{\Orb}(X, \B U(1)) \to ku^{X}\]
    in $\CAlg_{ku}$ correspond to  $E_{\infty}$-maps $\phi\colon \Map_{\Orb}(X, \B U(1)) \to \Omega^{\infty} ku^{X}$ under $BU(1)$, with the multiplicative structure $E_{\infty}$-structure on the target. 
    
    The $E_{\infty}$-monoid $\Map_{\Orb}(X, \B U(1))$ includes as the $1$-dimensional part into $\Vect(X)^{\simeq}$, whose group completion is $\Omega^{\infty} ku^{X}$. Composing inclusion and group completion yields thus a map $\phi$ and a corresponding map $\Phi$, which is natural in $X$.
    
    We will show that $\Phi$ is an equivalence if $X = \B G$ for $G$ abelian.  Since \[\Map_{\Orb}(\B G, \B U(1)) \simeq \widehat{G} \times BU(1),\] 
     the source of $\Phi$ simplifies to $ku \tensor \widehat{G}_+$. More precisely, $\Phi$ arises as the group completion of 
     \[\varphi\colon \Vect_{\C}^{\simeq}\tensor\, \widehat{G} \to \Rep_G, \qquad (V_L)_{L\in \widehat{G}} \mapsto \bigoplus_{L\in \widehat{G}} V_L\tensor L;\] 
     here we interpret the tensor product $\Vect_{\C}^{\simeq}\tensor\, \widehat{G}$ of $E_{\infty}$-spaces as the finitely supported part in $\prod_{\widehat{G}} \Vect_{\C}^{\simeq}$, using that finite direct sums and finite products agree in $E_{\infty}$-spaces and the tensor product commutes with filtered colimits. By Schur's lemma, $\varphi$ is fully faithful, and it is essentially surjective since all $G$-representations decompose into one-dimensional representations.
\end{proof}

\begin{proof}[Proof of \cref{prop:EquivalenceOnAbelian}: ]
    For every $\B G\in \Orb^{\ab}$ and $R\in \CAlg_{ku}$, \cref{lem:OneDimensional} implies a natural chain of equivalences 
    \begin{align*}
        \!\!\!\Map_{\PreAb(\Spc)}(\Map(X, \B\T), \G_{m, ku}(R)) &\simeq \Map_{\Ab(\Spc)}( \Map_{\Orb}(\B G, \B U(1)), \Gm(R))^{hBU(1)}\\
        &\simeq \Map_{\CAlg}( \Sigma^{\infty}_+\Map_{\Orb}(\B G, \B U(1)), R)^{hBU(1)}\\
        &\simeq \Map_{\CAlg_{ku}}(ku\!\!\!\!\!\!\underset{\Sigma^{\infty}_+BU(1)}{\otimes} \!\!\!\!\!\!\Sigma^{\infty}_+\Map_{\Orb}(\B G, \B U(1)), R)\\
        &\simeq \Map_{\CAlg_{ku}}(ku^{\B G}, R).\qedhere
    \end{align*}
\end{proof}
\section{Elliptic curves, formal completions and orientations}\label{sec:Orientations}
The goal of this section is to review the notion of an oriented elliptic curve over a non-connective spectral Deligne--Mumford stack, following the ideas of Lurie in \cite{Lur07} and \cite{LurEllII}. In the affine case, an orientation of an elliptic curve $\Ef \to \Spec R$ consists of an equivalence of the formal completion $\widehat{\Ef}$ with the formal group $\Spf R^{B\T}$, providing a purely spectral analogue of the algebro-topological definition of an elliptic cohomology theory (see \cite[Definition 1.2]{Lur07} or \cite[Definition 1.2]{AHR01}). While the notion of a preorientation suffices for the definition of equivariant elliptic cohomology, one only expects good properties if one starts with an oriented elliptic curve. Setting up the necessary definitions in the non-connective spectral Deligne--Mumford case will occupy the rest of this section.

We recall from \cite[Definition 1.4.4.2]{SAG} the notion of a non-connective spectral Deligne--Mumford stack and point to the end of the introduction for our conventions about these objects. Our first goal is to define an $\infty$-topos of sheaves on a given non-connective spectral Deligne--Mumford stack. Since simply defining it to be sheaves on the site of all morphisms into our given stack would run into size issues, we impose the following finiteness condition. Here and in the following, we will use for a non-connective spectral Deligne--Mumford stack $\Mf = (\MM, \OO_{\Mf})$ the shorthand $\tau_{\geq 0}\Mf$ for $(\MM, \tau_{\geq 0}\OO_{\Mf})$.
\begin{definition}\label{def:almostfinitepres}
A non-connective spectral Deligne--Mumford stack $\Mf = (\MM, \OO_{\Mf})$ is called \emph{quasi-compact} if every cover $\coprod U_i \to \ast$ in $\MM$ has a finite subcover. A morphism $f\colon \Nf \to \Mf$ is called \emph{quasi-compact} if for every \'etale morphism $\Spec R \to \Mf$ the pullback $\Nf \times_{\Mf} \Spec R$ is quasi-compact (cf.\ \cite[Definition 2.3.2.2]{SAG}). 

A morphism $f\colon \Nf \to \Mf$ of spectral\footnote{We recall that a spectral Deligne--Mumford stack is a non-connective spectral Deligne--Mumford stack whose structure sheaf is connective.} Deligne--Mumford stacks is called \emph{almost of finite presentation} if it is locally almost of finite presentation in the sense of \cite[Definition 4.2.0.1]{SAG} and quasi-compact.

We call a morphism $f\colon \Nf \to \Mf$ of non-connective spectral Deligne--Mumford stacks \emph{almost of finite presentation} if $f$ is the pullback of a morphism $\Nf' \to \tau_{\geq 0}\Mf$ almost of finite presentation along $\Mf \to \con \Mf$, where $\Nf'$ is also spectral Deligne--Mumford. 
\end{definition}

\begin{example}\label{ex:etale}
Every map $\Spec R \to \Mf$ into a non-connective spectral Deligne--Mumford stack is quasi-compact, as affine non-connective spectral schemes are quasi-compact by \cite[Proposition 2.3.1.2]{SAG} and we assumed all non-connective spectral Deligne--Mumford stacks to be quasi-separated. 

Moreover, \'etale morphisms are always locally almost of finite presentation (as follows e.g.\ by \cite[Corollary 4.1.3.5]{SAG}). Thus, every \'etale morphism $\Spec R \to \Mf$ (and more generally every quasi-compact \'etale morphism) is almost of finite presentation.
\end{example}

\begin{definition}\label{def:Shv}
We define the \emph{big \'etale site} of a spectral Deligne--Mumford stack $\Mf$  to be the full sub-$\infty$-category of spectral Deligne--Mumford stacks over $\Mf$ that are almost of finite presentation; coverings are given by jointly surjective \'etale morphisms. We define $\Shv(\Mf)$ to be the $\infty$-category of space-valued sheaves on the big \'etale site of $\Mf$. 

If $\Mf$ is more generally a \emph{non-connective} spectral Deligne--Mumford stack, we define $\Shv(\Mf)$ to be $\Shv(\tau_{\geq 0}\Mf)$.  
\end{definition}
\begin{remark} 
We imposed the finiteness condition in our definition of the big \'etale site to ensure that it is a small $\infty$-category (see \cref{lem:smallbig}). Thus, $\Shv(\Mf)$ becomes an $\infty$-topos.

Our insistence in the conventions that all our spectral Deligne--Mumford stacks are quasi-separated and 
locally noetherian is connected to these finiteness conditions. Milder finiteness conditions in \cref{def:Shv} would result in milder conditions on our stacks. 
\end{remark}

\begin{notation}
We will often use the notation $\pt$ for the terminal object in the $\infty$-topos $\Shv(\Mf)$; this is represented by $\id_{\con\Mf}\colon \con\Mf \to \con\Mf$ in the big \'etale site of $\con\Mf$. 
\end{notation}
 
Next we want to define the notion of an elliptic curve. 
\begin{definition}
Let $\Mf$ be a non-connective spectral Deligne--Mumford stack. An \emph{elliptic curve}\footnote{If the base is affine, our definition coincides with what Lurie calls \emph{strict} elliptic curves. As we do not consider non-strict elliptic curves in this article, we drop the adjective.} $\Ef$ over $\Mf$ is an abelian group object in the $\infty$-category of non-connective spectral Deligne--Mumford stacks over $\Mf$ such that
\begin{enumerate}
\item $\Ef \to \Mf$ is flat in the sense of \cite[Definition 2.8.2.1]{SAG},
\item $\tau_{\geq 0}\Ef \to \tau_{\geq 0}\Mf$ is almost of finite presentation and proper in the sense of \cite[Definition 5.1.2.1]{SAG},
\item For every morphism $i\colon \Spec k \to \tau_{\geq 0}\Mf$ with $k$ an algebraically closed (classical) field the pullback $i^*\Ef \to \Spec k$ is a classical elliptic curve. 
\end{enumerate}
\end{definition}

\begin{remark}
One can show that given an elliptic curve $\Ef$ over $\Mf$, the connective cover $\tau_{\geq 0}\Ef$ is an elliptic curve over $\tau_{\geq 0}\Mf$ and that this procedure provides an inverse of the base change from $\tau_{\geq 0}\Mf$ to $\Mf$. Thus, we obtain an equivalence between the $\infty$-categories of elliptic curves over $\Mf$ and over $\tau_{\geq 0}\Mf$ (cf.\ \cite[Remark 1.5.3]{LurEllI}).

We will view an elliptic curve $\Ef$ over $\Mf$ as an abelian group object in $\Shv(\Mf)$ by using the functor of points of $\con \Ef$. 
\end{remark}

We fix from now on a non-connective spectral Deligne--Mumford stack $\Mf$ and an elliptic curve  $p\colon \Ef \to \Mf$ over $\Mf$.

\begin{definition}
A \emph{preorientation} of $\Ef$ is a preorientation of $\Ef$ as an abelian group object in $\Shv(\Mf)$, i.e.\ a morphism of abelian group objects $B\T \tensor \pt \to \Ef$ or, equivalently, $B\T \to \Omega^{\infty}\Ef(\Mf)$.
\end{definition}

If $\Nf\to\Mf$ is a map from a non-connective spectral scheme, then $\Ef\times_\Mf\Nf\to\Nf$ is a morphism of nonconnective spectral schemes equipped with a unit section $i\colon\Nf\to\Ef\times_\Mf\Nf$ which is a closed immersion; in particular, it determines a closed subspace of the underlying topological space of $\Ef\times_\Mf \Nf$.
The {\em complement} of $i$ in $\Ef\times_\Mf\Nf$ is the open subscheme corresponding to the open subscheme of the underling ordinary scheme of $\Ef\times_\Mf\Nf$ (under the equivalence between the small Zariski sites of $\Ef\times_\Mf\Nf$ and that of its underlying ordinary scheme).
Letting $\Nf\to\Mf$ vary over the small \'etale site of $\Mf$, these local open complements glue together to form the open complement of the unit section $\Mf\to\Ef$.

\begin{definition}
We define the \emph{formal completion} $\widehat{\Ef}\in \Ab(\Shv(\Mf))$ of $\Ef$ as follows: let $U\to\tau_{\geq 0}\Ef$ be the open (relative) spectral subscheme which is the complement of the closed unit section in $\con \Ef$. For every $\Yf \to \con \Mf$ we let $\widehat{\Ef}(\Yf) \subset (\con \Ef)(\Yf)$ be the full sub-$\infty$-groupoid on those maps $\Yf \to \con \Ef$ such that the fiber product $\Yf \times_{\con \Ef} U$ is empty.
\end{definition}

\begin{warning}
We use the same notation for completion and for Pontryagin (or Picard) duality. The latter applies to a compact abelian topological group or an object in $\Orb^{\ab}$, while the former will typically be applied to a spectral elliptic curve. We trust that it should be clear in every case which of the two is meant. 
\end{warning}

\begin{remark}
The map $B\T \tensor \pt \to \Ef$ factors through the \emph{completion} $\widehat{\Ef}$ of $\Ef$.
In fact, $\widehat{\Ef}$ is the affinization of $B\T\tensor\pt$.
\end{remark}

The formal completion of an elliptic curve is an example of a \emph{formal hyperplane}. To recall this notion from \cite{LurEllII}, we have first to discuss the cospectrum of a commutative coalgebra (cf.\ \cite[Construction 1.5.4]{LurEllII}). 

\begin{definition}
For an $E_{\infty}$-ring $R$, define the $\infty$-category $\cCAlg_R$ of commutative coalgebras over $R$ as $\CAlg(\Mod_R^{\op})^{\op}$. In the case that $R$ is connective, we define the \emph{cospectrum} $\cSpec(C) \in \Fun(\CAlg_R^{\cn}, \Spc)$ of a coalgebra $C\in \cCAlg_R$ by \[\cSpec(C)(A) = \Map_{\cCAlg_A}(A, C \tensor_R A).\]

More generally for a spectral Deligne--Mumford stack and a coalgebra \[\CC \in \cCAlg_{\Mf} = \CAlg(\QCoh(\Mf)^{\op})^{\op},\] 
we define the relative cospectrum $\cSpec_{\Mf}(\CC) \in \Shv(\Mf)$ by 
\[\cSpec_{\Mf}(f\colon \Nf \to \Mf) = \Map_{\cCAlg_{\Nf}}(\OO_{\Nf}, f^*\CC).\]
\end{definition}

Given a discrete ring $R$, an important example of an $R$-coalgebra is the continuous dual of $R\llbracket t_1, \dots, t_n\rrbracket$, which we denote by $\Gamma_R(n)$ as it coincides as an $R$-module with the divided power algebra on $n$ generators. Note that its dual is again $R\llbracket t_1, \dots, t_n\rrbracket$. The cospectrum $\cSpec \Gamma_R(n)$ coincides with $\Spf R\llbracket t_1, \dots, t_n\rrbracket$ by \cite[Proposition 1.5.8]{LurEllII}.

\begin{definition}
Let $\Mf$ be a spectral Deligne--Mumford stack. An object $F \in \Shv(\Mf)$ is called a \emph{formal hyperplane} if it is of the form $\cSpec_{\Mf}(\CC)$ for some coalgebra $\CC \in \cCAlg_{\Mf}$ that is \emph{smooth}, i.e.\
\begin{itemize}
    \item locally, $\pi_0\CC$ is of the form $\Gamma_R(n)$ for some discrete ring $R$ and some $n\geq 0$, and
    \item the canonical map $\pi_0\CC \tensor_{\pi_0\OO_{\Mf}} \pi_k\OO_{\Mf} \to \pi_k\CC$ is an isomorphism for all $k\in\Z$.
\end{itemize}
\end{definition}

\begin{lemma}\label{lem:cSpecEquivalence}
   For every spectral Deligne--Mumford stack, the cospectrum functor defines an equivalence between the $\infty$-category of smooth coalgebras on $\Mf$ and formal hyperplanes on $\Mf$.
\end{lemma}
\begin{proof}
     The claim is equivalent to $\cSpec_{\Mf}$ being fully faithful on smooth coalgebras. One easily reduces to the case $\Mf = \Spec A$ for a connective $E_{\infty}$-ring $A$. As $\Mf$ is locally noetherian by assumption, we can further assume that $\pi_0A$ is noetherian and $\pi_iA$ is finitely generated over $\pi_0A$. 
     
     \cite[Proposition 1.5.9]{LurEllII} proves that $\cSpec$ is fully faithful on smooth $A$-coalgebras as a functor into $\Fun(\CAlg^{\cn}, \Spc)$. In contrast, we need to show fully faithfulness as a functor into $\Fun(\CAlg^{\cn,\afp}_A, \Spc)$, where the $\afp$ stands for almost of finite presentation. Tracing through the proof in \cite{LurEllII}, we need to show that $\Spf$ is fully faithful as a functor into $\Fun(\CAlg^{\cn,\afp}_A, \Spc)$ on adic $E_{\infty}$-rings that arise as duals of smooth coalgebras.
     
     Let $R$ and $S$ be adic $E_{\infty}$-rings. By \cite[Lemma 8.1.2.2]{SAG} we can find a tower 
     \[\cdots \to R_3 \to R_2 \to R_1\]
     in $\CAlg_R^{\cn}$ such that $\colim_i \Spec R_i \to \Spec R$  factors over an equivalence $\colim_i \Spec R_i \simeq \Spf R$ in $\Fun(\CAlg^{\cn}, \Spc)$ and every $R_i$ is almost perfect as a $R$-module. We claim that $R_i$ is almost of finite presentation over $A$ if $R$ is the dual of a smooth coalgebra $C$ over $A$. By \cite[Proposition 7.2.4.31]{HA}, this is equivalent to $\pi_0R_i$ being a finitely generated $\pi_0A$-algebra and $\pi_kR_i$ being a finitely generated $\pi_0R_i$-module for all $k$. By the definition of smooth coalgebras, $\pi_0R$ is of the form $(\pi_0A)\llbracket t_1, \dots, t_n\rrbracket$ for some $n$ and $\pi_kR \cong \pi_kA \tensor_{\pi_0A}\pi_0R$. By \cite[Proposition 7.2.4.17]{HA}, $\pi_kR_i$ is finitely generated over $\pi_0R$. The claim follows as $\Spec R_i \to \Spec R$ factors over $\Spf R$ by construction and thus a power of the ideal $(t_1, \dots, t_n)$ is zero in $R_i$. 
     
     Now let $R$ and $S$ be duals of smooth coalgebras. In particular, we see that $R$ is complete so that $R = \lim_i R_i$ in adic $E_{\infty}$-rings (cf.\ \cite[Lemma 8.1.2.3]{SAG}). 
     \begin{align*}
         \Map_{\CAlg^{\cn, \ad}_A}(S, R) &\simeq \lim_i \Map_{\CAlg^{\cn, \ad}_A}(S, R_i) \\
         &\simeq \lim_i (\Spf S)(R_i) \\
         &\simeq \lim_i \Map_{\Fun(\CAlg^{\cn, \afp}_A, \Spc)}(\Spec R_i, \Spf S) \\
         &\simeq \Map_{\Fun(\CAlg^{\cn, \afp}_A, \Spc)}(\Spf R, \Spf S)
     \end{align*}
     In the third step we have used the Yoneda lemma. 
\end{proof}

\begin{example}
 For every elliptic curve $\Ef$ over a non-connective spectral Deligne--Mumford stack $\Mf$, the formal completion $\widehat{\Ef}$ is a formal hyperplane over $\con \Mf$ \cite[Proposition 7.1.2]{LurEllII}.
\end{example}

We recall the following definition from \cite{LurEllII}:
\begin{definition}
We call an $E_{\infty}$-ring $R$ \emph{complex periodic} if it is complex orientable and Zariski locally there exists a unit in $\pi_2R$. 
For a complex periodic $E_{\infty}$-ring $R$, the \emph{Quillen formal group} $\GQ_R$ is defined as $\cSpec(\con (R \tensor B\T))$. More generally, for a locally complex periodic non-connective spectral Deligne--Mumford stack $\Mf$, the Quillen formal group $\GQ_{\Mf}$ is defined as $\cSpec_{\Mf}(\con (\OO_{\Mf} \tensor B\T))$. 
\end{definition}
By \cite[Theorem 4.1.11]{LurEllII} the Quillen formal group is a formal hyperplane. Moreover the functor 
\[\TT_{\Ab}^{\op} \to \text{Formal Hyperplanes}, \qquad M \mapsto \cSpec_{\Mf}(\con (\OO_{\Mf} \tensor B\Dual{M})) \]
equips it with the structure of an abelian group object (cf.\ \cite[Construction 4.1.13]{LurEllII}). 

Recall that a preorientation of an elliptic curve gives us a morphism $B\T \tensor \pt \to \widehat{\Ef}$, i.e.\ a $\T$-equivariant morphism $\pt \to \widehat{\Ef}$. As both $\pt$ and $\widehat{\Ef}$ are formal hyperplanes on $\con \Mf$, this is by \cref{lem:cSpecEquivalence} obtained from a $\T$-equivariant morphism $\OO_{\Mf} \to \CC$ of commutative coalgebras on $\con \Mf$, where $\widehat{\Ef} \simeq \cSpec \CC$. As $\OO_{\Mf} \tensor B\T \simeq (\OO_{\Mf})_{h\T}$, we obtain a morphism $\OO_{\Mf} \tensor B\T \to \CC$ and hence a morphism $\GQ_{\Mf} \to \widehat{\Ef}$
in the case that $\Mf$ is locally complex periodic. 
\begin{definition}
Let $\Mf$ be locally complex periodic and $p\colon\Ef\to\Mf$ a preoriented elliptic curve over $\Mf$.
We say that $\Ef$  {\em oriented} if the morphism $\GQ_{\Mf} \to \widehat{\Ef}$ is an equivalence.  
\end{definition}
Having such an orientation will force a version of the Atiyah--Segal completion theorem to hold in equivariant elliptic cohomology.

\section{Equivariant elliptic cohomology}\label{sec:EquivariantElliptic}
We are now ready to give a definition of the equivariant elliptic cohomology theory associated to a preoriented elliptic curve $p\colon \Ef\to \Mf$, which we fix throughout the section. Here, $\Mf$ denotes a non-connective spectral Deligne--Mumford stack.
\begin{construction}
To the preoriented elliptic curve $p\colon \Ef \to \Mf$ we can associate the $\infty$-topos $\XX=\Shv(\Mf)$ of space-valued sheaves on the big \'etale site of $\Mf$. As $\Ef$ defines a preoriented abelian group object in $\XX$, \cref{MainConstruction} yields a colimit-preserving functor
\[\Ell\colon \Spc_{\Orb} \to \XX,\]
where we leave the dependency on $\Ef$ implicit. 
\end{construction}
\begin{example}
We have $\Ell(\B\T) \simeq \Ef$ and $\Ell(\B C_n) \simeq \Ef[n]$, the $n$-torsion in the elliptic curve. More generally, we have $\Ell(\B G)\simeq \Hom(\widehat{G}, \Ef)$ for any compact abelian Lie group $G$. Indeed: By construction, $\Ell(\B G)$ is the internal mapping object in $\PreAb(\Shv(\Mf))$ between $\Mf \tensor \widehat{\B G}$ and $\Ef$. Here, $\widehat{\B G}$ is the shifted Pontryagin dual from \cref{con:PicardDuality}, which is equivalent to $\widehat{G} \times B\T$, i.e.\ the left adjoint of $\Ab(\Spc) \to \PreAb(\Spc)$ applied to $\widehat{G}$. It follows that $\Ell(\B G)$ is equivalent to $\Hom(\widehat{G}, \Ef)$, i.e.\ the internal mapping object in $\Ab(\Sh(\Mf))$ between $\Mf \tensor \widehat{G}$ and $\Ef$. See \cref{prop:computation} and \cref{rem:PresentableCase}.
\end{example}
We would like to specialize to a theory for $G$-spaces for a fixed $G$ by a pushforward construction. The key will be the following algebro-geometric observation. 
\begin{prop}\label{prop:affine}
    Let $H\subset G$ be an inclusion of compact abelian Lie groups. Then the induced morphism $f\colon \Ell(\B H) \to \Ell(\B G)$ is affine. 
\end{prop}
\begin{proof}
    We claim first that for a surjection $h\colon A \to B$ of finitely generated abelian groups and a classical elliptic curve $E$ over a base scheme $S$, the map $\Hom(B, E) \to \Hom(A, E)$ is a closed immersion and hence affine. As the kernel of $h$ is a sum of cyclic groups and a composition of closed immersions is a closed immersion, we can assume that $\ker(h)$ is cyclic. The pushout square
    \[
    \xymatrix{B & \ar[l] A \\
    0 \ar[u] & \ker(h) \ar[u]\ar[l] }
    \]
    induces a pullback square
    \[
    \xymatrix{\Hom(B, E) \ar[d]\ar[r] & \Hom(A, E)\ar[d] \\
    S \ar[r] & \Hom(\ker(h), E).
    }
    \]
    Furthermore, $S \to \Hom(\ker(h), E)$ is a closed immersion as it is a right inverse of the structure morphism $\Hom(\ker(h), E) \to S$ and the latter is a separated morphism. (See e.g.\ \cite[Lemma 2.4]{M-OMell} for this criterion for closed immersions.) Thus, its pullback $\Hom(B,E) \to \Hom(A,E)$ is also a closed immersion and thus affine. 

    Recall that $\Ell(\B G)\simeq\Hom(\widehat{G}, E)$, with $\widehat{G}$ denoting the Pontryagin dual, and that an inclusion $H\subset G$ induces a surjection $\widehat{G} \to \widehat{H}$. Thus, we know from the preceding paragraph that the underlying map of $\Ell(\B H) \to \Ell(\B G)$ is affine. We can moreover assume that the base of the elliptic curve $\Ef$ is an affine derived scheme $\Spec R$. By \cite[Remark 1.5.3]{LurEllI}, the elliptic curve $\Ef$ is based changed from $\Spec \tau_{\geq 0}R$. Thus we can assume additionally that $R$ is connective and hence that $\Ell(\B G)$ and $\Ell(\B H)$ are spectral schemes (and not more general non-connective ones). 
	
	Let $\Spec A \to \Ell(\B G)$ be a map from an affine and let $(\XX, \OO_{\XX})$ denote the pullback $\Spec A \times_{\Ell(\B G)} \Ell(\B H)$. As a pullback of spectral schemes, $(\XX, \OO_{\XX})$ is a spectral scheme again, and the underlying scheme of $(\XX, \OO_{\XX})$ is the pullback of the underlying scheme $(\XX,\pi_0\OO_{\XX})$, hence affine. Using \cite[Corollary 1.1.6.3]{SAG}, we conclude that $(\XX,\OO_{\XX})$ is affine.
\end{proof}

\begin{construction}Let $G$ be a compact abelian Lie group and $\Orb_G$ its orbit $\i$-category, which we identify using \cref{cor:OrbBG} with the full subcategory of $\Orb_{/\B G}$ on morphisms $\B H \to  \B G$ inducing an injection $H\to G$. We consider the functor
\[\Orb_G \to \QCoh(\Ell(\B G))^{\op}, \qquad G/H \mapsto f_*\OO_{\Ell(\B H)}, \]
where $f\colon \Ell(\B H) \to \Ell(\B G)$ is the map induced by the inclusion $H\subset G$. As $f$ is affine by the preceding proposition, $f_*$ sends quasi-coherent sheaves to quasi-coherent sheaves by \cite[Proposition 2.5.11]{SAG}. The functor from $\Orb_G$ extends to a colimit-preserving functor
\[\mathcal{E}ll_G\colon \SS^G \to \QCoh(\Ell(\B G))^{\op},\]
where $\SS^G \simeq \Fun(\Orb_G^{\op}, \Spc)$ denotes the $\infty$-category of $G$-spaces.  
As the target is pointed, this functor factors canonically to define a reduced version
\[\widetilde{\mathcal{E}ll}_G\colon \SS^G_* \to \QCoh(\Ell(\B G))^{\op}. \]
\end{construction}
For us, the most important case is $G = \T$. As $\Ell(\B\T) = \Ef$, we obtain in this case a functor from $\T$-spaces to quasi-coherent sheaves on $\Ef$. If desired, we can take homotopy groups to obtain an equivariant cohomology theory with values in quasi-coherent sheaves on the underlying classical elliptic curve of $\Ef$. This is the kind of target we are used to from the classical constructions of equivariant elliptic cohomology, for example by Grojnowski \cite{Groj}. 

Next we want to compare our two constructions of equivariant elliptic cohomology functors. To that purpose we want to recall two notions. First, given any (non-connective) spectral Deligne--Mumford stack $\Nf$ and any $\Zf \in \Shv(\Nf)$, we consider the restriction $\OO_{\Zf}$ of $\OO_{\Nf}$ to $\Shv(\Nf)_{/\Zf}$  and define $\Mod_{\OO_{\Zf}}$ as the $\infty$-category of modules over it. Here, on any $\Yf \to \tau_{\geq 0}\Nf$ in the big \'etale site of $\Nf$, with underlying morphism $g\colon \YY \to \NN$ of $\infty$-topoi, $\OO_{\Nf}$ is defined  as $(\OO_{\Yf}\tensor_{g^{-1}\OO_{\tau_{\geq 0}\Nf}}g^{-1}\OO_{\Nf})(\Yf)$; since on the small \'etale site of $\Nf$, the sheaf $\OO_{\Nf}$ just defined agrees with the structure sheaf of $\Nf$, we use the same symbol. Second, for any $\FF \in \Mod_{\OO_{\Zf}}$, we can consider the $\OO_{\Nf}$-module $f_*\FF$, associated with $f\colon \Zf \to \Nf$ and defined by 
\[f_*\FF(U) = \FF(U\times_{\Nf} \Zf), \]
where $\Nf = \pt$ is the final object in $\Shv(\Nf)$. 

\begin{lemma}\label{lem:pushforward}
    Assume that $\Mf$ is locally $2$-periodic and that there is an $n$ such that $\Mf(\Spec R)$ is $n$-truncated for every classical ring $R$. 
    
    For $G$ a compact abelian Lie group and $X \in \Spc^G$, there is a natural equivalence \[\mathcal{E}ll_G(X) \simeq \QQ(f_*\OO_{\Ell([X/G])}),\] where $f\colon \Ell([X/G]) \to \Ell(\B G)$ denotes the map induced by $X \to \pt$ and $\QQ$ is the right adjoint to the inclusion 
    $\QCoh(\Ell(\B G)) \subset \Mod_{\OO_{\Ell(\B G)}}$.\footnote{This right adjoint exists because $\QCoh(\Ell(\B G))$ is presentable and the inclusion preserves all colimits by \cite[Proposition 2.2.4.1]{SAG}. Strictly speaking, Lurie uses here a different definition of $\Mod_{\OO}$, namely just modules in sheaves of spectra on the small \'etale topos. But as explained at the end of \cref{app:SAG}, pullback defines a functor from this to our version of $\Mod_{\OO}$, which preserves all colimits.}
    If $X$ is finite, $f_*\OO_{\Ell([X/G])}$ is already quasi-coherent so that $\QQ(f_*\OO_{\Ell([X/G])})$ agrees with $f_*\OO_{\Ell([X/G])}$. 
\end{lemma}
\begin{proof}
    Throughout this proof we will abbreviate $\Ell(\B G)$ to $\Zf$. The key is to check that the functor 
    \[\PP\colon \Spc^G \to \Mod_{\OO_{\Zf}}^{\op},\qquad X \mapsto f_*\OO_{\Ell([X/G])}\]
    preserves colimits. This suffices as $\QQ$ preserves limits and hence the composite
    \[\QQ\PP\colon \Spc^G  \to \QCoh(\Zf)^{\op},\qquad X \mapsto \QQ(f_*\OO_{\Ell([X/G])})\]
    preserves colimits as well. As this functor agrees on orbits with $\mathcal{E}ll_G$, it agrees thus on all of $\Spc^G$. Moreover, $\QCoh(\Zf)$ is closed under finite limits, implying the last statement.

    To show that $\PP$ preserves colimits we reinterpret the pushforward using  $\A^1$, i.e.\ the spectrum of the free $E_{\infty}$-ring on one generator. For any map $h\colon \Yf \to \Zf$ with $\Yf \in \Shv(\Mf)$, there is a natural equivalence of $\Omega^{\infty}h_*\OO_{\Yf}$ with the sheaf 
    \[(U\to \Zf) \quad\mapsto\quad \Omega^{\infty}\Gamma(\OO_{\Yf\times_{\Zf}U}) \simeq \Map_U(\Yf\times_{\Zf} U, U\times \A^1).\]
    This implies that the functor from $\Shv(\Mf)_{/\Zf}$ to $\Shv(\Zf)^{\op}$, sending $h\colon \Yf \to \Zf$ to $\Omega^{\infty}h_*\OO_{\Yf}$, preserves all colimits. 
    
    We claim that $\Omega^{\infty}\colon \Mod_{\OO_{\Zf}} \to \Shv(\Zf)$ is conservative. By assumption $\Mf$ is locally of the form $\Spec R$ for a $2$-periodic $E_{\infty}$-ring $R$. As $\Zf = \Ell(\B G)$ maps to $\Mf$, the same is true for $\Zf$. Since the conservativity of $\Omega^{\infty}$ can be checked locally, we just have to show that $\Omega^{\infty}\colon \Mod_R \to \Spc$ is conservative if $R$ is $2$-periodic and this is obvious. 
    
    As $\Omega^{\infty}\colon \Mod_{\OO_{\Zf}} \to \Shv(\Zf)$ also preserves limits, we see that the functor 
    \[\Shv(\Mf)_{/\Zf} \to  \Mod_{\OO_{\Zf}}^{\op}, \qquad (h\colon \Yf \to \Zf)\mapsto h_*\OO_{\Yf}\]
    preserves colimits. Moreover, since $\Ell\colon \Spc_{\Orb} \to \Shv(\Mf)$ preserves colimits, the same is true for $(\Spc_{\Orb})_{/\B G} \to \Shv(\Mf)_{/\Ell(\B G)}$. Applying \cref{prop:embeddingG} finishes the proof. 
\end{proof}

Last we want to speak about the functoriality of equivariant elliptic cohomology in the elliptic curve. Thus denote for the moment the elliptic cohomology functors based on $\Ef$ by $\Ell^{\Ef}$ and $\mathcal{E}ll_G^{\Ef}$. 

\begin{proposition}\label{prop:functoriality}
Let $f\colon \Nf \to \Mf$ be a morphism of non-connective spectral Deligne--Mumford stacks that is almost of finite presentation, $\Ef$ be a preoriented elliptic curve over $\Mf$, and $f^*\Ef$ the pullback of $\Ef$ to $\Nf$. Then there are natural equivalences
\begin{align*}
f^*\Ell^{\Ef} \simeq \Ell^{f^*\Ef} &\text{ in } \Fun(\Spc_{\Orb}, \Shv(\Nf))\text{ and }\\
f^*\mathcal{E}ll_G^{\Ef} \simeq \mathcal{E}ll_G^{f^*\Ef} &\text { in }\Fun((\Spc^G)^{\fin, \op}, \QCoh(f^*\Ell^{\Ef}(\B G)))
\end{align*}
for all compact abelian Lie groups $G$. 
\end{proposition}
\begin{proof}
Write $f_*\colon \Shv(\Nf) \to \Shv(\Mf)$ for the induced functor 
on big \'etale topoi, which admits a left adjoint $f^*\colon\Shv(\Mf)\to\Shv(\Nf)$ preserving finite limits.
The preoriented abelian group object $\Ef$ induces a left adjoint functor $\Ell^{\Ef}\colon\Spc_{\Orb}\to\Shv(\Mf)$, which we can postcompose with $f^*$ to obtain a functor $f^*\Ell^{\Ef}\colon\Spc_{\Orb}\to\Shv(\Mf)\to\Shv(\Nf)$.

    By \cref{MainConstruction}, $f^*\Ell^{\Ef}$ is the functor $\Spc_{\Orb} \to \Shv(\Nf)$ associated with $f^*\Ef$. 
    
    The functor $f^*\mathcal{E}ll_G^{\Ef}\colon (\Spc^G)^{\fin, \op} \to \QCoh(f^*\Ell^{\Ef}(\B G))$ preserves all finite limits. Thus, we only have to provide natural equivalences $f^*\mathcal{E}ll_G^{\Ef}(G/H) \simeq \mathcal{E}ll^{f^*\Ef}_G(G/H)$ for all closed subgroups $H\subset G$. These are provided by applying the following commutative square to the structure sheaf of $\Ell^{\Ef}(\B H)$:
    \[
    \xymatrix{
    \QCoh(\Ell^{\Ef}(\B H)) \ar[rr]^{\text{pullback}} \ar[d]^{\text{pushforward}} &&  \QCoh(
    \Ell^{f^*\Ef}(\B H)) \ar[d]^{\text{pushforward}} \\
      \QCoh(\Ell^{\Ef}(\B G)) \ar[rr]^{\text{pullback}} && \QCoh(\Ell^{f^*\Ef}(\B G)) 
    }
    \]
    This commutative square in turn is associated by \cite[Proposition 2.5.4.5]{SAG} to the pullback diagram
    \[
    \xymatrix{
    \Ell^{f^*\Ef}(\B H) \ar[r] \ar[d] & \Ell^{\Ef}(\B H) \ar[d] \\
    \Ell^{f^*\Ef}(\B G) \ar[r] & \Ell^{\Ef}(\B G)
    }
    \]
    that we obtain from the first claim together with the fact that the sheaf in $\Shv(\Nf)$ represented by the pullback $\Nf\times_{\Mf}\Ell^{\Ef}(\B H)$ agrees with the pullback sheaf $f^*\Ell^{\Ef}(\B H)$ and similarly for $G$. 
\end{proof}

\begin{remark}
We do not claim that the functor $\Ell$ defines the ``correct'' version of elliptic cohomology for arbitrary orbispaces. For a general $X\in \Spc_{\Orb}$ there is for example no reason to believe that $\Ell(X)$ is a nonconnective spectral Deligne--Mumford stack, not even a formal one. It should be thus seen more as a starting point to obtain a reasonable geometric object. For example, given an abelian compact Lie group $G$ and a $G$-space $Y$, we have seen how to recover $\mathcal{E}ll_G(Y)$ from $\Ell([Y/G])$. Taking $\Spec$ (if $Y$ is finite) or a suitable version of $\Spf$ (if $Y$ is infinite) of $\mathcal{E}ll_G(Y)$ seems to be a reasonable guess for the ``correct'' geometric replacement of $\Ell([Y/G])$ in these cases. We will return to this point in a sequel to this paper \cite{GepnerMeierII}.
\end{remark}

\section{Eilenberg--Moore type statements}\label{app:EM}
The goal of this section is to recall and extend results of Eilenberg--Moore, Dwyer and Bousfield about the homology of fiber squares and to rephrase parts of them in terms of cospectra. Applied to classifying spaces of abelian compact Lie groups, this will be a key step to proving symmetric monoidality properties of equivariant elliptic cohomology. 

Throughout this section, we will consider a (homotopy) pullback diagram 
\[
\xymatrix{
M \ar[r]\ar[d] & X \ar[d] \\
Y \ar[r] & B} \]
of spaces. We assume that for every $b\in B$, the fundamental group $\pi_1(B,b)$ acts nilpotently on the integral homology of the homotopy fiber $Y_b$.

We denote by $\CC(X,B,Y)$ the cobar construction, i.e.\ the associated cosimplicial object whose $n$-th level is $X \times B^{\times n} \times Y$. This cosimplicial object is augmented by $M$. 

\begin{theorem}[Bousfield, Dwyer, Eilenberg--Moore]
	 Under the conditions above, the augmentation of the cobar construction induces a pro-isomorphism between the constant tower $H_*(M;\Z)$ and the tower $(H_*(\Tot_m \CC(X, B, Y));\Z)_{m\geq 0}$.
\end{theorem}
\begin{proof}
	This follows from Lemmas 2.2 and 2.3 and Section 4.1 in \cite{BouCosimplicial}. 
\end{proof} 
The following corollary is essentially also already contained in \cite{BouCosimplicial}.
\begin{cor}\label{BDEM2}
	 The augmentation of the cobar construction induces a pro-isomorphism between the constant tower $E_*(M)$ and the tower $(E_*(\Tot_m \CC(X, B, Y)))_{m\geq 0}$ for any bounded below spectrum $E$. 
	 In particular, $\lim^1_m E_*(\Tot_m \CC(X, B, Y)) = 0$, and the natural map 
	 \[E_*(M) \to \lim_m E_*(\Tot_m \CC(X, B, Y)) \]
	 is an isomorphism. 
\end{cor}
\begin{proof}
In the last theorem, we can replace $\Z$ by arbitrary direct sums of $\Z$ and get our claim for $E$ any shift of the associated Eilenberg--MacLane spectrum. By the five lemma in pro-abelian groups, we obtain the statement of this corollary for all truncated spectra $E$. The group $E_k(Z)$ coincides with $(\tau_{\leq k}E)_k(Z)$ for every space $Z$. This implies the result.
\end{proof}

We would like to pass from an isomorphism of pro-groups to an equivalence of spectra. In the following, let $R$ an arbitrary $E_{\infty}$-ring spectrum. Note that $\Tot_m$ coincides with the finite limit $\lim_{\Delta_{\leq m}}$ and it thus commutes with smash products. We deduce that the $R$-homology of the $\Tot_m$ of the cobar construction coincides with the homotopy groups of the spectrum
\[
\Tot_m R \tensor (X \times B^{\times \bullet} \times Y)  \simeq \Tot_m ((R\tensor X) \tensor_R (R\tensor B)^{\tensor_{R} \bullet} \tensor_{R} (R \tensor Y)).
\]
Note that as before we view all spaces here as unpointed so that $R \tensor X$ has the same meaning as $R \tensor \Sigma^{\infty}_+X$.

We recall the cotensor product: Given morphisms $C' \to C$ and $C''\to C$ in  $\cCAlg_R$, the cotensor product $C'\square_C C''$ is the fiber product in $\cCAlg_R$. As $\cCAlg_R = \CAlg(\Mod_R^{\op})^{\op}$, the usual formula for the pushout of commutative algebras, i.e.\ the relative tensor product, translates into the formula
\[C'\square_C C'' \simeq \lim_{\Delta} C' \tensor_R C^{\tensor_R \bullet} \tensor_R C'',\]
which we can also use to define the cotensor products for arbitrary comodules. 
Thus, $\lim_{\Delta}\CC(X,B,Y) \simeq (R\tensor X) \square_{R\tensor B} (R\tensor Y)$. In this language, \cref{BDEM2} implies:
\begin{cor}\label{cor:cobar}
	If $R$ is bounded below, the augmentation map 
	\[R\tensor M \to (R\tensor X) \square_{R \tensor B} (R \tensor Y)\]
	is an equivalence.
\end{cor}

\begin{cor}
    	If $R$ is connective and $R \tensor X$ is in the thick subcategory of $R\tensor B$ in $R\tensor B$-comodules, then the map 
    \[\cSpec (R\tensor M) \to \cSpec(R\tensor X) \times_{\cSpec(R\tensor B)} \cSpec(R\tensor Y)\]
    is an equivalence.
\end{cor}
\begin{proof}
    Given $A\in \CAlg_R^{\cn}$, the functor 
    \[\mathrm{coMod}(R\tensor B) \to \mathrm{coMod}(A\tensor B), \qquad N \mapsto A\tensor_R N \]
    preserves colimits and is in particular exact. Thus, 
    \[(R \tensor X)\square_{R\tensor B}(R\tensor Y)\tensor_R A \simeq (A \tensor X) \square_{A\tensor B}(A\tensor Y),\]
    using that $R\tensor X$ is in the thick subcategory of $R\tensor B$.
    It follows that
    \begin{align*}\cSpec((R \tensor X)\square_{R\tensor B}(R\tensor Y))(A) &\simeq \Map_{\cCAlg_A}(A, (A \tensor X) \square_{A\tensor B}(A\tensor Y)) \\
     &\simeq (\cSpec(R\tensor X)\times_{\cSpec(R\tensor B)}\cSpec(R\tensor Y))(A). \qedhere
    \end{align*}
\end{proof}

\begin{example}\label{ex:EM}
Let $R$ be the connective cover of a complex oriented and $2$-periodic $E_{\infty}$-ring, with chosen isomorphism $\pi_0R^{B\T} \cong (\pi_0R)\llbracket t\rrbracket$. The short exact sequence
\[0 \to (\pi_0R)\llbracket t\rrbracket \xrightarrow{t} (\pi_0R)\llbracket t\rrbracket \to \pi_0R \to 0\]
of topological $\pi_0R$-modules dualizes to a short exact sequence of comodules
\[0 \to \pi_0R \to \Gamma_{\pi_0R}(1) \to  \Gamma_{\pi_0R}(1) \to 0.\]
As for any $R$-free $(R \tensor B\T)$-comodule $C$ the map 
\[\pi_0\Map_{\mathrm{coMod}(R\tensor B\T)}(C, R\tensor B\T) \to \Map_{\Gamma_{\pi_0R}(1)}(\pi_0C, \Gamma_{\pi_0R}(1))
\]
is an isomorphism, we obtain a cofiber sequence
\[R \to R \tensor B\T \to R \tensor B\T \to \Sigma R \]
of $(R\tensor B\T)$-comodules. This implies that $R$ is in the thick subcategory of $R\tensor B\T$ in $(R\tensor B\T)\comodules$. Using the last corollaries, we obtain equivalences
\begin{align}\label{cSpecBCn}\nonumber R\tensor BC_n &\xrightarrow{\simeq} R\tensor B\T \,\square_{R\tensor B\T}\, R \\
\cSpec(R\tensor BC_n) &\xrightarrow{\simeq} \cSpec(R \tensor B\T) \times_{\cSpec(R\tensor B\T)} \Spec R,\end{align}
where the map $R\tensor B\T \to R\tensor B\T$ is multiplication by $n$ on $B\T$. The first equivalence shows in particular that $R\tensor BC_n$ is also in the thick subcategory of $R\tensor B\T$. Thus we obtain from the last corollary an equivalence
\[\cSpec(R \tensor (BC_n\times_{B\T} BC_m)) \xrightarrow{\simeq} \cSpec(R\tensor BC_n) \times_{\cSpec R\tensor B\T} \cSpec(R\tensor BC_m).\]
Note that \eqref{cSpecBCn} also implies that $\cSpec(R\tensor BC_n) \simeq \Hom(C_n, \cSpec(R\tensor B\T))$.  
\end{example}

\section{Torus-equivariant elliptic cohomology is symmetric monoidal}\label{sec:EllCohSymMon}
Given an oriented elliptic curve $\Ef$ over a (locally complex periodic) non-connective spectral Deligne--Mumford stack $\Mf$, our aim is to prove the following theorem. 
\begin{theorem}\label{thm:monodality}
If $X$ and $Y$ are finite $\T$-spaces, the natural map 
	\[\mathcal{E}ll_{\T}(X) \tensor_{\OO_{\Ef}} \mathcal{E}ll_{\T}(Y) \to \mathcal{E}ll_{\T}(X\times Y) \]
	is an equivalence. 
\end{theorem}
We first note that it suffices to prove the claim in the case where $X$ and $Y$ are $\T$-orbits, as $\EllT$ sends finite colimits to finite limits, and these commute with the tensor product. Moreover, by the base change property \cref{prop:functoriality} (in conjunction with \cref{ex:etale}) we can always assume that $\Mf \simeq \Spec R$, where $R$ is a complex-orientable and $2$-periodic $E_{\infty}$-ring.

As a first step, we will reformulate our claim into a statement about affine morphisms using the following lemma. 
 \begin{lemma}\label{lem:relativespec}
	     Let $S$ be a non-connective spectral Deligne--Mumford stack and $\Aff_S$ the $\infty$-category of affine morphisms $U \to S$. Then the functor
	     \[\Aff_S^{\op} \to \CAlg(\QCoh(X)), \qquad (U\xrightarrow{f}S) \mapsto f_*\OO_U \]
	     is an equivalence. We denote the inverse by $\Spec_{S}$.
	 \end{lemma}
 \begin{proof}
 	The proof is analogous to \cite[Proposition 2.5.1.2]{SAG}.
 \end{proof}
 Given a representable morphism $X \to \B\T$ (meaning that $X\times_{\B\T}\pt$ is a space), we define $\Ell_{\T}(X)$ as $\Spec_{\EE}\EllT(X \times_{\B\T} \pt)$. If $X$ is the image of a finite space along the embedding $\Spc \to \Spc_{\Orb}$, we observe that we have an equivalence
 \[
 \Ell_{\T}(X) \simeq \Spec R^X
 \]
 because the unit section $\Spec R \to \Ef$ is affine.
 Noting that $\Ell_{\T}(\B C_m) = \Ef[m]$ and $\Ell_{\T}(\B\T) = \Ef$, our claim reduces to the following lemma.
 \begin{lemma}
     The canonical map 
      \begin{equation}\label{SymToShow}\Ell_{\T}(\B C_m\times_{\B\T} \B C_n) \to \Ef[m]\times_{\Ef} \Ef[n]. \end{equation}
      is an equivalence for all $m,n \geq 1$.
 \end{lemma}
 \begin{proof}
     We first assume that $m$ and $n$ are relatively prime. One computes 
     \[\B C_m\times_{\B\T} \B C_n \simeq \T/C_{mn} \cong \T.\]
     Thus the source in \eqref{SymToShow} is equivalent to $\Spec R^{\T}$. We will construct next an equivalence
     \begin{equation}\label{ConnectiveEquivalence}\Spec (\con R)^{\T} \to (\con \Ef)[m] \times_{\con\Ef} (\con \Ef)[n]; \end{equation}
     base changing along $\Spec R \to \Spec (\con R)$ shows that \eqref{SymToShow} is an equivalence as well.
     
     Given a morphism $f\colon X \to \con \Ef$ from a spectral scheme, the map $X \to X\times_{\con \Ef} \widehat{\Ef}$ is an equivalence if the image of $f$ is contained in the image of the unit section; this follows from $\widehat{\Ef}\to \con \Ef$ being a monomorphism. As we can observe on underlying schemes, the image of  $(\con \Ef)[m]\times_{\con \Ef} (\con\Ef)[n] \to \con \Ef$ is contained in the unit section. Thus, 
     \[(\con \Ef)[m]\times_{\con \Ef} (\con \Ef)[n] \simeq (\con \Ef)[m]\times_{\con \Ef} (\con \Ef)[n]\times_{\con\Ef} \widehat{\Ef} \simeq \widehat{\Ef}[m] \times_{\widehat{\Ef}} \widehat{\Ef}[n]. \]
     As the orientation provides an equivalence $\widehat{\Ef} \simeq \cSpec(\con R \tensor B\T)$, we can further identify this fiber product with 
     \begin{align*}\cSpec(\con R \tensor BC_m) \times_{\cSpec (\con R \tensor B\T)} \cSpec(\con R \tensor BC_n) &\simeq \cSpec(\con R \tensor \B C_m\times_{\B\T} \B C_n) \\
     &\simeq \cSpec (\con R \tensor \T)\end{align*}
     using \cref{ex:EM}. The computation 
     \begin{align*}
         \cSpec(\con R \tensor \T)(A) &= \cCAlg_A(A, A \tensor \T) \\
         &\simeq \CAlg_A(A^{\T}, A) \\
         &\simeq \CAlg_{\con R}((\con R)^{\T}, A),
     \end{align*}
     for $A \in \CAlg^{\cn}_{\con R}$ shows that $\cSpec(\con R \tensor \T) \simeq \Spec (\con R)^{\T}$. This provides the equivalence \eqref{ConnectiveEquivalence}.
     
     For $m$ and $n$ general, let $d$ be their greatest common divisor. Choose relatively prime $k$ and $l$ such that $m = kd$ and $n = ld$ and consider the diagram
      \[
      \begin{tikzcd}
      & \Ell_{\T}(\B C_m \times_{\B \T} \B C_n)\arrow[rr]\arrow[dl]\arrow[dd] & & \Ell_{\T}(\B C_k \times_{\B\T} \B C_l)\arrow[dd]\arrow[dl]  \\
      \Ef[n]\times_{\Ef} \Ef[m] \arrow[dd]\arrow[rr, crossing over] & & \Ef[k] \times_{\Ef} \Ef[l]\\
      &\Ell_{\T}(\B \T)\arrow[dl] \arrow[rr, "{[} d {]}", near start] && \Ell_{\T}(\B\T)\arrow[dl] \\
     \Ef \arrow[rr, "{[} d {]}"] && \Ef
     \ar[from=2-3, to=4-3, crossing over]
     \end{tikzcd}
      \]
     The front square is easily seen to be a fiber square. Concerning the back square, we have $\B C_k \times_{\B \T} \B C_l \simeq \T/C_{kl} \cong \T$. Moreover, the resulting map $\T \to \B \T$ has to factor through the point as there are no non-trivial $\T$-principal bundles on $\T$. Thus, the back square decomposes into a rectangle:
     \[
     \xymatrix{
     \Ell_{\T}(\T \times \B C_d) \ar[r]\ar[d]& \Ell_{\T}(\T) \ar[d] \\
     \Ell_{\T}(\B C_d)\ar[d] \ar[r] & \Ell_{\T}(\pt)\ar[d]\\
     \Ell_{\T}(\B \T) \ar[r]^{[d]} & \Ell_{\T}(\B \T)
     }
     \]
     The lower square is cartesian by definition and for the cartesianity of the upper square one just has to observe that $\Ell_{\T}(\T \times \B C_d) \simeq \Spec \Ell_{\T}(\B C_d)^{\T}$ and $\Ell_{\T}(\T) \simeq \Spec R^{\T}$; thus the back square is cartesian. Now it remains to observe that three of the diagonal arrows are equivalences, either by definition or the above, and hence the arrow
     \[\Ell_{\T}(\B C_m \times_{\B \T} \B C_n) \to \Ef[n] \times_{\Ef} \Ef[m] \]
     is an equivalence as well.
 \end{proof}

\section{Representablity by genuine equivariant spectra}\label{sec:ellipticspectra}
The goal of this section is to connect our treatment of equivariant elliptic cohomology to stable equivariant homotopy theory. 
We let $\Ef$ denote an oriented spectral elliptic curve over a non-connective spectral Deligne--Mumford stack $\Mf$. For the next lemma, we fix the following notation: We denote by $\rho$ the tautological complex representation of $\T = U(1)$ and by $e_n\colon \Ef[n] \hookrightarrow \Ef$ the inclusion of the $n$-torsion. 
\begin{lemma}\label{lem:invertible}
	Applying $\wEllT$ to the cofiber sequence $(\T/\T[n])_+ \to S^0 \to S^{\rho^{\tensor n}}$ results in a cofiber sequence
	\begin{equation}\label{eq:divisor}(e_n)_*\OO_{\Ef[n]} \leftarrow \OO_{\Ef} \leftarrow \OO_{\Ef}(-e_n).\end{equation}
	The $\OO_{\Ef}$-module $\OO_{\Ef}(-e_n)$ is invertible. 
\end{lemma}
\begin{proof}
   As $e_n$ is affine, we can compute $\pi_*(e_n)_*\OO_{\Ef[n]}$ as $(e_n)_*\pi_*\OO_{\Ef[n]}$.\footnote{A reference is \cite[Lemma 2.8]{MeierLevel}, at least if the underlying Deligne--Mumford stack of $\Ef$ is separated. The separatedness assumption is only used at the top of p.1314 of loc.\ cit.\ and can be circumvented by choosing a hypercover of $\Ef$ by disjoint unions of affines. Note that all sheaves of abelian groups satisfy hyperdescent.} Here, we abuse notation to denote by $e_n$ also the map of underlying classical stacks. As the underlying map of the multiplication map $[n]\colon \Ef \to \Ef$ is flat by \cite[Theorem 2.3.1]{K-M85}, the underlying stack of $\Ef[n]$ is precisely the $n$-torsion in the underlying elliptic curve of $\Ef$ (cf.\ \cite[Lemma B.3]{MeierLevel}). 
   
    By definition, the map $\wEllT(S^0) \to \wEllT(\T/\T[n]_+)$ agrees with the canonical map $\OO_{\Ef} \to (e_n)_*\OO_{\Ef[n]}$. By the above, the map $\pi_0\OO_{\Ef} \to (e_n)_*\pi_0\OO_{\Ef[n]}$ is surjective as the underlying map of $\Ef[n] \to \Ef$ is a closed immersion. Moreover, by \cref{prop:functoriality} we can reduce to the universal case of the universal oriented elliptic curve over $\Mf = \Mf_{\Ell}^{\mathrm{or}}$ (see \cite{LurEllII}), where both source and target of $\OO_{\Ef} \to (e_n)_*\OO_{\Ef[n]}$ are even-periodic and thus the map $\OO_{\Ef} \to (e_n)_*\OO_{\Ef}$ is surjective on homotopy groups in all degrees. In particular, its fiber $\wEllT(S^{\rho^{\tensor n}})$ is also even-periodic and its $\pi_0$ agrees with the kernel of $\pi_0\OO_{\Ef} \to (e_n)_*\pi_0\OO_{\Ef[n]}$, i.e.\ the ideal sheaf associated with the underlying map of $e_n$. As this sheaf is invertible, so is $\wEllT(S^{\rho^{\tensor n}})$.
    \end{proof}

Using the universal property of equivariant stabilization, we can deduce that $\EllT$ factors over finite $\T$-spectra. More precisely, we denote by $\Spc^{\T, \fin}$ the $\i$-category of finite $\T$-spaces (i.e.\ the closure of the orbits under finite colimits) and by $\Sp^{\T,\omega}$ the compact objects in $\T$-spectra and obtain the following statement:

\begin{proposition}
We have an essentially unique factorization
\[
\xymatrix{
	\Spc^{\T, \fin}\ar[d]^{()_+} \ar[dr]^{\EllT} \\
	\Spc_*^{\T, \fin}\ar[d]^{\Sigma^\infty} \ar[r]^-{\wEllT} & \QCoh(\Ef)^{\op} \\
	\Sp^{\T, \omega} \ar@{-->}[ur]^{\EllT}.
}
\]
\end{proposition}
\begin{proof}
    By \cref{thm:monodality} the functor $\EllT$ and hence also the functor $\wEllT$ is symmetric monoidal. Thus the universal property from \cref{prop:universalprop} applies once we have checked that $\wEllT$ sends every representation sphere to an invertible object and every finite $\T$-space to a dualizable object. The first follows from \cref{lem:invertible} as every $\T$-representation is a sum of tensor powers of $\rho$. For the second it suffices to show that $\wEllT(\T/\T[n]_+)$ is dualizable for every $n$. This follows again from \cref{lem:invertible} as it provides a cofiber sequence with $\wEllT(\T/\T[n]_+)$ and two invertible quasi-coherent sheaves. 
\end{proof}

In the following construction, we will explain how to obtain a genuine $\T$-spectrum from $\T$-equivariant elliptic cohomology and also sketch the analogous process for other compact abelian Lie groups. 

\begin{construction}\label{constr:GSpectra}
Let $\Sp_G^{\omega} \to \Sp^{\op}$ be a finite colimit preserving functor. As $\Sp_G \simeq \Ind(\Sp_G^{\omega})$, this factors over a colimit preserving functor $F\colon \Sp_G \to \Sp^{\op}$. This we can also view as a right adjoint $F\colon \Sp_G^{\op} \to \Sp$ with left adjoint $L$. This functor is representable by $R = L(\mathbb{S})$. Indeed, the functor
\[\Omega^{\infty}F(-) \simeq \Map_{\Sp}(\mathbb{S}, F(-)) \simeq \Map_{\Sp_G}(L\mathbb{S}, -)\]
is equivalent to $\Omega^{\infty}$ of the mapping spectrum $\uMap_{\Sp_G}(L\mathbb{S}, -)$ and 
  \[\Fun^{\mathrm{R}}(\Sp_G^{\op}, \Sp) \xrightarrow{\Omega^{\infty}}\Fun^{\mathrm{R}}(\Sp_G^{\op}, \Spc)\]
  is an equivalence by the proof of \cite[Corollary 1.4.4.5]{HA}.

By definition, $F(\Sigma^{\infty}G/H_+)$ agrees with the mapping spectrum $\uMap_{\Sp_G}(\Sigma^{\infty}G/H_+,R)$, i.e.\ with the fixed point spectrum $R^{\B G}$. Note here that we use the notation $R^{\B H}$ for what traditionally would usually be denoted $R^H$, the reason for which is two-fold: First, it fits well with our philosophy that the fixed points should really be associated with the stack $\B H$ rather than the group $H$ (and could be viewed as the mapping spectrum from $\B H$ to $R$; cf.\ \cite[Theorem 4.4.3]{SchGlobal}). Second, it avoids possible confusion between the $H$-fixed points of $R$ and the cotensor $R^H$, where $H$ is viewed as a topological space. 

In our case, $G$ will be $\T$ and $F$ the composition of $\EllT\colon \Sp_{\T}^{\omega} \to \QCoh(\Ef)^{\op}$ with the global sections functor $\Gamma\colon \QCoh(\Ef) \to \Sp$ and we obtain a representing $\T$-spectrum $R$. 
By construction $R^{\B H}$ agrees with $\Gamma\EllT(\T/H)$ and in particular the underlying spectrum of $R$ agrees with the global sections $\Gamma(\OO_{\Mf})$. Thus, $R$ is a $\T$-equivariant refinement of $\Gamma(\OO_{\Mf})$; in particular, if $\Mf$ is the moduli stack of elliptic curve, we obtain a $\T$-equivariant refinement of the spectrum $\mathrm{TMF}$ of topological modular forms. Moreover, we see that more generally for every finite $\T$-space $X$, the spectrum $\uMap_{\Sp_{\T}}(\Sigma^{\infty}X, R)$ of $\T$-equivariant maps is equivalent to $\Gamma\EllT(X)$. 

We will show in a sequel to this paper how to extend the argument above to $\T^n$ for $n>1$, the key point being extensions of \cref{thm:monodality} and \cref{lem:invertible} \cite{GepnerMeierII}. Denoting the resulting $\T^n$-spectrum also by $R$, it is true by definition that the fixed points $R^{\B\T^n}$ are equivalent to the global sections of $\OO_{\Ef^{\times_{\Mf}n}}$. As every compact abelian Lie group embeds into a torus, we will get by restriction more generally an equivariant elliptic cohomology spectrum for any compact abelian Lie group. 
\end{construction} 

\section{The circle-equivariant elliptic cohomology of a point}\label{sec:computation}
Before we continue with elliptic cohomology, we need to recall the degree-shifting transfer in equivariant homotopy theory. Recall from \cref{cor:dual} the equivalence of $\Sigma^{\infty}_+G \tensor S^{-L}$ with the Spanier--Whitehead dual $D\Sigma^{\infty}_+G$ for the tangent representation $L$ of an arbitrary compact Lie group $G$. Taking the dual of the map $G_+ \to S^0$ induces thus a map \[\mathbb{S} \to D\Sigma^{\infty}G_+ \simeq S^{-L} \tensor \Sigma^{\infty}G_+.\] 
Mapping into a $G$-spectrum $X$ and taking $G$-fixed points results in a further map 
\[\res_{e}^G(S^L \tensor X) \simeq \uMap^G(S^{-L} \tensor \Sigma^{\infty}G_+, X) \to X^{\B G},\]
called the \emph{degree-shifting transfer}. 

As in the last section, we will fix an oriented elliptic curve $\Ef$ over a non-connective spectral Deligne--Mumford stack $\Mf$. In the last section, we constructed a $\T$-spectrum $R$ with fixed points $R^{\B\T} = \Gamma(\OO_{\Ef})$ and whose underlying spectrum is $\Gamma(\OO_{\Mf})$. Leaving out $\res_e^{\B\T}$ to simplify notation, the degree-shifting transfer thus takes the form of a map $\Sigma R \to R^{\B\T}$. Furthermore, the projection $p\colon \Ef \to \Mf$ induces a map $R \to R^{\B\T}$, which may be seen as restriction along $\T \to \{e\}$. We are now ready for our main calculation.  
\begin{thm}\label{thm:main}
Restriction along $\T \to \{e\}$ and degree shifting shifting transfer define an equivalence
\[R\oplus \Sigma R \xrightarrow{\simeq} R^{\B\T} \]
\end{thm}
\begin{proof}
By the Wirthm{\"u}ller isomorphism \cref{cor:dual}, we can identify the dual of the cofiber sequence
\[\Sigma^\infty \T_+ \to \mathbb{S} \to \Sigma^\infty S^{\rho} \]
with 
\[\Sigma^{-1}\Sigma^\infty \T_+ \leftarrow \mathbb{S} \leftarrow \Sigma^\infty S^{-\rho}.\]
As $\EllT$ is symmetric monoidal on finite $\T$-spectra by \cref{thm:monodality}, it preserves duals. Thus, applying $\EllT$ and taking global sections produces a cofiber sequence
\begin{equation}\label{eq:keycofibersequence}\Sigma R \to R^{\B\T} = \Gamma(\OO_{\Ef}) \to \Gamma(\OO_{\Ef}(e)),\end{equation}
where $e\colon \Mf \to \Ef$ is the unit section and we use \cref{lem:invertible} for the identification of the last term. Essentially by definition, the first map is the degree shifting transfer. 

We will compute $\Gamma(\OO_{\Ef}(e))$ by identifying $p_*\OO_{\Ef}(e)$. Denoting by $e_0\colon \Mf_0 \to \Ef_0$ the underlying classical morphism of $e$, we have already argued in the proof of \cref{lem:invertible} that we can reduce to the case where $\pi_*\OO_{\Mf}$ is even and thus $\pi_*\OO_{\Ef}(e)$ is concentrated in even degrees and the restriction of $\pi_0\OO_{\Ef}(e)$ to the classical locus is $\OO_{\Ef_0}(e_0)$. 

By \cite[\parag 1]{Del75} $R^1p_*\OO_{\Ef_0}(e_0) = 0$ and the morphism $\OO_{\Mf_0} \to p_*\OO_{\Ef_0}(e_0)$ is an isomorphism. By the flatness of $\Ef \to \Mf$ and the projection formula, 
\[R^1p_*(\pi_{2k}\OO_{\Ef}(e)) \cong R^1p_*(\OO_{\Ef_0}(e_0) \tensor_{\OO_{\Ef_0}}p^*\pi_{2k}\OO_{\Mf}) \cong R^1p_*\OO_{\Ef_0}(e_0) \tensor_{\OO_{\Mf_0}} \pi_{2k}\OO_{\Mf}\]
vanishes as well.
Moreover, the higher derived images vanish as $\Ef_0$ is smooth of relative dimension $1$. Thus the relative descent spectral sequence
\[R^s\pi_t\OO_{\Ef}(e) \Rightarrow \pi_{t-s}p_*\OO_{\Ef}(e)\]
is concentrated in the zero-line.\footnote{One can obtain the relative descent spectral sequence e.g.\ by applying $\pi_*p_*$ to the Postnikov tower of $\OO_{\Ef}(e)$ and observing that for a sheaf $\FF$ concentrated in degree $t$, we have $\pi_{t-s}p_*\FF \cong R^s\pi_t\FF$.} We see that the map $\OO_{\Mf} \to p_*\OO_{\Ef} \to p_*\OO_{\Ef}(e)$ is an equivalence and thus taking global sections shows that $R \to R^{\B\T} \to \Gamma(\OO_{\Ef}(e))$ is an equivalence as well. Hence \eqref{eq:keycofibersequence} takes the form of a split cofiber sequence
\[\Sigma R \to R^{\B\T} \to R. \qedhere\]
\end{proof}

\begin{remark}
Intuitively the equivalence $\Gamma(\OO_{\Ef}) \simeq R \oplus \Sigma R$ corresponds to the fact that an elliptic curve (say, over a field) has only non-trivial $H^0$ and $H^1$, both of rank $1$. The appearance of a non-trivial $H^1$-term yields a peculiar behaviour of the completion map $R^{\B\T} \simeq R \oplus \Sigma R \to R^{B\T}$. If $R$ is concentrated in even degrees, we see that the map factors over the standard map $R \to R^{B\T}$ (thus is not injective in any sense, in contrast to the situation for equivariant $K$-theory). At least implicitly, this factorization is a key ingredient for the equivariant proofs of the rigidity of the elliptic genus as in \cite{Rosu}.  
\end{remark}

\begin{cor}\label{cor:torus}
There is an equivalence 
\[\bigoplus_{k=0}^n\bigoplus_{\PP_k(n)} \Sigma^k R \to \Gamma(\OO_{\Ef^{\times n}}) = R^{\B\T^n},\]
where $\PP_k(n)$ runs over all $k$-element subsets of $\{1,\dots, n\}$ and $\Ef^{\times n}$ stands for the $n$-fold fiber product over $\Mf$. 
\end{cor}
\begin{proof}
    Consider the oriented elliptic curve $\Ef^n \to \Ef^{n-1}$ over $\Ef^{n-1}$ given by projection to the first $(n-1)$ coordinates, i.e.\ the pullback of $\Ef \to \Mf$ along $\Ef^{n-1} \to \Mf$. Applying \cref{thm:main}, we see that 
    \[\Gamma(\OO_{\Ef^{\times n}}) \simeq \Gamma(\OO_{\Ef^{\times (n-1)}}) \oplus \Sigma \Gamma(\OO_{\Ef^{\times (n-1)}}) \]
    Induction yields the result.
\end{proof}

\begin{example}
Let $\Mf = \Mf_{\Ell}^{or}$ be the  moduli stack of oriented elliptic curves and $\Ef$ be the universal oriented elliptic curve over it (see \cite[Proposition 7.2.10]{LurEllII} and \cite[Section 4.1]{MeierLevel} for definitions). The stack $\Mf_{\Ell}^{or}$ can be thought of as the classical moduli stack of elliptic curves equipped with the Goerss--Hopkins--Miller sheaf $\OO^{top}$ of $E_{\infty}$-ring spectra. In this case, $R = \Gamma(\OO^{top})$ is by definition the spectrum $\mathrm{TMF}$ of topological modular forms. Our theory specializes to define $\T^n$-spectra with underlying spectrum $\mathrm{TMF}$ and the previous corollary computes its $\T^n$-fixed points. Let us comment on the fixed points for other compact abelian Lie groups.

In \cite{MeierLevel} the second-named author computed the $G$-fixed points of $\mathrm{TMF}$ for $G$ a finite abelian group after completing at a prime $l$ not dividing the group order, namely as a sum of shifts of $\widehat{\mathrm{TMF}_1(3)}_2$ if $p=2$, of $\widehat{\mathrm{TMF}_1(2)}_3$ if $l=3$ and of $\widehat{\mathrm{TMF}}_p$ if $l>3$. In the latter case, the result can be strengthened as follows: Denote by $p_0\colon \Ef_0 \to \Mf_{\Ell}$ the underlying map of classical Deligne--Mumford stacks of $p\colon \Ef \to \Mf_{\Ell}^{or}$. By \cite[Lemma 4.7]{MeierLevel}, the morphism $\Hom(\Dual{G}, \Ef_0) \to \Mf_{\Ell}$ is finite and flat. Thus $(p_0)_*\OO_{\Hom(\Dual{G}, \Ef_0)}$ is a vector bundle on $\Mf_{\Ell}$. After localizing at a prime $l>3$, \cite[Theorem A]{Mei13} implies that this vector bundle splits into a sum of $\omega^{\tensor i} = \pi_{2i}\OO^{top}$. A descent spectral sequence argument shows that $\mathrm{TMF}^{\B G}_{(l)} = \Gamma(\OO_{\Hom(\Dual{G}, \Ef)})_{(l)}$ splits into shifts of $\mathrm{TMF}_{(l)}$, without the assumption that $l$ does not divide $|G|$. In contrast, if $l=2$ or $3$ and it divides the order of the group, the computation of $\mathrm{TMF}_{(l)}^{\B G}$ is still wide open in general, though Chua has computed recently $\mathrm{TMF}_{(2)}^{\B C_2}$ \cite{Chua}.

More generally, we can consider $\mathrm{TMF}^{\B(G \times \T^n)}$ for a finite abelian group $G$. This coincides with the global sections of the structure sheaf of $\Hom(\Dual{G}, \Ef) \times_{\Mf} \Ef^n$. Denoting the pullback of $\Ef$ to $\Hom(\Dual{G}, \Ef)$ by $\Ef'$ this agrees with the $n$-fold fiber product of $\Ef'$ over $\Hom(\Dual{G}, \Ef)$. Thus, \cref{cor:torus} implies
\[ \mathrm{TMF}^{\B(G\times \T^n)} \simeq \bigoplus_{k=0}^n\bigoplus_{\PP_k(n)} \Sigma^k \mathrm{TMF}^{\B G}. \]
\end{example}

\begin{cor}\label{cor:dualizable}
The $\mathrm{TMF}$-module $\mathrm{TMF}^{\B H} = \Gamma(\OO_{\Ell(\B H)})$ of $H$-fixed points is dualizable for every abelian compact Lie group $H$. 
\end{cor}
\begin{proof}
Let $\Mf = \Mf_{\Ell}^{or}$ and $\Ef$ as in the example above. We can split $H \cong G \times T$, where $G$ is finite and $T$ a torus of dimension $n$. As seen in the previous example, $\mathrm{TMF}^{\B A}$ splits as a sum of suspensions of $\mathrm{TMF}^{\B G}$ and so it suffices to show that $\mathrm{TMF}^{\B G}$ is dualizable. Arguing as in \cite[Proposition 2.13]{MeierLevel} this follows from $\Hom(\widehat{G}, \Ef_0) \to \Mf_{\Ell}$ being finite and flat, with notation as in the previous example.
\end{proof}

\begin{remark}We conjecture that $\mathrm{TMF}^{\B G}$ is dualizable for every compact Lie group. (Strictly speaking, we have not defined $\mathrm{TMF}^{\B G}$ if $G$ is not abelian, but this we see as the lesser problem.) Crucial evidence is given by Lurie's results on the finite group case, which we will summarize. 

Every oriented spectral elliptic curve $\Ef$ over an $E_{\infty}$-ring $A$ gives rise to a $\mathbf{P}$-divisible group $\mathbf{G}$ over $A$ in the sense of \cite[Definition 2.6.1]{LurEllIII} by \cite[Section 2.9]{LurEllIII}. By \cite[Notation 4.0.1]{LurEllIII} this in turn gives rise to a functor $A_{\mathbf{G}}\colon \TT^{\op} \to \CAlg_A$, where $\TT$ denotes the full subcategory of $\Spc$ on the spaces $BH$ for $H$ a finite abelian group, and \cite[Construction 3.2.16]{LurEllIII} allows to extend this to a functor on $H$-spaces for all finite groups $H$. We conjecture that the spectrum Lurie associates to $H \curvearrowright \pt$ is equivalent to our spectrum $\Gamma(\OO_{\Ell(\B H)})$ and moreover that these are the $H$-fixed points of a global spectrum $A$ (e.g.\ in the sense of \cite{SchGlobal}) -- in any case, we  denote the spectrum constructed by Lurie by $A^{\B H}$.

Lurie's  Tempered Ambidexterity Theorem implies that this spectrum is a dualizable (and even self-dual) $A$-module as follows: Lurie defines for every finite group $H$ an $\infty$-category $\LocSys_{\mathbf{G}}(\B H)$. For $H = e$, we have $\LocSys_{\mathbf{G}}(\pt) \simeq \Mod(A)$ and the pushforward $f_*\underline{A}_{\B H}$ of the ``constant local system'' along $f\colon \B H \to \pt$ corresponds to $A^{\B H}$ under this equivalence. Moreover, the Tempered Ambidexterity Theorem \cite[Theorem 1.1.21]{LurEllIII} identifies this with $f_!\underline{A}_{\B H}$, where $f_!$ denotes the left adjoint to the restriction functor. We compute
\[\uMap_A(A^{\B H}, A) \simeq \uMap_A(f_!\underline{A}_{\B H}, A) \simeq \uMap_{\LocSys_{\mathbf{G}}(\B H)}(\underline{A}_{\B H}, f_*A) \simeq f_* \underline{A}_{\B H} \simeq A^{\B H}.\]
\end{remark}

\appendix

\section{Quotient $\i$-categories}\label{app:q}
In this appendix we collect some basic results concerning quotient $\i$-categories.
Recall that a quotient of an ordinary category $\DD$ is a category $\CC$ obtained by identifying morphisms in $\CC$ by means of an equivalence relation that is suitably compatible with composition.
A quotient functor $q\colon\DD\to\CC$ is necessarily essentially surjective, as $\CC$ is typically taken to have the same objects as $\DD$, though nonisomorphic objects of $\DD$ may become isomorphic in $\CC$.
A standard example (and one which is important for our purposes) is the homotopy category of spaces, which is a quotient of the ordinary category of Kan complexes by the homotopy relation.

For some of the arguments in this appendix, it will be convenient to regard 
\[
\Cat_\i\subset\Fun(\Delta^{\op},\SS)
\]
as a reflective subcategory of the $\i$-category of simplicial spaces.
From this point of view, many different simplicial spaces $\CC_\bullet$ give rise to the same $\i$-category $\CC$, even if we assume that the simplicial space satisfies the Segal condition.
The standard simplicial model of $\CC$ is the simplicial space $\CC_\bullet$ with $\CC_n=\Map(\Delta^n,\CC)$; up to equivalence, this is the only {\em complete} Segal space model of $\CC$.
However, we will have occasion to consider other (necessarily not complete) models as well, especially those models which arise from restricting the space of objects along a given map.
Note that to compute the geometric realization of a simplicial object $\Delta^{\op}\to\Cat_\i$, we may first choose any lift to $\Fun(\Delta^{\op},\SS)$, where the realization is computed levelwise, and then localize the result (which amounts to enforcing the Segal and completeness conditions).

\begin{construction}
Let $\CC_\bullet\colon\Delta^{\op}\to\SS$ be a simplicial space and let $f\colon\DD_0\to\CC_0$ be a morphism in $\SS$.
Let $i\colon\Delta_{\leq 0}\to\Delta$ denote the inclusion of the terminal object $[0]$, and consider the unit map $\CC_\bullet\to i_*i^*\CC_\bullet\simeq i_*\CC_0$.
Define a simplicial space $\DD_\bullet\colon\Delta^{\op}\to\SS$ as the pullback
\[
\xymatrix{\DD_\bullet\ar[r]\ar[d] & \CC_\bullet\ar[d]\\
i_*\DD_0\ar[r] & i_*\CC_0 .}
\]
We also refer to this simplicial space as $f^*\CC_\bullet$.
Observe that, more or less by construction, if $\CC_\bullet$ is additionally a Segal space, then $\DD_\bullet$ is also a Segal space, and moreover that if $S$ and $T$ are any pair of points in $\DD_0$, then
\[
\Map_{\DD_\bullet}(S,T)\simeq\Map_{\CC_\bullet}(f(S),f(T)).
\]
In particular, $\DD_\bullet\to\CC_\bullet$ is fully faithful, and it is essentially surjective if $f$ is $\pi_0$-surjective (and conversely if $\CC_\bullet$ is additionally complete).
\end{construction}

Thus, if $f\colon S\to\CC_0$ is any $\pi_0$-surjective map of spaces, then the simplicial space $f^*\CC_\bullet$ obtained by restriction is a simplicial space model of $\CC$ which satisfies the Segal condition, though it will not be complete unless $f$ is an equivalence (cf.\ \cite[Theorem 7.7]{Rez01}). In particular, we can take $S$ to be discrete and see that every Segal space is equivalent to one with a discrete space of objects. Moreover, given an essentially surjective functor $q\colon \DD \to \CC$, we may choose a $\pi_0$-surjection $f\colon S \to \DD_0$ and replace $q$ by the equivalent map $f^*\DD \to (qf)^*\CC$, which is the identity on the (discrete) spaces of objects. If $q$ is not essentially surjective, we can factor it into an essentially surjective functor and an inclusion, thus obtaining a Segal space model that is an injection on discrete spaces of objects. 

\begin{definition}
A morphism of $\i$-categories $q\colon\DD\to\CC$ is a {\em quotient functor} if it is essentially surjective and $\pi_0\Map_\DD(s,t)\to\pi_0\Map_\CC(s,t)$ is surjective for all pairs of objects $s$ and $t$ of $\DD$.
\end{definition}

In other words, $q\colon\DD\to\CC$ is a quotient functor if all objects and arrows of $\CC$ lift to objects and arrows of $\DD$, respectively.

\begin{proposition}\label{prop:qcolim}
A functor $q\colon\DD\to\CC$ is a quotient if and only if the augmented simplicial diagram
\[
\cdots\rrrrarrow\DD\times_\CC\DD\times_\CC\DD\rrrarrow\DD\times_\CC\DD\rrarrow\DD\to\CC
\]
is a colimiting cone.
\end{proposition}

\begin{proof}
First suppose that the canonical map $|\DD^{\times_\CC\bullet}|\to\CC$ is an equivalence.
In particular, it is essentially surjective, so that  $\pi_0|\DD^{\times_\CC\bullet}|_0\to\pi_0\CC_0$ is surjective.
Note that $\pi_0\DD_0$ surjects onto $\pi_0|\DD^{\times_\CC\bullet}|_0$. This can be seen for example by using a Segal space model where $\DD \to \CC$ is injective on discrete spaces of objects and thus \cref{lem:segalcat} implies that $|\DD^{\times_\CC\bullet}|$ can be viewed as a Segal space with the same set of objects as $\DD$. Thus, the composite $\pi_0\DD_0\to\pi_0\CC_0$ is surjective as well and it follows that $\DD\to\CC$ is essentially surjective.
We now show that $\DD\to\CC$ is also surjective on mapping spaces.
 By the essential surjectivity, the map $q\colon\DD\to\CC$ admits a Segal space model which is the identity on (discrete) spaces of objects $S$, in which case, by Lemma \ref{lem:segalcat} below, the realization of $\DD^{\times_\CC\bullet}$ can be computed levelwise. But then 
$\Map_\CC(A,B)$ is the realization of $\Map_{\DD^{\times_\CC\bullet}}(A,B)$, and consequently $\pi_0\Map_{\DD^{\times_\CC\bullet}}(A,B)$ surjects onto $\pi_0\Map_\CC(A,B)$.

Conversely, suppose that $q\colon\DD\to\CC$ is a quotient functor, and consider the comparison map $|\DD^{\times_\CC\bullet}|\to\CC$.
Since $\pi_0\DD_0\to\pi_0\CC_0$ is surjective, we may again suppose without loss of generality that $\CC$ and $\DD$ admit Segal space models with same discrete space of objects $S$, in which case it follows that $\pi_0|\DD^{\times_\CC\bullet}|_0\cong\pi_0|\DD^{\times_\CC\bullet}_0|\cong\pi_0\CC$.
In particular, $|\DD^{\times_\CC\bullet}|\to\CC$ is essentially surjective.
To see that it is fully faithful, we use that for any $\pi_0$-surjective map $X \to Y$ of spaces, the induced map $|X^{\times_Y \bullet}| \to Y$ is an equivalence (see e.g.\ \cite[Corollary 7.2.1.15]{HTT}). We thus compute the mapping spaces as follows: \begin{align*}
    \Map_{|\DD^{\times_\CC\bullet}|}(A,B)&\simeq|\Map_{\DD^{\times_\CC\bullet}}(A,B)|\\
    &\simeq |\Map_{\DD}(A,B)^{\times_{\Map_{\CC}(A,B)}\bullet}| \\
    &\simeq \Map_{\CC}(A,B) \qedhere
\end{align*}
\end{proof}

\begin{lemma}\label{lem:segalcat}
    Suppose given a simplicial Segal space
    $\DD_\bullet\colon\Delta^{\op}\to\Fun^\mathrm{Seg}(\Delta^{\op},\SS)$
    with a constant and discrete space of objects; that is, the functor $\Delta^{\op}\to\Fun^\mathrm{Seg}(\Delta^{\op},\SS)\to\SS$ obtained by composing $\DD_\bullet$ with evaluation at $[0]$ is constant with value some discrete space $S\in\tau_{\leq 0}\SS$.
    Then $|\DD_\bullet|\in\Fun(\Delta^{\op},\SS)$ is again a Segal space with object space $S$. Moreover, for any pair of objects $s_0, s_1 \in S$, there is an equivalence $\Map_{|\DD_{\bullet}|}(s_0, s_1) \simeq |\Map_{\DD_{\bullet}}(s_0, s_1)|$.
\end{lemma}
\begin{proof}
    The space of objects of $|\DD_\bullet|$ is again $S$, which is discrete.
    Using the fact that geometric realization commutes with colimits and finite products, we see that for $0<k<n$,
    \[
    \Map(\Delta^n,|\DD_\bullet|)\simeq\coprod_{(s_0,\ldots,s_n)\in S^{\times[n]}}\Map_{|\DD_\bullet|}(s_0,s_1)\times\cdots\times\Map_{|\DD_\bullet|}(s_{n-1},s_n)\simeq\Map(N^n,|\DD_\bullet|),
    \]
    where $N^n\simeq\Delta^1\coprod_{\Delta^0}\coprod\cdots\coprod_{\Delta^0}\Delta^1$ denotes the $n$-spine.
    It follows that $|\DD_\bullet|$ satisfies the Segal condition. The last statement is the case $n=1$, which is similar, but strictly easier, as the term involving the 1-spine $N^1=\Delta^1$ is irrelevant.
\end{proof}
\begin{prop}
Let $\DD_\bullet\colon\Delta^{\op}\to\Cat_\infty$ be a functor with colimit $\CC\simeq |\DD_\bullet|$.
Then the projection $\DD_0\to\CC$ is a quotient functor.
\end{prop}
\begin{proof}
    Similar to the first part of the proof of Proposition \ref{prop:qcolim}.
\end{proof}
Let $\CC$ be an $\i$-category with finite limits and let $X\in\CC$ be an object.
We can define a new $\i$-category $\CC[X]$ in which the mapping spaces
\[
\Map_{\CC[X]}(S,T)\simeq\Map_\CC(X\times S,T)
\]
are $X$-indexed families of maps.
More precisely, we consider a functor $\mu_X\colon\CC\to\CC$ given by multiplication with $X$.
Since $X\simeq\mu_X(\pt)$, this factors through the slice $\CC_{/X}$,
and we further factor this functor as a composite
\[
\CC\to\CC[X]\subset\CC_{/X}.
\]
Here $\CC[X]\subset\CC_{/X}$ denotes the full subcategory consisting of those objects in the image of $\mu_X$, and (for ease of notation) we will simply denote these objects by their name in $\CC$.
This approach has the advantage that composition is well-defined ($\CC[X]$ is a full subcategory of $\CC_{/X}$), and we have an equivalence
\[
\Map_{\CC[X]}(S,T)=\Map_{\CC_{/X}}(X\times S,X\times T)\simeq\Map_{\CC}(X\times S,T).
\]
To see that this construction is natural in $X$, we use an auxiliary construction.
Consider the target projection $p\colon\Fun(\Delta^1,\CC)\to\CC$.
It is a cartesian fibration via pullback, since $\CC$ admits finite limits, and the straightening of this fibration is a functor $\CC^{\op}\to\Cat_\i$, which sends $X$ to $\CC_{/X}$.
In particular, given a simplicial set $I$ and a functor $f\colon I\to\CC$, we obtain the desired functor $I^{\op}\to\Cat_\i$ by precomposing $\CC^{\op}\to\Cat_\i$ with $f^{\op}$.

\begin{construction}\label{con:cosimplicialquotient}
Let $\DD$ be an $\i$-category with finite products equipped with a cosimplicial object $\Delta_\DD\colon\Delta\to\DD$.
Let $\DD_n:=\DD[\Delta^n_\DD]\subset\DD_{/\Delta^n_\DD}$, so that we obtain a simplicial $\i$-category $\DD_\bullet$, and set $\CC=|\DD_\bullet|$.
Then the projection $q\colon\DD\simeq\DD_0\to\CC$ is a quotient functor.
\end{construction}

\section{Comparison of models for the global orbit category}\label{app:GlobalOrbitModels}
The goal of this appendix is to compare our definition of the global orbit category $\Orb$ with the construction in \cite{GH}. This is especially important because the latter one has been compared to other models of unstable global homotopy theory in \cite{KorOrbi} and \cite{SchwedeOrbi}. As this appendix will involve point set considerations, we will use $\hocolim$ to denote a homotopy colimit in a relative category or, equivalently, a colimit in the associated $\infty$-category. 

The model of \cite{GH} of the global orbit category is based on the notion of the fat geometric realization of a simplicial space, which we denote by $||-||$. In contrast to the usual geometric realization, $||X_{\bullet}||$ is always a model for the homotopy colimit $\hocolim_{\Delta^{\op}}X_{\bullet}$. Segal has shown in \cite[Appendix A]{SegalCategories} that the fat realization is always homotopy invariant and equivalent to the usual realization on good spaces and in particular on Reedy cofibrant ones; on these it is well-known that the geometric realization is a model of the homotopy colimit.

 One usually defines the fat realizations like the geometric realizations, but one only takes identifications along face and not along degeneracy maps. We will use the equivalent definition as the geometric realization of the simplicial space $\overline{X}_{\bullet}$ that arises as the left Kan extension of $X_{\bullet}$ along $\Delta_{\mathrm{inj}} \to \Delta$, where $\Delta_{\mathrm{inj}}$ stands for the subcategory of $\Delta$ of injections. Concretely, we have $\overline{X}_n =\coprod_{f\colon [n]\twoheadrightarrow [k]} X_k$; given $\alpha\colon [m] \to [n]$, we factor the composite to $[k]$ as $h\circ g\colon [m] \twoheadrightarrow [l] \hookrightarrow [k]$ and get $\alpha^*(f,x) = (g, h^*x)$. 
 
 Homotopically, fat realization is strong monoidal: Segal showed in \cite[Appendix A]{SegalCategories} that the canonical map $||X \times Y||\to ||X||\times ||Y||$ is a weak equivalence and $||\pt|| = \Delta^{\infty}$ is contractible. We claim moreover that fat realization is lax monoidal on the nose: The map $||X \times Y||\to ||X||\times ||Y||$  admits a retract $r$, which is induced by a retract $\overline{r}\colon \overline{X} \times \overline{Y} \to \overline{X \times Y}$ of the canonical map: Given a pair of $(f\colon [n] \twoheadrightarrow [k], x\in X_k)$ and $(g\colon [m] \twoheadrightarrow [l], y\in Y_k)$, there is a unique surjection $h\colon [n] \twoheadrightarrow [p]$ such that $f$ and $g$ factor over $h$ and $p$ is minimal with this property. We define $\overline{r}((f,x), (g, y))$ as the pair of $h$ and the pair of images of $x$ and $y$ in $(X \times Y)_p$. 
 
 For a simplicial diagram $\CC_{\bullet}$ in $\TopCat$ whose functor of objects is constant, applying the fat realization to the mapping spaces thus defines a topological category $||\CC_{\bullet}||$. The analogous construction also works for simplicial diagrams in simplicial categories. 
 
 \begin{definition}
 Consider the category of topological groupoids as enriched over itself. Define $\Orb^{\Gpd}$ as the full enriched subcategory on the groupoids $\{G\rrarrow \pt\}$ for $G$ compact Lie. We define the Gepner--Henriques global orbit category $\Orb'$ as the topological category $||\Orb^{\Gpd}||$ (cf.\ \cite[Remark 4.3]{GH}). Concretely, we have
 \[
 \Map_{\Orb'}(H,G)=||\underline{\Map}_{\TopGpd}(H\rrarrow \pt,G\rrarrow \pt)|| 
 \]
if we identify its objects with the corresponding compact Lie groups. 
 \end{definition}
 
 To compare $\Orb'$ with our definition of $\Orb$, we have to investigate the precise homotopical meaning of taking homwise fat realization of a simplicial diagram of topological categories.
 
 For us, a functor between topological or simplicial categories is a \emph{(Dwyer--Kan) equivalence} if it is a weak equivalence on all mapping spaces and an equivalence on homotopy categories. Homwise geometric realization and singular complex define inverse equivalences between the relative categories $\TopCat$ and $\sCat$. As the latter can be equipped with the Bergner model structure, we see that both have all homotopy colimits. We moreover recall the Rezk model structure on $\Fun(\Delta^{\op}, \sSet)$ from \cite{Rez01}.
 
 \begin{prop}
     Let $\CC_{\bullet}$ be a simplicial diagram in $\TopCat$ whose functor of objects is constant. Then  $||\CC_{\bullet}||$ is a model for the homotopy colimit  $\hocolim_{\Delta^{\op}}\CC_{\bullet}$ in topological categories. Furthermore, the analogous statement is true for simplicial categories.
 \end{prop}
 \begin{proof}
 Observe first that homwise geometric realization and singular complex define an equivalence between the relative categories $\TopCat$ and $\sCat$. As the geometric realization functor $\sSet \to \Top$ commutes with fat realization of simplicial objects, it suffices to show the statement for simplicial categories.
     
     We will use the nerve functor $\mathfrak{N}\colon \sCat \to \Fun(\Delta^{\op},\sSet)$ defined as follows: given a simplicial category $\CC$, the zeroth space of $\mathfrak{N}\CC$ is the set of objects, the first space is the disjoint union over all mapping spaces and the higher ones are disjoint unions of products of composable mapping spaces. To analyze this functor, we introduce the intermediate category $\SeCat$ of Segal precategories, i.e.\ objects of $\Fun(\Delta^{\op},\sSet)$ whose zeroth space is discrete. In \cite{BergnerThree}, Bergner writes $\mathfrak{N}$ as a composite
     \[\sCat \xrightarrow{R} \SeCat \xrightarrow{I} \Fun(\Delta^{\op}, \sSet). \]
     The functor $R$ is a right Quillen equivalence for a certain model structure on $\SeCat$. The functor $I$ is a left Quillen equivalence from a different model structure with the same weak equivalences and all objects cofibrant to the Rezk model structure on $\Fun(\Delta^{\op}, \sSet)$. Thus we see that $\mathfrak{N}$ induces an equivalence of relative categories between fibrant simplicial categories and $\Fun(\Delta^{\op},\sSet)$ with the Rezk equivalences. In particular, $\mathfrak{N}$ preserves and reflects homotopy colimits on fibrant simplicial categories. 
     
     As fat realization commutes with disjoint unions and is homotopically strong symmetric monoidal, we see that $\mathfrak{N}$ intertwines the homwise fat realization in $\sCat$ with the levelwise fat realization in $\Fun(\Delta^{\op}, \sSet)$ up to levelwise equivalence of simplicial spaces. Moreover, levelwise fat realization is a homotopy colimit for levelwise equivalences of simplicial spaces and hence also for Rezk equivalences (as the Rezk model structure is a left Bousfield localization of the Reedy model structure). We deduce that homwise fat realization is a homotopy colimit in simplicial categories.
 \end{proof}
 
 \begin{cor}\label{cor:hocolim}
 Let $\CC_{\bullet}$ be a simplicial diagram $\CC_{\bullet}$ in $\TopCat$ whose functor of objects is constant. Then $N^{\coh}\Sing ||\CC_{\bullet}||$ is equivalent to $\hocolim_{\Delta^{\op}}N^{\coh}\Sing \CC_{\bullet}$, where $\Sing$ and $||-||$ are formed on the level of mapping (simplicial) spaces.
 \end{cor}
 \begin{proof}
     Applying the singular complex to the mapping spaces of a topological category results in a fibrant simplicial category. The coherent nerve $N^{\coh}$ preserves homotopy colimits of fibrant simplicial categories (as the right derived functor of $N^{\coh}$ defines an equivalence of $\infty$-categories between simplicial categories and quasi-categories). As $\Sing$ also commutes with homotopy colimits, the result follows from the preceding proposition. 
 \end{proof}
 
 Our main interest in simplicial diagrams of topological categories lies in their association with topological groupoids (and thus to $\Orb'$): 
 We can associate to a topological groupoid $X_{\bullet}$ a simplicial topological space and following \cite{GH} we denote by $||X_{\bullet}||$ the fat realization of  this space. On the level of categories we obtain from a category enriched in topological groupoids a simplicial diagram in topological categories with a constant set of objects. We obtain the following further corollary. 
 
 \begin{cor}\label{cor:TopGpdInfinity}
 Let $\CC$ be a category enriched in topological groupoids. Let $||\CC||$ be the topological category obtained by applying $||-||$ homwise. Then the associated $\infty$-category $N^{\coh}\Sing ||\CC||$ can alternatively be computed as follows: Apply $\Sing$ homwise to $\CC$ to obtain a simplicial diagram in groupoid-enriched categories, take the associated $\infty$-categories and take their homotopy colimit. 
 \end{cor}
 
 Recall that $\Orb^{\Gpd}$ is the full subcategory of the $\TopGpd$-enriched version of $\TopGpd$ on the objects of the form $\{G \rrarrow \pt\}$ for $G$ compact Lie. Applying $\Sing$ to the mapping groupoids results in a simplicial object $\Orb^{\Gpd}_{\bullet}$ in groupoid-enriched categories whose mapping groupoids in level $n$ are $\Map_{\TopGpd}(G\times\Delta^n \rrarrow \Delta^n , H \rrarrow \pt)$. (Recall here that $\Map_{\TopGpd}$ denotes the groupoid and not the topological groupoid of maps.) There is a corresponding simplicial object $\Orb^{\Stk}_{\bullet}$ whose $n$-th groupoid enriched category has mapping groupoids $\Map_{\Stk}(\Delta^n \times \B G, \B H)$. Stackification provides a map $\Orb^{\Gpd}_{\bullet}\to \Orb^{\Stk}_{\bullet}$ of simplicial groupoid-enriched categories. 
 
 \begin{lemma}\label{lem:GpdStk}
     The map $\Orb^{\Gpd}_{\bullet}\to \Orb^{\Stk}_{\bullet}$ defines in each level on each mapping groupoid an equivalence.
 \end{lemma}
 \begin{proof}
     Let $H$ and $G$ be compact Lie groups. We have to show that 
     \[ \Map_{\TopGpd}(G\times\Delta^n \rrarrow \Delta^n, H \rrarrow \pt) \to \Map_{\Stk}(\Delta^n\times \B G, \B H)
     \]
     is an equivalence of groupoids. As $\B H$ is a stack and stackification is left adjoint, the target is equivalent to 
     $\Map_{\Pre\Stk}(G\times\Delta^n \rrarrow \Delta^n , \B H)$. 
     This groupoid is the groupoid of $H$-principal bundles on $\{G\times\Delta^n \rrarrow \Delta^n\}$: an object is a principal $H$-bundle on $\Delta^n$ with an
     isomorphism between the two pullbacks to $G\times\Delta^n$ satisfying a cocycle condition. As every principal $H$-bundle on $\Delta^n$ is trivial, we may up to equivalence replace this groupoid by the subgroupoid where we require the $H$-principal bundle on $\Delta^n$ to be equal to $H\times\Delta^n$. This subgroupoid is precisely $\Map_{\TopGpd}(G\times\Delta^n \rrarrow \Delta^n, H \rrarrow \pt)$. The set of objects is $\Map(\Delta^n, \Hom(G,H))$ and the set of morphisms is $\Map(\Delta^n, \Hom(G,H) \times H)$.
 \end{proof}
 
 Recall that we can view every groupoid-enriched category as a quasi-category via the Duskin nerve $N^{\mathrm{Dusk}}$; this is isomorphic to taking the homwise nerve and applying the coherent nerve. Note further by construction there is for each $n$ a fully faithful embedding $N^{\mathrm{Dusk}}\Orb_n^{\Stk} \to \Stk[\Delta^n]$, with $[\Delta^n]$ as in the preceding appendix. These define a morphism of simplicial diagrams which are constant on objects, and thus taking $\hocolim_{\Delta^{\op}}$ results in a fully faithful embedding
 \[\hocolim_{\Delta^{\op}}N^{\mathrm{Dusk}}\Orb_{\bullet}^{\Stk} \to \Stk_{\infty}\]
 with image $\Orb$, where $\Stk_{\infty}$ is as in \cref{def:Stki} (cf.\ \cref{lem:segalcat}). By composing this with the equivalence of the preceding lemma, we obtain an equivalence
 \[\hocolim_{\Delta^{\op}}N^{\mathrm{Dusk}}\Orb_{\bullet}^{\Gpd} \to \Orb.\]
 By \cref{cor:TopGpdInfinity} the source can be identified with $N^{\coh}\Sing \Orb'$, i.e.\ with the $\infty$-category associated with $\Orb'$. Thus we obtain the goal of this appendix:
 
 \begin{prop}\label{prop:OrbComparison}
 Stackification induces an equivalence between $\Orb$ and the $\infty$-category associated with $\Orb'$.
 \end{prop}

\section{The $\infty$-category of $G$-spectra}\label{app:GSpectra}
The aim of this appendix is to introduce the $\infty$-category of equivariant spectra, compare it to the theory of orthogonal spectra and deduce from Robalo's thesis a universal property of this $\infty$-category. Other $\infty$-categorical treatments of $G$-spectra for (pro)finite groups $G$ include \cite{BarwickSpec} and \cite{NardinOrb}. 

\subsection{$G$-spaces and $G$-spectra}We first review a construction of $\i$-categories from $1$-categorical data:
A relative category is a category with a chosen class of morphisms, called \emph{weak equivalences}, containing all identities and closed under composition. To every relative category $\CC$, we can consider an $\infty$-category $L\CC$, which universally inverts all weak equivalences. We recall the following functorial construction: Let $\mathrm{core}\colon \Cat_{\infty} \to \Spc$ be the right adjoint of the inclusion (i.e.\ the maximal $\i$-subgroupoid) and consider the resulting functor 
\[\Cat_{\i} \to \Cat_{\i}^{\Delta[1]},\qquad \CC \,\mapsto\, (\mathrm{core}(\CC) \to \CC).\]
This functor preserves limits and is accessible (as the same is true for $\mathrm{core}$) and thus the adjoint functor theorem implies that it admits a left adjoint $L_{\i}\colon \Cat_{\i}^{\Delta[1]} \to \Cat_{\i}$. The universal property of the adjunction implies that $\CC \simeq L_{\i}(\varnothing \subset \CC) \to L_{\infty}(\WW \subset \CC)$ is a localization at $\WW$ in the sense of \cite[Definition 7.1.2]{Cisinski}. Let $\RelCat$ denote the $(2,1)$-category of relative categories, weak equivalence preserving functors and natural isomorphisms. Viewing $(2,1)$-categories as $\i$-categories as before, we obtain a composite functor \[\RelCat \to \Cat^{\Delta[1]} \to \Cat_{\i}^{\Delta[1]} \xrightarrow{L_{\i}} \Cat_{\i},\] which we define to be $L$.  

From now on let $G$ be a compact Lie group. We denote by $\Top^G$ the relative category of (compactly generated, weak Hausdorff) topological spaces with $G$-action, where weak equivalences are detected on fixed points for all closed subgroups $H\subset G$. We denote by $\Spc^G = L\Top^G$ the associated $\infty$-category. 

Denote by $\Orb_G$ the orbit category. This is the $\infty$-category associated to the full topological subcategory of $\Top_G$ on the $G/H$ where $H$ is a closed subgroup; equivalently, it is the full sub-$\infty$-category of $\Spc^G$ on the same $G/H$ (cf.\ \cite[Theorem 1.3.4.20]{HA}). Elmendorf's theorem implies that 
\[\Spc^G\to \Fun(\Orb^{\op}_G, \Spc), \qquad X \mapsto \Map(-, X)\] is an equivalence. While Elmendorf's original article \cite{Elmendorf} only showed an equivalence of homotopy categories, this was upgraded to an equivalence of homotopy theories in \cite{DwyerKanSingularFunctors}. The precise $\infty$-categorical version we are stating can be found in \cite[Example 8.2]{LinskensNardinPol}. 

There are also pointed versions: Calling a $G$-space \emph{well-pointed} if all its fixed points are well-pointed, we denote by $\Top_*^G$ the relative category of well-pointed topological $G$-spaces and by $\Spc_*^G$ the associated $\infty$-category. By the dual version of \cite[Corollary 7.6.13]{Cisinski} or by \cite[Proposition 2.3]{LinskensNardinPol}, this agrees with pointed objects in $\Spc^G$ and thus $\Spc_*^G  \simeq \Fun(\Orb^{\op}_G, \Spc_*)$. 

Given an orthogonal $G$-representation $V$, we denote by $S^V$ its one-point compactification and by $\Sigma^V$ the functor $- \sm S^V$. This defines a homotopical functor $\Top_*^G \to \Top_*^G$. Indeed: By \cite[Proposition B.1(iii)]{SchGlobal}, we have $(S^V \sm X)^H\cong S^{V^H}\sm X^H$. This reduces to the non-equivariant version, where it is well-known (see also \cite[p.258]{SchGlobal}). We remark that the induced functor $\SS_*^G \to \SS_*^G$ agrees with the tensor $- \tensor S^V$ in the pointwise monoidal structure on $\Fun(\Orb^{\op}_G, \Spc_*)$. 

We fix in the following an orthogonal $G$-representation $\UU$. In the main body of this article, this will always be a complete universe, i.e.\ $\UU$ is countably-dimensional and contains up to isomorphism every countable direct sum of finite-dimensional orthogonal $G$-representations. (This always exists as e.g.\ shown in \cite[p.20]{SchGlobal}.)

Informally speaking, we obtain the $\infty$-category $\Sp^G_{\UU}$ of $G$-spectra by inverting all subrepresentations of $\UU$ on $\Spc_*^G$. More precisely, let $\Sub_{\UU}$ be the poset of finite-dimensional subrepresentations of $\UU$. We define a (pseudo-)functor $T\colon \Sub_{\UU} \to \RelCat$ sending each $V$ to $\Top_*^G$. Given an inclusion $V\subset W$, we denote by $W-V$ the orthogonal complement of $V$ in $W$ and define the relative functor $T(V\subset W)$ as $\Top_*^G \xrightarrow{\Sigma^{W-V}} \Top_*^G$. In particular, this induces a functor $LT\colon \Sub_{\UU} \to \Cat_{\infty}$.
	
	By its characterization as a functor category, $\Spc_*^G$ is presentable, and for each $V\subset W$, the functor $LT(V\subset W)$ is a left adjoint with right adjoint $\Omega^{W-V}$. As by Illman's theorem \cite[Corollary 7.2]{Illman}, $S^{W-V}$ has the structure of a finite $G$-CW-complex, $\Omega^{W-V} \colon \Spc_*^G \to \Spc_*^G$ preserves filtered colimits (as can, for example, be shown by induction over the cells) and hence $\Sigma^{W-V}$ preserves compact objects. Thus, $LT$ can be seen as taking values in the $\infty$-category $\Cat^{\omega}_{\infty}$ of compactly generated $\infty$-categories and compact-object preserving left adjoints. (We refer to \cite[Appendix A]{HeutsGood} for a quick overview of compactly generated $\infty$-categories.)
	
\begin{defi}
The $\infty$-category $\Sp^G_{\UU}$ of $G$-spectra is the colimit over $LT$ in $\Cat^{\omega}_{\infty}$. If $\UU$ is a complete universe, we write $\Sp^G$ for $\Sp^G_{\UU}$. We denote the map $\Spc_*^G \to \Sp^G_{\UU}$ corresponding to the subspace $0\subset \UU$ by $\Sigma^{\infty}$. 
\end{defi}

By \cite[Lemma A.4]{HeutsGood}, we can write the compact objects in $\Sp^G_{\UU}$ as the idempotent completion of $\colim_{\Sub_{\UU}}(\Spc_*^G)^{\omega}$, where $(\Spc_*^G)^{\omega}$ are the compact objects and the colimit is taken in $\infty$-categories. Thus, $\Sp^G_{\UU}$ agrees with the ind-completion of this colimit \cite[Proposition 5.5.7.8]{HTT}. If $\UU$ is countably-dimensional, we can find a cofinal map $\N \to \Sub_{\UU}$ and furthermore the map from the underlying graph of $\N$ to the nerve of $\N$ is cofinal (e.g.\ using \cite[Corollary 4.1.1.9]{HTT}), making this colimit particularly easy to understand.

Another perspective on this colimit is as the limit over $\Sub_{\UU}^{\op}$ of the diagram $T^R$ with $T^R(V\subset W) = (\Spc_*^G \xrightarrow{\Omega^{V^{\perp}}} \Spc_*^G)$ (see \cite[Proposition 5.5.7.6, Remark 5.5.7.7]{HTT}). This way we also see that the colimit defining $\Sp^G_{\UU}$ agrees with the corresponding colimit in the $\infty$-category $\Pr^L$ of presentable $\infty$-categories. 

\subsection{Comparison to orthogonal spectra}
We want to compare our definition of $G$-spectra with the (maybe more traditional) definition via orthogonal spectra. Thus let $\Sp^G_O$ be the relative category of $G$-objects in orthogonal spectra, whose weak equivalences are the stable equivalences with respect to a complete universe $\UU$ (which we will fix from now on). We refer to \cite{HHR}, \cite{SchwedeEquiv}, \cite{MandellMay} and \cite{SchGlobal} for general background on equivariant orthogonal spectra. We will use the stable model structure from \cite[Section III.4]{MandellMay}. 

\begin{lemma}
	The $\infty$-category $L\Sp^G_O$ is compactly generated. 
\end{lemma}
\begin{proof}
	The $\infty$-category $L\Sp^G_O$ has all colimits and these can be computed as homotopy colimits in $\Sp^G_O$ \cite[Section 4.2.4]{HTT}. By \cite[Proposition 5.4.2.2]{HTT} it thus suffices to obtain a set of compact objects generating $\Sp^G_O$ under filtered homotopy colimits. 
	
	To that purpose we first choose a set $\CC'$ of pointed $G$-spaces such that every pointed finite $G$-CW complex is isomorphic to an object in $\CC'$. We define $\CC$ as the set of all $\Sigma^{-V} \Sigma^\infty X$ for $X\in \CC'$ and $V$ an orthogonal subrepresentation of the universe $\UU$. Here, $\Sigma^{-V} \Sigma^\infty X = F_V(X)$ with $F_V$ being the left adjoint to the evaluation $\ev_V$ from $\Sp^G_O$ to pointed $G$-spaces. As in \cite[B.4.3]{HHR}, we can pick a cofinal map $\N \to \Sub_{\UU}$ with images $V_n$ and an arbitrary orthogonal $G$-spectrum $E$ to form zig-zags
	\[\Sigma^{-V_n}\Sigma^{\i}E(V_n) 
	\xleftarrow{\simeq} \Sigma^{-V_{n+1}}\Sigma^{\i}\Sigma^{V_{n+1}-V_n}E(V_{n+1}) 
	\to \Sigma^{-V_{n+1}}\Sigma^{\i}E(V_{n+1}), \]
	all mapping to $E$. Pasting the zig-zags and taking the homotopy colimit produces an equivalence $\hocolim \Sigma^{-V_n}\Sigma^{\i}E(V_n) \to E$ (as one sees by taking homotopy groups). Each $E(V_n)$ in turn can be written as a filtered colimit in $\Spc_*^G$ of finite $G$-CW complexes and thus each $\Sigma^{-V_n}\Sigma^{\i}E(V_n)$ as a filtered colimit of objects in $\CC$.  
\end{proof}

 We remark that $\Sigma^{\infty}\colon \Top_*^G \to \Sp_O^G$ is homotopical by \cite[Proposition 3.1.44]{SchGlobal}. Thus $\Sigma^{\infty}$ descends to a functor $\Spc_*^G \to L\Sp_O^G$.

\begin{lemma}
    Given $X,Y\in \Spc_*^{G,\omega}$, the zig-zag
    \[\xymatrix{
    &\colim_{V\in \Sub_{\UU}} \Map_{\Spc_*^G}(\Sigma^V X, \Sigma^V Y) \ar[d]\\
    \Map_{L\Sp^G_O}(\Sigma^{\infty}X, \Sigma^{\infty} Y) \ar[r] & \colim_{V\in \Sub_{\UU}}\Map_{L\Sp^G_O}(\Sigma^{\infty}\Sigma^VX, \Sigma^{\infty} \Sigma^VY)  
    }\]
    consists of equivalences. 
\end{lemma}
\begin{proof}
The horizontal map is an equivalence as $\Sigma^{\i}$ commutes with $\Sigma^V$ and the latter defines an equivalence on $L\Sp_O^G$. 

    For the other equivalence recall that the fibrant objects in the stable model structure are the $G$-$\Omega$-spectra, i.e.\ the orthogonal $G$-spectra $Z$ such that the adjoint structure maps $Z(V) \to \Omega^WZ(V\oplus W)$ are $G$-equivalences for all representations $V,W$. One checks that the orthogonal $G$-spectrum $(\Sigma^V Y)'$ given by $(\Sigma^V Y)'_n = \hocolim_{V'\in \Sub_{\UU}} \Omega^{V'}\Sigma^{V'+n}\Sigma^VY$ defines a fibrant replacement for $\Sigma^{\infty}\Sigma^VY$. 
    Moreover, we can assume that $X$ is a $G$-CW complex so that $\Sigma^{\infty}\Sigma^VX$ is cofibrant. 
        Thus, we can compute $\Map_{L\Sp^G_O}(\Sigma^{\infty}\Sigma^VX, \Sigma^{\infty} \Sigma^VY)$ as the mapping space in the topological category of orthogonal $G$-spectra between $\Sigma^{\infty} \Sigma^VX$ and $(\Sigma^V Y)'$. As $\Sigma^{\infty}$ is adjoint to taking the zeroth space, this agrees with 
        \[\Map_{\Spc_*^G}(\Sigma^VX, \hocolim_{V'\in \Sub_{\UU}} \Omega^{V'}\Sigma^{V'}\Sigma^VY) \simeq \colim_{V'\in \Sub_{\UU}} \Map_{\Spc_*^G}(\Sigma^{V'\oplus V}X, \Sigma^{V'\oplus V} Y). \]
        The map 
        \[\colim_{V\in \Sub_{\UU}} \Map_{\Spc_*^G}(\Sigma^V X, \Sigma^V Y) \to \colim_{V,V'\in \Sub_{\UU}} \Map_{\Spc_*^G}(\Sigma^{V'\oplus V} X, \Sigma^{V'\oplus V} Y)\]
        is clearly an equivalence.
\end{proof}

\begin{prop}\label{prop:orthogonalequivalence}
    There is an equivalence $\Sp^G \simeq L\Sp_O^G$.
\end{prop}
\begin{proof}
    Analogously to the functor $T\colon \Sub_{\UU} \to \RelCat$ considered before, we can consider a (pseudo-)functor $T' \colon \Sub_{\UU} \to \RelCat$ that sends each $V$ to $\Sp_O^G$ and each inclusion $V \subset W$ to $\Sigma^{W-V}$, with $W-V$ the orthogonal complement of $V$ in $W$. (The functors $\Sigma^{W-V}\colon \Sp_O^G \to \Sp_O^G$ are indeed homotopical by \cite[Proposition 3.2.19]{SchGlobal}.) As $\Sigma^{W-V}\colon \Sp_O^G \to \Sp_O^G$ defines an equivalence of associated $\infty$-categories, we can identify the colimit of $LT'\colon \Sub_{\UU} \to \Cat_{\infty}^{\omega}$ with $L\Sp_O^G$. 
    
    As $\Sigma^{W-V}\Sigma^{\infty}$ is canonically isomorphic to $\Sigma^{\infty}\Sigma^{W-V}$, we see that $\Sigma^{\infty}$ defines a natural transformation $T \Rightarrow T'$. Taking colimits of the associated natural transformation $LT \Rightarrow LT'$ (with values in $\Cat^{\omega}_{\infty}$), we obtain a functor $F\colon \Sp^G \to L\Sp_O^G$ in $\Cat_{\infty}^{\omega}$. 
    
    By definition, precomposing $F$ with the map $\iota_V\colon \Spc_*^G \to \Sp^G$ associated with a subrepresentation $V\subset \UU$, sends some $X\in \Spc_*^G$ up to equivalence to $\Sigma^{-V}\Sigma^{\infty}X$ (as all choices of $V$-fold desuspensions are equivalent). As such orthogonal spectra generate $L\Sp_O^G$ via colimits, $\Sp^G$ has all colimits and $F$ preserves them, we see that $F$ is essentially surjective. 
    
    As $\Sp^G$ is compactly generated, every object in $\Sp^G$ is a filtered colimit of objects of the form $\iota_V(X)$ with $X \in \Spc_*^G$ compact. As a mapping space in a filtered colimit of $\infty$-categories is just the filtered colimit of mapping spaces,\footnote{One easy way of seeing this is by identifying the mapping space between $X$ and $Y$ in some $\infty$-category $\CC$ with the fiber product $\pt \times_{\CC^{\partial \Delta^1}}\CC^{\Delta^1}$ in $\Cat_{\infty}$, where the map $\pt \to \CC^{\partial \Delta^1}$ classifies $(X,Y)$. As both $\partial \Delta^1$ and $\Delta^1$ are compact in $\Cat_{\infty}$, the result follows.} we can identify the map 
    \[\Map_{\Sp^G}(\iota_V(X), \iota_V(Y)) \to \Map_{L\Sp_O^G}(F\iota_V(X), F\iota_V(Y))\] 
    with the natural map from 
    $\colim_{V\subset W\in \Sub_{\UU}} \Map_{\Spc_*^G}(\Sigma^{W-V}X, \Sigma^{W-V}Y)$ to 
    \[\colim_{V \subset W\in \Sub_{\UU}} \Map_{L\Sp_O^G}(\Sigma^{\i}\Sigma^{W-V}X, \Sigma^{\i}\Sigma^{W-V}Y) \simeq \Map_{L\Sp_O^G}(\Sigma^{-V}\Sigma^{\i}X, \Sigma^{-V}\Sigma^{\i}X),\]
   which is an equivalence by the previous  lemma.
    Thus, $F$ is fully faithful. 
\end{proof}

\subsection{Symmetric monoidal structures and universal properties}
Next, we want to deduce a symmetric monoidal universal property for $\Sp^G$. Recall to that purpose that an object $X$ of a symmetric monoidal $\infty$-category is called \emph{symmetric} if the cyclic permutation map acting on $X \tensor X \tensor X$ is homotopic to the identity. 

\begin{lemma}
	Let $G$ be a compact Lie group and $V$ an orthogonal $G$-representation. Then $S^V$ is symmetric in $G$-spaces. 
\end{lemma}
\begin{proof}
	We can write $V$ as $\bigoplus_i W_i^{\oplus n_i}$, where the $W_i$ are irreducible. By Schur's lemma the group of $G$-equivariant automorphisms $\Aut_G(V)$ of $V$ is isomorphic to $\prod_i GL_{n_i}k_i$, where $k_i = \R, \C$ or $\mathbb{H}$. Thus, $\pi_0\Aut_G V$ is a finite product of $\Z/2$. As the cyclic permutation $\sigma\colon V^{\oplus 3} \to V^{\oplus 3}$ is of order $3$, we see that $\sigma$ is in the path-component of the identity in $\Aut_G(V)$. Taking one-point compactifications, we see that the cyclic permutation of $S^V \sm S^V \sm S^V$ is homotopic to the identity. 
\end{proof}

Recall that a symmetric monoidal $\infty$-category is \emph{presentably symmetric monoidal} if it is presentable and the tensor products commutes in both variables with colimits. (Equivalently, it is a commuative algebra in the $\infty$-category $\Pr^L$ of presentable $\infty$-categories.) Robalo shows in \cite[Proposition 2.9 and Corollary 2.20]{RobaloBridge} the following:
\begin{theorem}\label{thm:Robalo}
    Let $\CC$ be a presentably symmetric monoidal category and $X \in \CC$ a symmetric object. Let $\Stab_X\CC$ be the colimit of 
    \[ \CC \xrightarrow{\tensor X} \CC \xrightarrow{\tensor X} \cdots\]
    in $\Pr^L$. Then $\CC \to \Stab_X\CC$ refines to a symmetric monoidal functor and for any other presentably symmetric monoidal $\infty$-category $\DD$ and a symmetric monoidal left adjoint $F\colon \CC \to \DD$, sending $X$ to an invertible object, there is an essentially unique factorization over a symmetric monoidal left adjoint $\Stab_X \CC \to \DD$. Moreover, the resulting square
    \[
    \xymatrix{
    \CC \ar[r]\ar[d]^{F} & \Stab_X \CC\ar@{-->}[dl]\ar[d] \\
    \DD \ar[r] & \Stab_{F(X)} \DD
    }
    \]
    consists of two commutative triangles.
\end{theorem}
\begin{proof}
    Everything except for the last statement is directly contained in the cited statements of \cite{RobaloBridge}. For the last statement, we have to open up slightly the box of proofs. In \cite[Proposition 2.1]{RobaloBridge} Robalo shows the existence of a left adjoint $\LL^{\tensor}_{\CC^{\tensor}, X}$ to the inclusion of presentably stable symmetric monoidal $\infty$-categories under $\CC$ where $X$ acts invertibly to all presentably stable symmetric monoidal $\infty$-categories under $\CC$. The symmetric monoidal functor $F$ induces a commutative diagram
    \[
    \xymatrix{\CC\ar[d]^F \ar[r] & \LL^{\tensor}_{\CC^{\tensor}, X}(\CC) \ar[r]^-{\simeq}\ar[d] \ar@{-->}[dl] & \LL^{\tensor}_{\CC^{\tensor}, X}(\Stab_X\CC) \ar[r]^-{\simeq}\ar[d] & \Stab_X \CC \ar[d]\\
    \DD \ar[r]^-{\simeq} &  \LL^{\tensor}_{\CC^{\tensor}, X}(\DD) \ar[r]^-{\simeq} &  \LL^{\tensor}_{\CC^{\tensor}, X}(\Stab_{F(X)}\DD) \ar[r]^-{\simeq} & \Stab_{F(X)}\DD
    }
    \]
    The functor $\Stab_X\CC \to \DD$ is the composite of the inverses of the two upper horizontal arrows and the diagonal arrow. The statement follows from the commutativity of the two triangles in the diagram, which follows in turn from the triangle identity of adjoints. 
\end{proof}
This theorem does not directly apply to our situation, as we do not invert just a single object, but $S^V$ for all (irreducible) finite-dimensional representations $V$. We nevertheless obtain the following: 
\begin{cor}\label{prop:universalpropPres}
The functor $\Sigma^{\infty}\colon\Spc_*^G \to \Sp^G$ refines to a symmetric monoidal functor. Moreover, for any symmetric monoidal left adjoint $F\colon \Spc_*^G \to \DD$ into a presentably symmetric monoidal $\infty$-category $\DD$ such that $F(S^V)$ is invertible for every irreducible $G$-representation $V$, there is an essentially unique symmetric monoidal left adjoint $\overline{F}\colon \Sp^G \to \DD$ with an equivalence $\overline{F}\Sigma^{\infty}\simeq F$. 
 \end{cor}
 \begin{proof}
     Choose irreducible orthogonal $G$-representations $V_1, V_2, \dots$ such that every irreducible $G$-representation is isomorphic to exactly one of these. Then $\Sp^G$ can be identified with the directed colimit
     \[\Spc_*^G \to \Stab_{S^{V_1}}\Spc_*^G \to \Stab_{S^{V_2}}\Stab_{S^{V_1}}\Spc_*^G \to \cdots  \]
     in $\Pr^L$. 
     
     The forgetful functor from presentably symmetric monoidal $\infty$-categories under $\Spc_*^G$ to $\Pr^L$ preserves filtered colimits. Inductively, we see by the previous theorem that all the mapping spaces from $\Stab_{S^{V_n}}\cdots \Stab_{S^{V_1}}\Spc_*^G$ to $\DD$, computed in presentably symmetric monoidal $\infty$-categories under $\Spc_*^G$, are contractible, and thus the same is true if we go the colimit. 
    \end{proof}
    
 The result we actually use will rather be a version for small categories. For this we denote by $\Spc_*^{G, \fin}$ the subcategory of $\Spc_*^G$ of \emph{finite $G$-spaces}, i.e.\ the subcategory generated by finite colimits from the orbits $G/H_+$ for all closed subgroups $H\subset G$. Moreover, $\Sp^{G, \omega}$ denotes the $\infty$-category of (retracts of) finite $G$-spectra, i.e.\ the compact objects in $\Sp^G$. It is easy to see that the functor $\Sigma^{\infty}\colon \Spc_*^G \to \Sp^G$ restricts to a functor $\Sigma^{\infty}\colon \Spc_*^{G,\fin} \to \Sp^{G, \omega}$. Analogously, we can define an $\infty$-category $\Spc^{G, \fin}$ of unpointed finite $G$-CW-complexes and obtain a functor $\Spc^{G,\fin} \to \Spc_*^{G, \fin}$ by adjoining a disjoint base point. The composition of these functors will be denoted by $\Sigma^{\infty}_+$.
  \begin{cor}\label{prop:universalprop}
 Let $\DD$ be a presentably symmetric monoidal $\infty$-category and $F\colon \Spc_*^{G,\fin} \to \DD^{\op}$ be a symmetric monoidal functor that preserves finite colimits, sends every object to a dualizable object and every representation sphere to an invertible object. Then $F$ factors over $\Sigma^{\infty}\colon \Spc_*^{G, \fin} \to \Sp^{G,\omega}$ to produce a finite colimit preserving functor $\Sp^{G,\omega} \to \DD^{\op}$.
  \end{cor}
  \begin{proof}
  Denote by $\DD^{\dual}$ the dualizable objects in $\DD$. Dualizing yields a symmetric monoidal functor 
  \[F'\colon \Spc_*^{G,\fin} \to \DD^{\dual, \op} \to \DD^{\dual}\subset \DD.\]
      The functor $F'$ factors over $\Spc_*^{G, \fin} \to \Ind(\Spc_*^{G, \fin}) \simeq \Spc_*^G$, yielding a colimit-preserving symmetric monoidal functor $\Spc_*^G \to \DD$. Applying the previous corollary, this factors over a functor $\Sp^G \to \DD$, which we can restrict again to $\Sp^{G, \omega}$. The compact objects agree with the dualizable objects and thus dualizing once more yields the result. 
  \end{proof}

We want to compare the symmetric monoidal structure defined on $\Sp^G$ above with the symmetric monoidal structure coming from the category $\Sp_O^G$ of orthogonal spectra. For that purpose we want to recall how to pass from a symmetric monoidal model category to a symmetric monoidal $\infty$-category. This was already discussed in \cite[Section 4.1.7]{HA} and \cite[Appendix A]{NikolausScholze}, but we sketch a more elementary treatment suggested to us by Daniel Sch{\"a}ppi. The starting point is a symmetric monoidal relative category $(\CC, \WW, \tensor)$, i.e.\ a symmetric monoidal category with a subcategory of weak of equivalences $\WW$ such that $c \tensor d \to c' \tensor d$ is a weak equivalence if $c \to c'$ is. We will assume that $\WW$ contains all objects and satisfies $2$-out-of-$6$; this implies in particular that $\WW$ contains all isomorphisms. An important class of examples is the category of cofibrant objects in a symmetric monoidal model category. 

As sketched in \cite[Section 2]{SegalCategories}, every symmetric monoidal category defines a functor $\Gamma^{\op} \to \Cat$, sending $0$ to the terminal category and satisfying the Segal condition. In case of a symmetric monoidal relative category, this lifts to a functor $\Gamma^{\op} \to \RelCat$ into the category of relative categories and weak-equivalence preserving functors. This satisfies the Segal conditions both on underlying categories and categories of weak equivalences. The localization functor $L\colon \RelCat \to \Cat_{\i}$ preserves products \cite[Proposition 7.1.13]{Cisinski}. Thus, composing our functor $\Gamma^{\op} \to \RelCat$ with $L$ we obtain a functor from $\Gamma^{\op}$ into $\Cat_{\i}$ satisfying the Segal conditions and thus defining a symmetric monoidal $\infty$-category (see \cite[Remark 2.4.2.2, Proposition 2.4.2.4]{HA}). 

Going back to equivariant homotopy theory, we consider again the stable model structure of \cite[Section III.4]{MandellMay} on $\Sp_O^G$. The stable model structure satisfies the pushout-product axiom with respect to the smash product \cite[Proposition III.7.5]{MandellMay} and its unit $\mathbb{S}$ is cofibrant -- thus its category $\Sp_O^{G,\cof}$ of cofibrant objects defines a symmetric monoidal relative category and $L\Sp_O^{G,\cof}$ obtains the structure of a symmetric monoidal $\infty$-category.

\begin{prop}
    The equivalence in \cref{prop:orthogonalequivalence} refines to a symmetric monoidal equivalence between $\Sp^G$ with the symmetric monoidal structure from \cref{prop:universalpropPres} and $L\Sp_O^{G,\cof}\simeq L\Sp_O^G$ with the symmetric monoidal structure induced by the smash product. 
\end{prop}
\begin{proof}
    By definition of the generating cofibrations in \cite[Section III.2]{MandellMay} one observes that the suspension spectrum functor $\Sigma^{\infty}\colon\Top^G_* \to \Sp_O^G$ preserves cofibrant objects, where we consider on $\Top^G_*$ the model structure where fibrations and weak equivalences are defined to be those maps that are fibrations and weak equivalences, respectively, on fixed points for all subgroups. As $\Sigma^{\infty}$ is strong symmetric monoidal, we see that the resulting functor $\Spc_*^G \to L\Sp_O^G$ is strong symmetric monoidal again (since it defines a natural transformation of functors from $\Gamma^{\op}$).

As the smash product on $\Sp_O^G$ commutes in each variable with colimits and colimits in $L\Sp_O^G$ can be computed as homotopy colimits in $\Sp_O^G$, the induces symmetric monoidal structure on $L\Sp_O^G$ commutes with colimits in both variables. As moreover $\Sigma^{\infty}$ sends all representation spheres to invertible objects, the universal property \cref{prop:universalpropPres} yields a symmetric monoidal functor $\Sp^G \to L\Sp_O^{G,\cof}$. The last part of \cref{thm:Robalo} implies that the underlying functor of $\infty$-categories agrees up to equivalence with the equivalence we have constructed in \cref{prop:orthogonalequivalence} (once restricted to cofibrant objects).
\end{proof}

\subsection{The Wirthm{\"u}ller isomorphism}
As a last point we want to state the Wirthm{\"u}ller isomorphism. Note to that purpose that for every compact Lie group $G$ and every closed subgroup $H\subset G$ the restriction functor $\Sp^G \to \Sp^H$ has two adjoints, which we denote by $G_+ \tensor_H -$ and $\uMap_H(G_+, -)$. We need these adjoints only on the level of homotopy categories, where they are well-known, e.g.\ by comparing to orthogonal spectra. 
\begin{thm}[Wirthm{\"u}ller]
Let $L = T_{eH}G/H$ be the tangent $G$-representation. Then there is for every $X \in \Sp^H$ an equivalence $G_+ \tensor_H X \to \uMap_H(G_+, S^L \tensor X)$.
\end{thm}
References include \cite{MayWirth} and \cite[Section 3.2]{SchGlobal}. (While the latter constructs in (3.2.6) only the Wirthm{\"u}ller map on the level of homotopy groups, it is clear the same construction defines it as a transformation of homology theories. As the $G$-stable homotopy category is a Brown category \cite[Example 1.2.3b, Section 4.1]{HoveyPalmieriStrickland}, this can be lifted to a map in $\Sp^G$.) The special case we need is the following:
\begin{cor}\label{cor:dual}
Let $L$ be the tangent representation of $G$. Then there is an equivalence between $\Sigma^{\infty}G_+ \tensor S^{-L}$ and the Spanier--Whitehead dual $D\Sigma^{\infty}G_+$.
\end{cor}
This special case can also be seen as special cases of equivariant Atiyah--Duality \cite[Theorem III.5.1]{LMS}.

\section{Spectral algebraic geometry of the big \'etale site}\label{app:SAG}
The purpose of this appendix is to establish a few facts that are necessary to work with the big \'etale site. Before we recall its definition, we need the following definition. 

\begin{defi}\label{def:nqsep}
Call a morphism $f\colon \Yf \to \Zf$ of spectral Deligne--Mumford stacks \emph{$0$-quasi-separated} if it is quasi-compact and \emph{$n$-quasi-separated} if the diagonal $\Delta_f\colon \Yf \to \Yf\times_{\Zf}\Yf$ is $(n-1)$-quasi-separated. If $\Yf \to \Spec \mathbb{S}$ is $n$-quasi-separated, we say that $\Yf$ is $n$-quasi-separated.   
\end{defi}
Note that being $1$-quasi-separated agrees per definition with being quasi-separated. As will follow from the next lemma, the notion of being $n$-quasi-separated is moreover closely related to being $n$-quasi-compact in the sense of \cite{SAG}: a spectral Deligne--Mumford stack $\Yf$ is $k$-quasi-separated for all $0\leq k\leq n$ if and only if it is $n$-quasi-compact (cf.\ \cite[Proposition 2.3.6.2]{SAG}).

\begin{lemma}Let $f\colon \Yf \to \Zf$ be a morphism of spectral Deligne--Mumford stacks and let $n\geq 0$. 
    \begin{enumerate}
        \item Being $n$-quasi-separated is local in the \'etale topology on the target and closed under base change. Moreover, it is closed under composition and under products.
        \item Every spectral affine scheme is $n$-quasi-separated for all $n$. 
        \item If $\Zf$ is affine, $f$ is $n$-quasi-separated if and only if $\Yf$ is $n$-quasi-separated. 
        \item The morphism $f$ is $n$-quasi-separated if and only if $\Spec A\times_{\Zf}\Yf$ is $n$-quasi-separated for all $\Spec A \to \Yf$. In fact, it suffices to check against an \'etale cover. 
        \item The stack $\Yf$ is $(n+1)$-quasi-separated if and only if for all $n$-quasi-separated $U$ and $V$, the fiber product $U\times_{\Yf} V$ is $n$-quasi-separated. 
    \end{enumerate}
\end{lemma}
\begin{proof}
    That being $n$-quasi-separated is local in the \'etale topology, closed under base change and products follows inductively from the corresponding statement for quasi-compact morphisms \cite[Proposition 2.3.3.1, Corollary 2.3.5.3, Example 6.3.3.6]{SAG}. For composition, note that for a morphism $g\colon \Xf\to \Yf$, the morphism $\Xf\times_{\Yf}\Xf \to \Xf \times_{\Zf}\Xf$ is a base change of $\Delta_f\colon \Yf \to \Yf \times_{\Zf} \Yf$. Thus, if  $\Delta_f$ and $\Delta_g$ are $(n-1)$-quasi-separated, so is the composite $\Xf \xrightarrow{\Delta_g} \Xf\times_{\Yf}\Xf \to \Xf \times_{\Zf}\Xf$, using induction (see \cite[Corollary 2.3.5.2]{SAG} for the case $n=0$). This yields (1).

    By induction, one shows that morphisms between affine schemes are $n$-quasi-separated for all $n$, giving (2). Assume that Point (3) holds for a given $n$ and that $\Zf$ is affine. Then $\Yf \to \Zf$ is $(n+1)$-quasi-separated if and only if $\Spec A \times_{\Yf}\Spec A \to \Spec A \times_{\Zf}\Spec A$  is $n$-quasi-separated for all $\Spec A \to \Yf$ that are part of an \'etale cover. By induction, this is true if and only if $\Spec A \times_{\Yf}\Spec A$ is $n$-quasi-separated. This happens if and only if $\Yf$ itself is $(n+1)$-quasi-separated. This yields (3). 

    Point (4) follows from (1) and (3). The pullback square
    \[
\xymatrix{
U\times_{\Yf} V \ar[r] \ar[d]& U\times V \ar[d]\\
\Yf \ar[r] & \Yf\times \Yf}
\]
shows that (5) follows from (1) and (4).
\end{proof} 

Recall from the conventions that we assumed all our spectral Deligne--Mumford stacks to be locally noetherian and $n$-quasi separated for all $n\geq 1$. For clarity, we will be explicit about these conventions in this appendix and not assume them implicitly. It will be established in the proof of \cref{lem:smallet} that every quasi-separated spectral scheme and every quasi-separated spectral Deligne--Mumford $1$-stack with a quasi-separated diagonal is $n$-quasi-separated for all $n\geq 1$. Thus, this condition is rather mild and forms a natural higher analogue of the usual quasi-separatedness condition. In particular, it is fulfilled by the (connective cover) of the moduli stack of oriented spectral elliptic curves, which underlies $TMF$.  

Let $\Mf$ be a (locally noetherian) spectral Deligne--Mumford stack that is $n$-quasi-separated for all $n$. Recall from \cref{def:Shv} that its big \'etale site is defined to consist of all spectral Deligne--Mumford stacks that are almost of finite presentation over $\Mf$ and $n$-quasi-separated for all $n\geq 1$. (Local noetherianity of objects in the big \'etale site is automatic if $\Mf$ is locally noetherian by \cite[Remark 4.2.0.4]{SAG}.) This $\infty$-category has all pullbacks. Indeed: As observed above, being $n$-quasi separated for all $n\geq 1$ is preserved by pullbacks. For being almost of finite presentation, use \cite[Corollary 7.4.3.19]{HA}. With the \'etale topology, the big \'etale site thus becomes indeed a site and we define $\Shv(\Mf)$ as the $\infty$-category of space-valued sheaves on the big \'etale site. To establish that $\Shv(\Mf)$ is an $\infty$-topos, we need to establish that the big \'etale site is a small $\infty$-category (which is the whole reason for imposing a finiteness condition). For this in turn, we first need the following result of general interest. 

\begin{lemma}\label{lem:local}
     Let $\PP$ be a property of morphisms of spectral Deligne--Mumford stacks that is \'etale-local on the target, and denote for any spectral Deligne--Mumford stack $\Mf$ by $\textrm{SpDM}^{\PP}_{/\Mf}$ the $\infty$-category of morphisms $\Nf \to \Mf$ in $\PP$. Then $\Mf \mapsto \textrm{SpDM}^{\PP}_{/\Mf}$ satisfies \'etale hyperdescent. This applies in particular to the functor sending $\Mf$ to its big \'etale site. 
 \end{lemma}
 \begin{proof}
     Let $\widehat{\Spc}$ 
     be the very large $\infty$-category of large spaces and denote by 
     $\XX = \widehat{\Shv}_{\acute{e}t}(\mathrm{Aff}, \widehat{\Spc})$
     the very large $\infty$-topos of large \'etale hypersheaves on spectral affine schemes. Let $U_{\bullet} \to \Mf$ be a hypercover of a spectral Deligne--Mumford stack. Since $\XX$ is hypercomplete, this defines a colimiting diagram in $\XX$. By descent in $\infty$-topoi, this implies that $\XX_{/\Mf} \to \lim_{\bullet\in \Delta}\XX_{/U_{\bullet}}$ is an equivalence. Since the functor of points functor from the $\infty$-category of spectral Deligne--Mumford stacks $\mathrm{SpDM}$ to $\XX$ is fully faithful \cite[Proposition 1.6.4.2]{SAG}, we see that $\mathrm{SpDM}^{\PP}_{/\Mf} \to \lim_{\bullet\in \Delta} \mathrm{SpDM}^{\PP}_{/U_{\bullet}}$ is fully faithful. For essential surjectivity, we need to show first that for a morphism $\Nf \to \Mf$ in $\XX$, the functor $\Nf$ is spectral Deligne--Mumford if $U_0 \times_{\Mf}\Nf$ is. This follows since $U_0 \times_{\Mf}\Nf \to \Nf$ is an \'etale cover, which we can precompose with an \'etale cover from a disjoint union of affines. Essential surjectivity follows now since $\PP$ is \'etale-local on the target. 

     Since being almost of finite presentation is a property that is \'etale on the target by \cite[Example 6.3.3.6(11,13)]{SAG}, the last point follows as well.
 \end{proof}

\begin{lemma}\label{lem:smallbig}
    Given a locally noetherian non-connective spectral Deligne--Mumford stack $\Mf$ that is $n$-quasi-separated for all $n\geq 1$, its big \'etale site is an essentially small $\infty$-category.
\end{lemma}

\begin{proof}
    Recall from \cite[Proposition 7.2.4.27(4)]{HA} that for a connective $E_{\infty}$-ring $A$, a connective $E_{\infty}$-$A$-algebra $B$ is almost of finite presentation if each of its truncations $\tau_{\leq n}B$ is a compact object of $\tau_{\leq n}\CAlg_A$. As the $\infty$-category of compact objects in $\CAlg_A$ is essentially small and $B \simeq \lim_n \tau_{\leq n}B$, we see that the full subcategory of $\CAlg_A$ consisting of those connective algebras almost of finite presentation is essentially small as well. 

        For the general case, we may assume that $\Mf$ is a spectral Deligne--Mumford stack. By \cref{lem:local}, formation of the big \'etale site satisfies hyperdescent. Choosing a hypercover by disjoint unions of affines reduces to the case $\Mf = \Spec A$. By definition we see that for every \'etale map $\Spec B \to \Zf$ for $\Zf$ almost of finite presentation over $\Spec A$, the $A$-algebra $B$ must be almost of finite presentation over $A$. Thus, there is up to equivalence only a set of possible hypercovers of some $\Zf$ almost of finite presentation over $\Spec A$ such that each stage is a finite union of affines. As we can recover $\Zf$ as the geometric realization of the hypercover (e.g.\ in \'etale hypersheaves, as in the proof of \cref{lem:local}) and $\Zf$ is quasi-compact by \cite[Propositions 2.3.1.2 and 2.3.5.1]{SAG}, this proves the lemma.     

\end{proof}
Let $\Mf = (\MM, \OO_{\Mf})$ be a spectral Deligne--Mumford stack. Our final goal for this appendix is to compare $\MM$ and $\Shv(\Mf)$. 
By the definition of a spectral Deligne--Mumford stack, for every $U\in \MM$, the pair consisting of the slice topos $\MM_{/U}$ and the restriction of $\OO_{\Mf}$ to it is a spectral Deligne--Mumford stack. By the definition of \'etale morphisms (\cite[Section 1.4.10]{SAG}), these are (up to equivalence) precisely the spectral Deligne--Mumford stacks that are \'etale over $\Mf$. If $\Mf$ is $n$-quasi-separated for all $n\geq 1$, then $\Yf\in \MM$ lies in the big \'etale site of $\Mf$  if $\Yf$ is $n$-quasi-separated for all $n\geq 0$. Indeed, by \cite[Proposition 2.3.1.2]{SAG} \'etale morphisms are automatically locally almost of finite presentation. If $\Mf$ is an $n$-stack, denote by $\Mf_{\et}$ the full subcategory of $\MM$ on these objects which are additionally $n$-truncated. Recall that $\Mf$ is an \emph{$n$-stack} if $\Mf(\Spec R)$ is $n$-truncated for all classical rings $R$. The restriction of the canonical topology on $\MM$ defines a topology on $\Mf_{\et}$.

\begin{lemma}\label{lem:smallet}
    Let $\Mf$ be a spectral Deligne--Mumford $n$-stack. Then restriction defines an equivalence $\MM \simeq \Shv(\MM) \simeq \Shv(\Mf_{\et})$. 
\end{lemma}
\begin{proof}
We claim that every $n$-stack $\Yf$ is $(n+2)$-quasi-separated. Indeed: Since $U\times_{\Yf} V$ is an $(n-1)$-stack for affines $U$ and $V$ mapping to $\Yf$, we  can assume that $\Yf$ is a $0$-stack, i.e.\ a spectral algebraic space. Thus, $\Delta\colon \Yf \to \Yf \times \Yf$ induces on each classical affine $\Spec R$ an injection of sets; hence $\Delta_{\Delta}\colon \Yf \to \Yf \times_{\Yf \times \Yf} \Yf$ induces an isomorphism on every classical $\Spec R$ and is thus an isomorphism on $0$-truncations. In particular, it is quasi-compact and hence $\Yf$ is $2$-quasi-separated. 

Denote by $\MM_{\leq n} \subset \MM$ the full subcategory on $n$-truncated objects and by $\MM_{\leq n}^k\subset \MM_{\leq n}$ the further full subcategory of $m$-quasi-separated objects for every $m\geq k$. By \cite[Lemma 1.4.7.7]{SAG}, every $n$-truncated object in $\MM$ defines an $n$-stack and thus $\MM_{\leq n} = \MM_{\leq n}^{n+2}$. By the comparison lemma for Grothendieck topologies \cite[Lemma C.3]{HoyoisQuadratic}, $\MM_{\leq n}^k \subset \MM_{\leq n}^{k+1}$ defines an equivalence on $\infty$-categories of sheaves, for every $k$: indeed, the crucial property is that every $X\in \MM_{\leq n}^{k+1}$ has a cover by  $U_i\in \MM_{\leq n}^k$ (one can take $U_i$ to be affine by \cite[Proposition 1.4.7.9]{SAG}) and that $U_i\times_X U_j \in \MM_{\leq n}^k$ (as follows from $X$ being $(k+1)$-quasi-separated). Thus, we obtain a chain of equivalences \[\Shv(\MM_{\leq n}^0) \simeq \cdots \simeq \Shv(\MM_{\leq n}^{n+2}) = \Shv(\MM_{\leq n}) \simeq \MM.\]
The last equivalence follows from \cite[Proposition 1.6.8.5]{SAG}.
\end{proof}

Thus, restriction along the inclusion of $\Mf_{\et}$ into the big \'etale site defines for any spectral Deligne--Mumford $n$-stack $\Mf = (\MM, \OO_{\Mf})$ (that is locally noetherian and $k$-quasi-separated for all $k\geq 1$) a geometric morphism $\Shv(\Mf) \to \MM$.

\bibliographystyle{alpha}
\bibliography{Chromatic}
\end{document}